\documentclass[11pt]{amsart}

\usepackage[margin=1.1in]{geometry}
\usepackage{microtype}

\usepackage{graphicx}
\usepackage{amssymb}
\usepackage{enumerate}
\usepackage{times}
\usepackage{cite}

\usepackage[colorlinks=true, pdfstartview=FitV, linkcolor=blue, citecolor=blue, urlcolor=blue]{hyperref}

\def\Xint#1{\mathchoice
{\XXint\displaystyle\textstyle{#1}}%
{\XXint\textstyle\scriptstyle{#1}}%
{\XXint\scriptstyle\scriptscriptstyle{#1}}%
{\XXint\scriptscriptstyle\scriptscriptstyle{#1}}%
\!\int}
\def\XXint#1#2#3{{\setbox0=\hbox{$#1{#2#3}{\int}$}
\vcenter{\hbox{$#2#3$}}\kern-.5\wd0}}

\def\avgint{\Xint-}

\makeatletter
\newcommand*\rel@kern[1]{\kern#1\dimexpr\macc@kerna}
\newcommand*\widebar[1]{%
  \begingroup
  \def\mathaccent##1##2{%
    \rel@kern{0.8}%
    \overline{\rel@kern{-0.8}\macc@nucleus\rel@kern{0.2}}%
    \rel@kern{-0.2}%
  }%
  \macc@depth\@ne
  \let\math@bgroup\@empty \let\math@egroup\macc@set@skewchar
  \mathsurround\z@ \frozen@everymath{\mathgroup\macc@group\relax}%
  \macc@set@skewchar\relax
  \let\mathaccentV\macc@nested@a
  \macc@nested@a\relax111{#1}%
  \endgroup
}
\makeatother

\newtheorem{theorem}{Theorem}[section]
\newtheorem{lemma}[theorem]{Lemma}
\newtheorem{prop}[theorem]{Proposition}
\newtheorem{corollary}[theorem]{Corollary}

\newtheorem{definition}[theorem]{Definition}

\theoremstyle{definition}

\theoremstyle{remark}
\newtheorem{remark}[theorem]{Remark}

\numberwithin{equation}{section}
 \allowdisplaybreaks

\newcommand{\R}{\mathbb R}
\newcommand{\N}{\mathbb N}
\newcommand{\Z}{\mathbb Z}

\newcommand{\subRn}{{{\mathbb R}^n}}

\newcommand{\Q}{\mathcal Q}
\newcommand{\D}{\mathcal D}
\newcommand{\Ss}{\mathcal S}

\newcommand{\F}{\mathcal F}

\newcommand{\M}{\mathcal M}

\newcommand{\W}{\mathcal{W}}
\newcommand{\barW}{\overline{\mathcal{W}}}

\newcommand{\A}{\mathcal A}
\newcommand{\B}{\mathcal B}

\DeclareMathOperator*{\esssup}{ess\,sup}

\DeclareMathOperator{\op}{op}
\DeclareMathOperator{\diag}{diag}

\DeclareMathOperator{\conv}{conv}
\DeclareMathOperator{\clconv}{\overline{conv}}

\newcommand{\K}{\mathcal{K}}
\newcommand{\acs}{\K_{acs}(\R^d)}
\newcommand{\bcs}{\K_{bcs}(\R^d)}
\newcommand{\abcs}{\K_{abcs}(\R^d)}

\newcommand{\Rdf}{\mathcal{R}}

\newcommand{\Md}{\mathcal{M}_d}
\newcommand{\Sd}{\mathcal{S}_d}

\newcommand{\ld}{m_d} 
\renewcommand{\ln}{{m}_n}

\newcommand{\E}{\mathcal E}

\makeatletter
\def\l@subsection{\@tocline{2}{0pt}{3.5pc}{5pc}{}}
\makeatother

\title{Extrapolation and Factorization of matrix weights}

\author{Marcin Bownik}
\address{Department of Mathematics,
University of Oregon,
Eugene, OR 97403--1222, USA}
\email{mbownik@uoregon.edu}

\author{David Cruz-Uribe, OFS}
\address{Department of Mathematics \\ University of Alabama \\
Tuscaloosa, AL 35487, USA}
\email{dcruzuribe@ua.edu}

\thanks{ The first author was partially supported by NSF grants
  DMS-1956395 and DMS-2349756. The second author was supported by research funds from
  the Dean of the College of Arts \& Sciences, the University of
  Alabama, and is currently partially supported by a Simons Foundation
  Travel Support for Mathematicians Grant.  The authors would like to
  thank Fedor Nazarov for an inspiring lecture on matrix weights at
  the University of Missouri over two decades ago, which led to the
  formulation of Proposition
  \ref{prop:general-reverse-factorization}.}



\begin{document}

\begin{abstract}
  In this paper we prove the Jones factorization theorem and the Rubio
  de Francia extrapolation theorem  for matrix $\A_p$ weights.  
  These results answer longstanding open questions in the study of
  matrix weights.  The proof
  requires the development of the theory of convex-set valued
  functions and measurable seminorm functions.  In particular, we define a convex-set valued version of the
  Hardy Littlewood maximal operator and construct
  an appropriate generalization of the Rubio de Francia iteration
  algorithm, which is central to the proof of both results in the
  scalar case. 
\end{abstract}


\vspace*{-0.1in}

\maketitle


\section{Introduction}
\label{section:introduction}

The purpose of this paper is to extend the theory of matrix $\A_p$
weights by proving the Jones factorization theorem \cite{jones}  and the Rubio de
Francia extrapolation theorem \cite{rubio} in this setting.  Our work answers a
longstanding open question first raised (we believe) by Nazarov and
Treil in 1996~\cite[Section~11.5.4]{MR1428988}.  To provide some
context for our results, we briefly recall some earlier work. For
further details, we refer the reader to~\cite{duoandikoetxea01,
  garcia-cuerva-rubiodefrancia85}.  The now classical $A_p$ weights
were introduced by Muckenhoupt and others in the 1970s.  A weight
(i.e., a non-negative, measurable function $w$ that satisfies
$0<w(x)<\infty$ a.e.) is said to satisfy $w\in A_p$, $1<p<\infty$, if
\begin{equation} \label{eqn:Ap-defn}
[w]_{A_p} =  \sup_Q \avgint_Q w(x)\,dx \bigg(\avgint_Q w(x)^{1-p'}\,dx\bigg)^{p-1}
  < \infty,  
\end{equation}
where the supremum is taken over all cubes in $\R^n$ with sides
parallel to the coordinate axes.  A weight $w$ is in $A_1$ if
\begin{equation} \label{eqn:A1-defn}
 [w]_{A_1} =  \sup_Q \esssup_{x\in Q} w(x)^{-1} \avgint_Q w(y)\,dy <
  \infty. 
\end{equation}
It was shown that \eqref{eqn:Ap-defn} is a
sufficient condition, when  $1<p<\infty$, for norm inequalities of the form
\[ \int_{\R^n} |Tf(x)|^p w(x)\,dx \leq C \int_{\R^n} |f(x)|^p
  w(x)\,dx,\]
and \eqref{eqn:A1-defn} is sufficient for the corresponding weak type estimate when $p=1$,
where $T$ is the Hardy-Littlewood maximal operator, a
Calder\'on-Zygmund singular integral, a square function, and other
classical operators of harmonic analysis.  

Two fundamental and closely related results in the study of weighted
norm inequalities are the Jones factorization theorem and the Rubio de
Francia extrapolation theorem.

\begin{theorem}[Jones Factorization Theorem]
  Given a weight $w$ and $1<p<\infty$, $w\in A_p$ if and only if there
  exist weights $w_0,\,w_1\in A_1$ such that $w=w_0w_1^{1-p}$.  
\end{theorem}

\begin{theorem}[Rubio de Francia Extrapolation]
Given $1\leq p_0 < \infty$, suppose that an operator $T$ is such that
for every $w_0 \in A_{p_0}$ and $f\in L^{p_0}(w_0)$,
\[ \int_{\R^n} |Tf(x)|^{p_0} w_0(x)\,dx
  \leq C_0([w_0]_{A_{p_0}}) \int_{\R^n} |f(x)|^{p_0}
  w_0(x)\,dx. \]
Then for every $p$, $1<p<\infty$, every $w\in A_p$, and every $f\in L^{p}(w)$,
\[ \int_{\R^n} |Tf(x)|^p w(x)\,dx \leq C([w]_{A_{p}}) \int_{\R^n} |f(x)|^p
  w(x)\,dx. \]
\end{theorem}

The proofs of both of these results are very closely related:  each depends
on the properties of the Rubio de Francia iteration algorithm
\[ \Rdf h(x) = \sum_{k=0}^\infty \frac{M^kh(x)}{2^k
    \|M\|_{L^p(w)}^k}, \]
where $M$ is the Hardy-Littlewood maximal operator.  (See the above
references and also~\cite{preprint-DCU,MR2797562}.)

Rubio de Francia extrapolation has had many important applications in
harmonic analysis and PDEs:  see, for
instance,~\cite{DuRu,MR1439553,MR2319768}.  In particular, it was
was central to the original proofs of the so-called ``$A_2$
conjecture'':  that is, the sharp constant estimate 
\[ \|Tf\|_{L^p(w)} \leq
  C(n,p,T)[w]_{A_p}^{\max\{1,p'-1\}}\|f\|_{L^p(w)}, \]
where $T$ is a Calder\'on-Zygmund singular integral.  By using a
sharp, quantitative version of extrapolation, the proof is reduced to
showing this inequality holds for $p=2$.  See
Hyt\"onen~\cite{hytonenP2010,MR3204859} and Lerner~\cite{Lern2012}. 

\medskip

We now turn from the theory of scalar weights to matrix weights.
Given a Calder\'on-Zygmund singular integral operator $T$, it extends
to an operator on vector-valued functions ${f}=(f_1,\ldots,f_d)^t$ by
applying it to each coordinate: $T{f}=(Tf_1,\ldots,Tf_d)^t$.  In a
series of papers in the
1990s~\cite{MR1428988,MR1428818,Vol,MR1030053,MR1478786}, Nazarov,
Treil and Volberg considered the question of whether there existed a
corresponding ``matrix'' $A_p$ condition on positive semidefinite,
symmetric (i.e., real self-adjoint) matrix functions $W$ such that
\[ \int_{\R^n} |W^{1/p}(x)T{f}(x)|^p \,dx
  \leq C \int_{\R^n} |W^{1/p}(x){f} (x)|^p \,dx, \]
This problem was motivated by applications to stationary processes and
to Toeplitz operators acting on vector-valued functions.  It was first
solved on the real line when $T$ is the Hilbert transform and 
$p=2$ by Treil and Volberg~\cite{MR1428818}.  They showed that a
sufficient condition on the matrix $W$ is a matrix analog of the $A_2$
condition:
\[ [W]_{A_2} =
  \sup_Q \bigg| \left(\avgint_Q W(x)\,dx\right)^{\frac{1}{2}}
  \left(\avgint_Q W^{-1}(x)\,dx\right)^{\frac{1}{2}}\bigg|_{\op}
  <\infty. \]
This condition, however, does not extend to the case $p\neq 2$.  An
equivalent, but more technical definition of matrix $A_p$ in terms of
norm functions was conjectured by Treil~\cite{MR1030053} and used by
Nazarov and Treil~\cite{MR1428988} and separately by
Volberg~\cite{Vol} to prove matrix weighted norm inequalities for the
Hilbert transform.  These authors noted two significant technical
obstructions.  The first was the lack of a ``vector-valued'' version
of the Hardy-Littlewood maximal operator that could bound
vector-valued operators but not lose the geometric information
imbedded in the vector structure.  The second was that proofs were
much easier in the case $p=2$, but that there was no version of the
Rubio de Francia extrapolation theorem to extend these results to
$p\neq 2$.

These results were extended to general Calder\'on-Zygmund singular
integrals in $\R^n$ by Christ and Goldberg~\cite{CG,MR2015733}.  A key
component of their proofs is to define for each $p$ a scalar-valued,
matrix weighted maximal operator :
\[ M_W{f}(x) =
  \sup_Q \avgint_Q |W^{1/p}(x)W^{-1/p}(y){f}(y)|\,dy\cdot \chi_Q(x).  \]
While sufficient for their approach, here we note one drawback of this
operator:  while ${f}$ and $T{f}$ are vector-valued operators,
$M_W{f}$ is scalar-valued, and so cannot be iterated.

Finally, we note that Roudenko~\cite{MR1928089} gave an equivalent
definition of matrix $A_p$ that looked more like the definition in the
scalar case:  $W\in A_p$ if and only if
\begin{equation} \label{eqn:roudenko}
  [W]_{A_p} = \sup_Q \avgint_Q \bigg( \avgint_Q
    |W^{1/p}(x)W^{-1/p}(y)|_{\op}^{p'}\,dy\bigg)^{\frac{p}{p'}}\,dx <
    \infty. 
  \end{equation}

  All of the estimates for singular integrals were qualitative: like
  the early proofs in the scalar case they did not give good estimates
  on the dependence of the constant on the value of $[W]_{A_p}$.
  After the sharp result in the scalar case was proved by Hyt\"onen,
  it was natural to conjecture that the same result holds in the
  matrix case: more precisely, that
\[ \bigg(\int_{\R^n} |W(x)^{1/p}T{f}(x)|^p \,dx\bigg)^{1/p}
  \leq C[W]_{A_p}^{\max\{1,p'-1\}} \bigg( \int_{\R^n} |W(x)^{1/p}{f} (x)|^p
  \,dx\bigg)^{1/p}.\]
This problem is referred to as the matrix $A_2$ conjecture; it was
first considered by Bickel, Petermichl and Wick~\cite{MR3452715} and
by Pott and Stoica~\cite{MR3698161} when $p=2$.  In 2017,
Nazarov, Petermichl, Treil and Volberg~\cite{MR3689742} proved that in
this case, the best constant is bounded above by
$C(n,d,T)[W]_{A_2}^{3/2}$.  (Also see~\cite{MR3818613}.)  Very
recently, Domolevo, Petermichl, Treil and Volberg~\cite{DPTV-2024}
proved that this is the best possible exponent.   For this problem
most of the work has been done on the case $p=2$ since this case is
easier than working with arbitrary $p$.  The only known quantitative results for $p\neq 2$ were
proved by the second author, Isralowitz and Moen~\cite{MR3803292}, who
got a constant of the form
\begin{equation} \label{eqn:sharp-p-neq-2}
 C(n,d,p,T)[W]_{A_p}^{1+\frac{1}{p-1}-\frac{1}{p}}.  
\end{equation}
It is an open question whether this estimate is sharp when $p\neq 2$. 


A very important tool in the more recent proofs of the $A_2$
conjecture in the scalar case is the domination of singular integrals
by sparse operators introduced by Lerner~\cite{Lern2012}.  Nazarov,
Petermichl, Treil and Volberg~\cite{MR3689742} extended this result to
vector-valued singular integrals by interpreting the vector $T{f}$
as a point in a convex set.  More precisely, they showed that there
exists a sparse collection of dyadic cubes $\Ss$,  depending on $T$ and $f$, such that
\[ T{f}(x) \in C\sum_{Q\in \Ss}
  \langle\langle{f}\rangle\rangle_Q \chi_Q(x), \]
where $\langle\langle{f}\rangle\rangle_Q$ is the convex set
\[  \langle\langle{f}\rangle\rangle_Q
  = \bigg\{ \avgint_Q k(y) {f}(y)\,dy : k \in L^\infty(Q),
  \|k\|_\infty \leq 1 \bigg\},  \]
and the sum is the (infinite) Minkowski
sum of convex sets.   However, instead of working directly with
convex-set valued functions, they reduced the problem to estimating
vector-valued sparse operators of the form
\[ T^S {f}(x) = \sum_{Q\in \Ss} \avgint_Q
  \varphi_Q(x,y){f}(y)\,dy, \]
where for each $Q$, $\varphi_Q$ is a real-valued function supported on
$Q\times Q$  such that, for each $x$,
$\|\varphi_Q(x,\cdot)\|_\infty \leq 1$.
Sparse domination has been generalized to other operators:
see~\cite{MR4225835, MR4454483, MR4357365,MR4245601,MR3936542}.
\medskip

Given this background, we can now describe our main results.  To do so
we must first introduce a change in notation.  For a number of reasons
connected to our proofs, we have chosen to write a
matrix weighted norm of a function $f: \R^n \to \R^d$ in the form
\[ 
||f||_{L^p(\R^n,W)} =
\bigg(\int_{\R^n} |W(x){f}(x)|^p\,dx\bigg)^{\frac{1}{p}}. \]
This is equivalent to replacing the matrix weight $W$ by $W^p$.  In
doing this we replace the class $A_p$ with the equivalent class
$\A_p$:
\[ [W]_{\A_p} = \sup_Q \bigg(\avgint_Q \bigg( \avgint_Q
    |W(x)W^{-1}(y)|_{\op}^{p'}\,dy\bigg)^{\frac{p}{p'}}\,dx\bigg)^{\frac{1}{p}} <
    \infty. \]
  We also define the classes $\A_1$ and $\A_\infty$ by
  \[ [W]_{\A_1} = \sup_Q \esssup_{x\in Q} \avgint_Q
    |W^{-1}(x)W(y)|_{\op}\,dy < \infty,  \]
and
\[ [W]_{\A_\infty} =\sup_Q \esssup_{x\in Q} \avgint_Q
  |W(x)W^{-1}(y)|_{\op}\,dy < \infty.  \]
The class $\A_1$ was first introduced by Frazier and
Roudenko~\cite{MR2104276}; the class $\A_\infty$ is new, though it was
implicit in the literature in the scalar case.  Note that in the scalar
case we can write the definition of $\A_p$ as 

\[ [w]_{\A_p}
  = \sup_Q |Q|^{-1}\|w\chi_Q\|_{L^p}
  \|w^{-1}\chi_Q\|_{L^{p'}}  < \infty, \]
and this makes sense even when $p=1$ or $p=\infty$.
This definition of $\A_\infty$ was implicit in 
Muckenhoupt~\cite{muckenhoupt72} but mostly overlooked.  It has been used to define a uniform $A_p$ condition,
$1\leq p \leq \infty$:  see Nieraeth~\cite{MR4000248}.   We also remark that this approach to 
scalar weighted norm inequalities is used for off-diagonal
inequalities and norm inequalities on Banach function spaces: see, for
instance,~\cite{CruzUribe:2016ji,MR2927495, muckenhoupt-wheeden74}.

With this notation, our main results are the following.

\begin{theorem} \label{thm:jones-factorization-intro}
  Fix $1<p<\infty$.  Given a matrix weight $W$, we have $W\in \A_p$ if and
  only if
  \[  W = W_0^{1/p} W_1^{1/p'}, \]
for some commuting  matrix weights  $W_0\in
  \A_1$ and  $W_1 \in \A_\infty$.
\end{theorem}

\begin{theorem} \label{thm:matrix-extrapolation-intro} Given an
  operator $T$,  suppose that for some $p_0$,
  $1 \leq p_0\leq\infty$, there exists an increasing function
  $K_{p_0}$ such that for every $W_0\in \A_{p_0}$,
\begin{equation}\label{ext1}
 \|Tf\|_{L^{p_0}(\R^n, W_0)}
  \leq K_{p_0}([W_0]_{\A_{p_0}})\|f\|_{L^{p_0}(\R^n, W_0)}. 
  \end{equation}
  Then for all $p$, $1<p<\infty$,
  for all $W\in \A_p$, and for all $f\in L^\infty_c(\R^n)$, 
\begin{equation}\label{ext2}
    \|Tf\|_{L^p(\R^n,W)}
  \leq K_{p}(p,p_0,n,d,[W]_{\A_{p}})\|f\|_{L^{p}(\R^n,W)},
  \end{equation}
  where
  \[ K_p (p,p_0,n,d,[W]_{\A_{p}})
    = C(p,p_0)K_{p_0}\bigg(C(n,d,p,p_0)[W]_{\A_p}^{\max\big\{\frac{p}{p_0},\frac{p'}{p_0'}\big\}}\bigg). \]
  Moreover, if $T$ is linear, it has a continuous extension to all
  $f\in L^{p}(\R^n,W)$ that satisfies the same bound. 
\end{theorem}

\begin{remark}
  For simplicity and ease of comparison to the scalar case, we state
  Theorem~\ref{thm:jones-factorization-intro} assuming that the
  matrices $W_0$ and $W_1$ 
  commute.  We can remove this hypothesis, but to do so we must
  replace the product $W_0^{1/p} W_1^{1/p'}$ with the geometric
  mean of the two matrices.  See Proposition~\ref{prop:reverse-factorization}.
   One interesting feature of our proof is
  that in  constructing the matrices $W_0$ and $W_1$, we
  show that they can be realized as scalar multiples of $W$.
\end{remark}

\begin{remark}
  We actually prove a more general version of
  Theorem~\ref{thm:matrix-extrapolation-intro}, replacing the operator
  $T$ by a family of pairs of functions $(f,g)$.  This more abstract
  approach to extrapolation was first suggested
  in~\cite{cruz-uribe-perez00} and systematically developed
  in~\cite{MR2797562}.  
\end{remark}

\begin{remark}
  In Theorem~\ref{thm:matrix-extrapolation-intro} the function $K_p$
  depending on $K_{p_0}$ has exactly the same form as the function
  gotten in the sharp constant extrapolation theorem of
  Dragi{\v{c}}evi{\'c}, {\em et al.}~\cite{MR2140200}.  Note also that
  we are able to begin the extrapolation from $p_0=\infty$; this gives
  a quantitative version of a result proved in the scalar case by Harboure, {\em et
    al.}~\cite{MR944321}; this quantitative version  was recently
  proved by Nieraeth~\cite[Corollary~4.14]{MR4000248}.  
\end{remark}

\begin{remark}
By~\cite[Propositions~3.6,~3.7]{MR3544941} we have
that $L^\infty_c(\R^n,\R^d)$ is dense in $L^p(\R^n,W)$ for any matrix weight
$W$ and $1\leq p<\infty$.  In Theorem~\ref{thm:matrix-extrapolation-intro}, the set
$L^\infty_c(\R^n)$ can be replaced by any  collection which is dense
in $L^p(\R^n,W)$ and contained in every scalar weighted space $L^p(\R^n,w)$.   If
$T$ is not linear, the problem of proving the continuous extension
exists is more delicate, as is the problem of showing this abstract
extension agrees with the original operator.   We consider a specific
example in Theorem~\ref{thm:rough-sio} below.
\end{remark}


To prove the Jones factorization theorem and Rubio de Francia
extrapolation for matrix weights, we considerably expand upon the
ideas underlying the convex-set sparse domination theorem described
above.  To do so, we draw upon an extensive literature on convex-set
analysis (see, for instance,~\cite{MR2458436,CV,Rock}) which does not
seem to have been previously applied to problems in harmonic analysis.
We define measurable functions $F : \R^n \rightarrow \K$, where $\K$
is (a subset of) the collection of convex sets in $\R^d$, and develop
the connection between norm functions and convex-set valued functions.
There is a one-to-one correspondence between measurable norm functions
and measurable convex-set valued functions.  As noted above, the
matrix $A_p$ condition was originally defined in terms of norm
functions, but the trend, at least since the work of
Roudenko~\cite{MR1928089} and Goldberg~\cite{MR2015733},  has been to
interpret it only in terms of matrices.  We go back to this definition
in terms of norm functions; this proved to be essential at several
points in our proofs as it provides the necessary link between
matrices and convex-set valued functions.

We define a convex-set valued version of the
Hardy-Littlewood maximal operator by using the so-called Aumann
integral of convex-set valued functions (see~\cite{MR2458436}) to define
the maximal operator
\[ MF(x) = \clconv\bigg(\bigcup_Q \avgint_Q F(y)\,dy\cdot\chi_Q(x)
  \bigg). \]
With this definition we get analogs of all the properties of the scalar
maximal operator:  in particular,   it dominates $F$ via inclusion,
$F(x)\subset MF(x)$;  and is bounded on $L^p(W)$ when $W\in
\A_p$.   Most importantly,  it maps convex-set
valued functions to convex-set valued functions, and therefore can be
iterated.  This allows us to define the Rubio de Francia iteration
algorithm for convex-set valued functions:
  \[\Rdf H(x)=\sum_{k=0}^\infty \frac{M^kH(x)}{2^k\|M\|_{L^p(W)}^k}. \]
  This operator has properties analogous to the scalar operator:
  $H(x) \subset \Rdf H(x)$;
  $\|\Rdf H\|_{L^p(W)} \leq 2\|H\|_{L^p(W)}$; and $\Rdf H$ satisfies a
  convex-set valued $A_1$ condition;
  $M(\Rdf H)(x) \subset C\Rdf H(x)$.  This property is closely related
  to the $A_1$ condition for norm functions (and so for matrix
  weights).  With this version of the iteration algorithm, we are able
  to extend the scalar proofs of factorization and extrapolation to
  the matrix case.  The overall outline of the proofs is similar to
  those of the scalar results (see~\cite{CruzUribe:2016ji,MR2797562},
  but there are a significant number of technical obstacles which must
  be addressed.  Here we note the two most difficult: first, matrix
  functions do not, in general, commute.  Second, while it is possible
  to define powers of matrices (and so of ellipsoids), it is not
  possible to define powers of arbitrary convex sets.  (See Milman and
  Rotem~\cite{MR3709538,MR3985279}.)  Therefore, at several points we
  need to pass back and forth between convex-set valued functions and
  ellipsoid valued functions.


  \begin{remark}
    The fact that we must specialize to consider ellipsoid valued
    functions might suggest that we could simplify our approach to
    extrapolation by restricting to these kinds of functions rather
    than working with the more general convex-set valued functions.
    However, even for vector-valued functions, the Aumann averages and
    the convex-set valued maximal operator will yield convex sets that
    are not ellipsoids.  We give an example in
    Section~\ref{section:maximal}.   Therefore, it is necessary for us
    to develop the general machinery of convex-set valued functions
    for our proofs.
  \end{remark}

  \medskip
  
  We now want to briefly consider applications of our results.  As has been noted in the literature (e.g.,
  in~\cite{domelevo2021boundedness,MR4159390}), many problems in
  matrix weighted inequalities are significantly easier to prove when
  $p=2$ than for all $p$ (see, for instance,
  \cite{MR1428818,CG,MR3698161,MR3689742,MR3452715,MR3937322,MR4245601}).
  But by applying Theorem~\ref{thm:matrix-extrapolation-intro}, these
  results can immediately be extended to the full range $1<p<\infty$.
  For instance,  the $L^2$ bounds in~\cite{MR3689742} for singular
  integrals immediately
  extend to all $p$.   In~\cite{MR4245601}, the authors prove
  matrix-weighted $L^2$
  bounds for maximal rough singular integrals; using extrapolation we
 extend these results to all $L^p$:  see Theorem~\ref{thm:rough-sio} below.  This improves the results
  of~\cite{MR4357365}, which give $L^p$ bounds for rough singular
  integrals. 

  However, it is surprising (at least to the authors), that
  extrapolation, which yields the best possible constants for singular
  integrals in the scalar case, does not yield sharp results in the
  matrix case.  For example, by extrapolation, starting with the sharp exponent
  $[W]_{A_2}^{3/2}$ from~\cite{MR3689742}, we get
  $[W]_{A_p}^{\frac{3}{2}\{1,\frac{1}{p-1}\}}$, which is worse than the
  constant \eqref{eqn:sharp-p-neq-2} gotten in~\cite{MR3803292}.
  Similarly, extrapolating the $L^2$ bound for rough singular
  integrals in~\cite{MR4245601} gets a worse constant than gotten
  in~\cite{MR4357365} for $p\neq 2$.  

  Extrapolation should also prove to be useful in other settings.  For
  instance, Vuorinen~\cite{vuorinen2023strong} has proved that our
  extrapolation theorem could be extended to the setting of ``strong
  matrix $A_p$'' which is associated with the basis of rectangles.  He
  used this to prove that a result due to Domelevo, {\em et
    al.}~\cite{domelevo2021boundedness} for bi-parameter Journ\'e operators,  which they were only able to
  prove in $L^2$ for strong matrix $A_2$ weights, holds for all $p$.


  \medskip

The remainder of this paper is organized as follows.  To prove
factorization and extrapolation we need to establish a large number of
preliminary results.  This is done in
Sections~\ref{section:prelim}--\ref{section:seminorm}.  In
Section~\ref{section:prelim} we present a number of results about
convex sets and seminorms.  Most of these results are known and we
gather them here for ease of reference and to establish consistent
notation.  However, some results are new (or rather, we could not find
them in the literature).  In particular, we prove some basic results
about the geometric mean of two norms that are essential to the proof
of factorization.

    In Section~\ref{section:convex-set-valued-functions} we define
    measurable, convex-set valued functions and establish the
    properties of the Aumann integral necessary to define the maximal
    operator on convex-set valued functions.  We have gathered
    together, with consistent hypotheses and notation, a number of
    theorems from across the literature and proved some results
    specific to our needs, such as a version of Minkowski's inequality
    for the Aumann integral (Proposition~\ref{prop:convex-minkowski}).  Since much of this material appears
    unfamiliar to most harmonic analysts, and since there are a number
    of delicate issues related to measurability of convex-set valued
    functions, we have included most details and we give extensive
    references to the literature.

    In Section~\ref{section:seminorm} we define seminorm functions,
    explore their connection with measurable convex-set valued
    functions, and define the norm-weighted $L^p$ spaces of convex-set
    valued functions.  These are not Banach spaces, but have most of
    the same properties, which allows us to rigorously define the
    Rubio de Francia iteration algorithm.  Finally, we make explicit
    the connection between measurable seminorm functions and matrix weights.

In the remaining sections we develop our new results.  
In Section~\ref{section:maximal} we define averaging operators and the
maximal operator on convex-set
valued functions.  We prove that this maximal operator has properties
that are the exact analogs of those of the scalar maximal operator, and we
prove unweighted norm inequalities by adapting
the scalar proof (using dyadic cubes) to the setting of convex-set
valued functions (Theorem~\ref{thm:M-norm-ineq}). 

In Section~\ref{section:matrix-weights} we turn to the
definition of matrix $\A_p$ in terms of norm functions.  Many of these results are already in
the literature, but, because we have chosen to take a different approach
than what has been done previously, we believed it was important to carefully
restate these results to incorporate the endpoint results when $p=1$
and $p=\infty$. 
The main result of this section is that the convex-set valued maximal
operator is bounded on $L^p(W)$, $1<p\leq \infty$, when $W\in \A_p$.
The proof uses a measurable version of the John ellipsoid theorem (Theorem~\ref{mj})
to reduce to norm inequalities for the Christ-Goldberg matrix weighted
maximal operator.

In Section~\ref{section:rubio} we define convex-set valued
$\A_1^\K$ weights and show that there is a one-to-one correspondence between
them and norm functions in $\A_1$ (Theorem~\ref{thm:matrix-convex-A1}); this gives us a connection between convex-set
$\A_1^K$ and matrix $\A_1$ weights that is needed for the proof of
extrapolation.  We then define a generalized Rubio de Francia
iteration algorithm which includes the version given above and
which covers the various forms of the operator used in the proofs of
factorization and extrapolation
(Theorem~\ref{thm:gen-rdf-algorithm}). 

In Section~\ref{section:factorization} we state and prove our version
of the Jones factorization theorem (restated there as
Theorem~\ref{thm:jones-factorization}).  The proof is based on that of the scalar version given in~\cite{preprint-DCU}.  In the
scalar case, the difficult direction is to prove that an $A_p$ weight
can be factored as the product of $A_1$ weights; the other direction,
sometimes referred to as ``reverse factorization'', is an immediate
consequence of the definition of $A_p$ weights.  In the matrix case,
however, both directions are difficult.  The proof of factorization is
based on the Rubio de Francia iteration algorithm and follows the
scalar proof given in~\cite{preprint-DCU}.  The proof of reverse
factorization is more delicate: it is here that we were required to
work with the definition of $\A_p$ in terms of norm functions. Our
final proof is, implicitly, based on an interpolation argument between
finite dimensional spaces.

In Section~\ref{section:extrapolation} we state and prove
a sharp constant version of Rubio de Francia extrapolation for matrix weights
(Theorem~\ref{thm:matrix-extrapolation}).  The proof is based on the
approach to extrapolation developed in~\cite{MR2797562}, and so
reverse factorization is a central part of the proof.  We adopt the
perspective of working with families of extrapolation pairs $(f,g)$,
which completely avoids any mention of operators.
Using our
  definition of $\A_p$ weights, we are also able to give a uniform
  proof that includes the endpoint results when $p=\infty$.  This
  yields a quantitative version of a result proved in the scalar case
  by Harboure, Mac\'\i as and Segovia~\cite{MR944321}.

  Finally, in Section~\ref{section:applications} we discuss some of
  the technical details involved in applying
  Theorem~\ref{thm:matrix-extrapolation}, and we illustrate this by
  proving quantitative $L^p$ bounds for maximal rough singular
  integral operators, extending the results from~\cite{MR4245601}.

\medskip

Throughout this paper we will use the following notation.  We will
develop some things in the setting of abstract measure spaces; in
this setting $(\Omega,\A, \mu)$ will denote a 
$\sigma$-finite, complete measure space endowed with a positive
measure $\mu$.  In Euclidean space the constant $n$ will denote
the dimension of $\R^n$, which will be the domain of our functions.  The value
$d$ will denote the dimension of vector and set-valued functions.   In
$\R^d$, $\B$  will denote
the $\sigma$-algebra of Borel sets, and  $\ld$ will denote the Lebesgue
measure.  For $1\leq p \leq \infty$, $L^p(\R^n)$ will denote the Lebesgue space
 of scalar functions, and $L^p(\R^n,\R^d)$ will denote the Lebesgue
 space of vector-valued functions.  

 Given $v=(v_1,\ldots,v_d)^t \in \R^d$, the Euclidean norm of $v$ will
 be denoted by $|v|$.  The standard orthonormal basis in $\R^d$ will
 be denoted by $\{e_i\}_{i=1}^d$.  The open unit ball in
 $\{ v \in \R^d : |v|<1\}$ will be denoted by ${\mathbf B}$ and its
 closure by $\overline{\mathbf B}$.  Matrices will be $d\times d$
 matrices with real-valued entries unless otherwise specified.  The
 set of all such matrices will be denoted by $\Md$.  The set of all
 $d\times d$, symmetric (i.e., self-adjoint), positive semidefinite matrices will be denoted by
 $\Sd$.  We will denote the transpose of a matrix $A$ by $A^*$.  Given
 two quantities $A$ and $B$, we will write $A \lesssim B$, or
 $B\gtrsim A$ if there is a constant $c>0$ such that $A\leq cB$.  If
 $A\lesssim B$ and $B\lesssim A$, we will write $A\approx B$.

\section{Convex sets and seminorms}
\label{section:prelim}

In this section we develop the connections between convex sets in
$\R^d$ and seminorms defined on $\R^d$.  We begin with some basic
definitions and notation.  Given a set $E\subset \R^d$, let
$\overline{E}$ denote the closure of $E$.  Given two sets
$E,\,F \subset\R^d$, define their Minkowski sum to be the set
\[ E+F = \{ x+y : x\in E, y\in F\}. \]
For $\lambda \in \R$, define $\lambda E = \{ \lambda x : x\in E\}$.  A
set $E$ is symmetric if $-E=E$.  A set $E$ is absorbing if for every
$v\in \R^d$, $v\in tE$ for some $t>0$.  

A set $K\subset \R^d$ is convex if for all $x,\,y\in K$ and
$0<\lambda<1$, $\lambda x+ (1-\lambda)y \in K$.  For the basic
properties of convex sets, see~\cite{Rock,MR1216521}. Given a set $E$,
let $\conv(E)$ denote the convex hull of $E$: the smallest convex set
that contains $E$.  Equivalently, $\conv(E)$ consists of all finite
convex combinations elements in $E$:
\[ \conv(E)= \bigg\{ \sum_{i=1}^k \alpha_i x_i : x_i \in E, \alpha_i \ge
  0, \sum_{i=1}^k \alpha_i = 1 \bigg\}.  \]
The convex hull is additive:  given two sets
$E,\,F$, $\conv(E)+\conv(F)=\conv(E+F)$.  We will denote the closure
of the convex hull of $E$ by $\clconv(E)$.

Let $\K(\R^d)$ be the collection of all closed, nonempty subsets of
$\R^d$. The subscripts $a$, $b$, $c$, $s$ appended to $\K$ will denote
absorbing, bounded, convex, and symmetric sets, respectively. Since
$\R^d$ is finite dimensional, a convex set $K$ is absorbing if and
only if $0\in \operatorname{int}(K)$. We are particularly interested
in the following two subsets of $\K$:
\begin{itemize}
\item $\K_{acs}(\R^d)$:  absorbing, convex, symmetric,
  and closed subsets of $\R^d$;
\item $\bcs$: bounded, convex, symmetric, and closed
  subsets of $\R^d$.
\end{itemize}

We generalize the norm on $\R^d$ by introducing the concept of a
seminorm. 

\begin{definition}
 A {\it seminorm} is a function $p: \R^d \to [0,\infty)$ that
 satisfies the following properties:  for all $u,\,v \in \R^d$ and
 $\alpha \in \R$,
 \begin{align}
\label{min1}
p(u+v) &\le p(u) + p(v), \\
\label{min2}
p(\alpha v) &= |\alpha| p(v).
 \end{align}
A seminorm is a norm if $p(v)=0$ if and only if $v=0$. 
\end{definition}

\begin{definition} \label{mink}
Given $K \in \acs$ define the corresponding {\it Minkowski functional} $p_K: \R^d \to [0,\infty)$ by
\[
p_K(v) = \inf \{ r>0: v/r \in K \}.
\]
\end{definition}

\begin{definition}
Given a seminorm $p$, define the unit ball of $p$ to be the set
\[ K(p) = \{ v \in \R^d : p(v) \leq 1 \}. \]
\end{definition}

By properties \eqref{min1} and \eqref{min2}, the unit ball $K(p)$
is a convex, absorbing, symmetric set.  In fact, using the Minkowski
functional, there is a one-to-one correspondence between sets
$K\in \acs$ and seminorms $p$.  For a proof of the following result,
see~\cite[p. 210]{MR2302906} or~\cite[Theorems 1.34 and 1.35]{Ru}.

\begin{theorem}\label{min}
  Given any $K\in \acs$, the Minkowski functional $p_K$ satisfies
  seminorm properties \eqref{min1} and \eqref{min2}.  Conversely, given
  any seminorm $p$, the unit ball $K(p)\in \acs$.  
This correspondence between sets in $\acs$ and seminorms is one-to-one.
\end{theorem}

Since $\R^d$ is finite dimensional, all norms on it are equivalent.
Therefore, given a norm $p$, $K(p)$ is bounded.  Conversely, if $K\in
\abcs$, then $p_K$ is a norm~\cite[p.~210]{MR2302906}.  This gives the
following corollary to Theorem~\ref{min}

\begin{corollary} \label{cor:min-norm}
  There is a one-to-one correspondence between norms on $\R^d$ and the
  set  $\abcs$, given by the map $K\mapsto p_K$.
\end{corollary}

\medskip

There is another correspondence between convex sets and seminorms, one
which will be more useful for our purposes below.  To state it, we
need to introduce the concept of the dual seminorm and the polar of a
convex set.  The proof of the following result follows at once from
the properties of a seminorm. 

\begin{lemma} \label{lemma:dual-seminorm}
  Given a seminorm $p$, define $p^* : \R^d \rightarrow [0,\infty)$ by
\[
p^*(v)= \sup_{w\in \R^d, \ p(w) \le 1} |\langle v, w \rangle|.
\]
Then $p^*$ is a seminorm.  If $p$ is a norm, the definition may be
written as
\begin{equation} \label{eqn:dual-norm}
 p^*(v)
  = \sup_{w \in \R^d, w \neq 0}
  \frac{ | \langle v, w \rangle |}{|p(w)|}. 
\end{equation}
\end{lemma}

\begin{definition}\label{polar}
Given $K \in \mathcal K_{cs}(\R^d)$, define its {\it polar set} by
\[
K^\circ= \{ v\in \R^d: |\langle v , w \rangle| \le 1 \quad\text{for all } w \in K\}.
\]
\end{definition}

A polar set can be thought of as the ``dual'' of a convex set.  This
is made precise by the  following result; for a proof, see \cite[Theorem 14.5]{Rock}.

\begin{theorem}\label{pol}
Let $K \in \mathcal K_{cs}(\R^d)$. The following statements hold:
\begin{enumerate}[(a)]
\item
If $K\in \acs$, then $K^\circ \in \bcs$. 
\item
If $K \in \bcs$, then $K^\circ \in \acs$. 
\item
$(K^\circ)^{ \circ}=K$.
\item If $K$ is bounded and absorbing, then $p_K$ is a norm and
\[
p_{K^\circ} = (p_K)^*.
\]
\end{enumerate}
\end{theorem}

As a corollary to Theorem~\ref{pol} and Corollary~\ref{cor:min-norm}
we have the following.

\begin{corollary} \label{cor:dual-norm}
  Given a norm $p$, then the dual seminorm $p^*$ is a norm.  Moreover,
  $p^{**}=(p^*)^*=p$.
\end{corollary}

We can now state another characterization of seminorms in terms of
convex sets.  The proof is an immediate consequence of
Theorems~\ref{min} and~\ref{pol}.  

\begin{theorem}\label{eqfun}
  The mapping $ K \mapsto p_{K^\circ}$ defines a one-to-one
correspondence between bounded, convex, symmetric sets and
seminorms on $\R^d$.  More precisely,
 if $K \in \bcs$, then $p_{K^\circ}$ is a seminorm; conversely, 
 if $p$ is a seminorm, then $K(p)^\circ \in  \bcs$.
\end{theorem}

The next result gives some important properties of seminorms induced
by polar sets.

\begin{theorem} \label{eqnfun-bis}
  The following are true:
  
\begin{enumerate}[(a)]
\item
Let $K_1, K_2 \in \bcs$. Then $K_1 + K_2 \in \bcs$ and 
\[
p_{(K_1 + K_2)^\circ}=p_{K_1^\circ}+p_{K_2^\circ}.
\]
\item
Let $K\in\bcs$ and $\alpha \in \R$. Then $\alpha K = |\alpha| K \in\bcs$ and
\[
p_{(\alpha K)^\circ}= |\alpha| p_{K^\circ}.
\]
\item
Let $\{K_i\}_{i\in \N}\subset\bcs$. If $K=\clconv (\bigcup_{i\in \N} K_i)$ is bounded, then 
\[
p_{K^\circ} = \sup_{i\in \N} p_{(K_i)^\circ}.
\]
\item
Let $\{K_i\}_{i\in \N} \subset\bcs$ be a family of nested sets, with
$K_{i+1}\subset K_i$ for all $i$.   If $K=\bigcap_{i\in \N} K_i$, then 
\[
p_{K^\circ} = \inf_{i\in \N} p_{(K_i)^\circ}.
\]

\end{enumerate}
\end{theorem}

To prove Theorem \ref{eqnfun-bis} we introduce the concept of a support
function.

\begin{definition}\label{supf}
Given $K \in \bcs$, its support function $h_K: \R^d \to [0,\infty)$ is
defined to be
\[
h_K(v) = \sup_{w\in K} \langle v, w \rangle.
\]
\end{definition}

Note that since $K$ is symmetric, we can write $|\langle v, w
\rangle|$ in the definition of the support function. 

\begin{lemma}\label{sf}
  Given $K\in \bcs$, $h_K= p_{K^\circ}$. 
\end{lemma}

\begin{proof}
  By Definitions \ref{mink}, \ref{polar}, and \ref{supf}, for any
  $v\in \R^d$,
\[
\begin{aligned}
p_{K^\circ}(v) &= \inf\{ r>0 : v/r \in K^\circ\}
= \inf \{ r>0: |\langle v/r, w \rangle | \le 1 \text{ for all }w\in K \}\\
&= \inf \{ r>0: |\langle v, w \rangle | \le r \text{ for all }w\in K \} = \sup \{ |\langle v, w \rangle |: w \in K \} = h_K(v). 
\qedhere
\end{aligned}
\]
\end{proof}

\begin{proof}[Proof of Theorem~\ref{eqnfun-bis}]
  To prove $(a)$, first note that by Lemma~\ref{sf} applied twice,
  \[ p_{(K_1+K_2)^\circ}(v)
    =
    \sup_{w\in K_1+K_2} \langle v, w \rangle
    = \sup_{w_1\in K_1} \langle v, w_1 \rangle + \sup_{w_2\in K_2}
    \langle v, w_2 \rangle
    = p_{K_1^\circ}(v) + p_{K_2^\circ}(v). \]
  Part $(b)$ is proved similarly.  To prove $(c)$ we use
  Lemma~\ref{sf} and the definition of the convex hull:
  \begin{multline*}
    p_{K^\circ}(v)
    = \sup_{w\in K } \langle v, w \rangle 
    = \sup\bigg\{ \sum_{i=1}^k  \alpha_i \langle v, w_i \rangle : w_i \in
    K_i, k>0, \alpha_i \geq 0, \sum_{i=1}^k \alpha_i =1 \bigg\} \\
    = \sup\bigg\{ \sum_{i=1}^k  \alpha_i  p_{K_i^\circ}(v)
    : k>0, \alpha_i \geq 0, \sum_{i=1}^k \alpha_i =1\bigg\}
    = \sup_i p_{K_i^\circ}(v). 
  \end{multline*}

  Finally, to prove $(d)$, first note that for each $i$, $K\subset K_i$,
  so we have 
  \[  p_{K^\circ}(v)
    =
    \sup_{w\in K } \langle v, w \rangle
    \leq
    \sup_{w_i\in K_i } \langle v, w_i \rangle
    =
    p_{K_i^\circ}(v). \]
  Hence,  $p_{K^\circ}(v) \leq \inf_i p_{K_i^\circ}(v)$.

 To prove that equality holds, suppose to the contrary that this is a
 strict inequality.  Let
 \[ \epsilon = \inf_i p_{K_i^\circ}(v) - p_{K^\circ}(v) > 0. \]
Hence, for each $i$,
$ p_{K_i^\circ}(v) - p_{K^\circ}(v) \geq \epsilon$. 
By Lemma~\ref{sf}, for each $i$ there exists $w_i\in K_i$ such that 
$\langle v, w_i \rangle + \frac{\epsilon}{2} > p_{K_i^\circ}(v)$.
Hence, 
\[ \langle v, w_i \rangle - p_{K^\circ}(v) > \frac{\epsilon}{2}.  \]
Since the $K_i$ are nested, by passing to a subsequence we may assume
that the $w_i$ converge to a point $w\in K$ as $i\rightarrow \infty$.
Therefore, by continuity,
\[ \langle v, w \rangle - p_{K^\circ}(v) \geq
  \frac{\epsilon}{2},   \]
which contradicts the fact that $\langle v, w \rangle \leq
h_k(v)=p_{K^\circ}(v)$.  Hence, equality holds.
\end{proof}

\medskip

We now consider the weighted geometric mean of two norms.  These results will
be very important in the proof of reverse factorization in
Section~\ref{section:factorization}.  Let $p_0,\, p_1$ be two norms.
For $0<t<1$, define for all $v\in \R^d$,
\begin{equation} \label{eqn:geo-mean}
 p_t(v) = p_0(v)^{1-t}p_1(v)^t. 
\end{equation}
The function $p_t$ need not be a norm, though it is homogenous:
$p_t(\alpha v) =|\alpha| p_t(v)$.  However, we can still use the
the definition in Lemma~\ref{lemma:dual-seminorm} to define $p_t^*$,
which will be norm.

\begin{lemma} \label{lemma:geometric-mean}
  Given norms $p_0,\,p_1$ and $0<t<1$, define $p_t$ by
  \eqref{eqn:geo-mean}.   If we define $p_t^*$ by
  equation~\eqref{eqn:dual-norm}, then $p_t^*$ is a norm.
\end{lemma}

\begin{proof}
  First note that since $p_0,\,p_1$ are norms, if $w\neq 0$,
  $p_t(w)\neq 0$, so $p_t^*$ is well defined.  That it is a norm then
  follows immediately from the properties of the Euclidean inner
  product.
\end{proof}

By Corollary~\ref{cor:dual-norm} we also have that $p_t^{**}$ is a norm; though
it is not equal to $p_t$ we will be able to use it in place of $p_t$. 

\begin{lemma} \label{lemma:double-dual-geo-mean}
  Given norms $p_0,\,p_1$ and $0<t<1$, define $p_t$ by
  \eqref{eqn:geo-mean}.   Then for all $v\in \R^d$, $p_t^{**}(v) \leq
  p_t(v)$. 
\end{lemma}

\begin{proof}
  Fix $\epsilon>0$; by the
  definition of the dual norm, there exists $w\in \R^d$ such that
  \[ p_t^{**}(v)  \leq (1+\epsilon) \frac{ | \langle v, w \rangle
      |}{p_t^*(w)}. \]
  By the definition of $p_t^*$,
  \[ \frac{1}{p_t^*(w)}
    = \inf_{u\in \R^d, u\neq 0} \frac{p_t(u)}
    {|\langle w, u \rangle|}
    \leq \frac{p_t(v)}{|\langle w, v \rangle|}. \]
  If we combine these inequalities we get $p_t^{**}(v) \leq
  (1+\epsilon)p_t(v)$; since $\epsilon$ is arbitrary the desired
  inequality holds.
\end{proof}

\begin{remark}
  As a consequence of this result, we have that the unit ball of $p_t^{**}$ is the
  convex hull of the set $\{ v \in \R^d : p_t(v) \leq 1\}$.  Since we
  do not need this fact, we omit the details.
\end{remark}

To prove the next result about the dual of $p_t^*$, we need a lemma
which follows from the existence of the John
ellipsoid~\cite[Theorem~3.13]{TV}.  Since we prove this result in
detail for measurable norm functions in Section~\ref{normf} (see
Theorem~\ref{thm:matrix-norm} and
Proposition~\ref{prop:dual-matrix-norm}) we omit the simpler proof
here.

\begin{lemma} \label{lemma:norm-matrix}
  The following hold for all $v\in \R^d$:
  \begin{enumerate}

    \item Given a norm $p$, there exists a positive definite matrix $A$ such
  that $p(v)\approx |Av|$, where the implicit constants depend only on
  $d$.

  \item Given any invertible matrix $A\in \M_d$, $p(v)=|Av|$ is a norm and
    $p^*(v)= |(A^*)^{-1}v|$.

  \end{enumerate}
\end{lemma}

The following result shows two positive definite matrices are
simultaneously congruent to diagonal matrices, see
\cite[Ex. 1.6.1]{bhatia}.  For completeness we include the short proof.

\begin{lemma}\label{simd}
Let $A, B \in \Sd$. Then, there exists an
invertible matrix $S$, and diagonal matrices $D_A$ and $D_B$ such that
\begin{equation}\label{simd1}
 A= S^*D_A S, \qquad B= S^* D_B S.
\end{equation}
In particular, we can choose $D_A$ to be the identity matrix.

\end{lemma}

\begin{proof} Since the matrix $A^{-1/2} B A^{-1/2}$ is symmetric, there exists an orthogonal matrix $U$ and diagonal matrix $D_B$ such that $A^{-1/2} B A^{-1/2}=U^*D_BU$. Let $S=UA^{1/2}$ and $D_A=\mathbf I$, where $\mathbf I$ is the identity matrix. Then, a simple calculation shows that
\begin{align*}
S^* D_A S &= A^{1/2} U^* U A^{1/2} = A, \\
S^* D_B S & = A^{1/2} U^* (U A^{-1/2} B A^{-1/2} U^*) U A^{1/2} = B. \qedhere
\end{align*} 
\end{proof}

\begin{prop} \label{prop:geometric-double-dual}
  Given norms $p_0,\,p_1$ and $0<t<1$, define $p_t$ by
  \eqref{eqn:geo-mean}.   Then for all $v\in \R^d$,
  \begin{equation} \label{eqn:geometric-double-dual1}
    p_t^{**}(v) \approx \big( p_0^*(\cdot)^{1-t}p_1^*(\cdot)^t\big)^{*} (v). 
  \end{equation}
  The implicit equivalence constant depends only on  $d$.
\end{prop}

\begin{proof}
Given the two norms $p_0$ and $p_1$, by  Lemma~\ref{lemma:norm-matrix}
there exist positive definite matrices
$C$ and $D$ such that $p_0(v)\approx |Cv|$ and $p_1(v) \approx |Dv|$,
where the implicit constants depend only on $d$.   Let $A=C^2$, $B=D^2$; then we
have that
\[ p_0(v) \approx |A^{1/2}v| = \langle Av, v \rangle^{\frac{1}{2}}, \quad
  p_1(v) \approx |B^{1/2}v| = \langle Bv, v \rangle^{\frac{1}{2}}.  \]
By Lemma \ref{simd} there exists an
invertible matrix $S$ and diagonal matrices $D_A$ and $D_B$ such that 
\eqref{simd1} holds.
Then for all $v\in \R^d$,  $p_0(v) \approx|D_A^{1/2}Sv|$ and
$p_1(v) \approx|D_B^{1/2}Sv|$.  
Let $D_A^{1/2} = \diag(\lambda_1,\ldots,\lambda_d)$ and
  $D_B^{1/2}=\diag(\mu_1,\ldots,\mu_d)$.

  Since $\{S^*e_i\}_{i=1}^d$ is a basis, we
  can write $w\in \R^d$ as
  \[ w = \sum_{i=1}^d b_i S^*e_i. \]
 If we let $v=S^{-1}e_i$,
  then $p_t(v) = \lambda_i^{1-t}\mu_i^t$. Hence, 
  \begin{equation}\label{no1}
  p_t^*(w) = \sup_{\substack{v\in \R^d\\v\neq 0}} \frac{|\langle w, v \rangle|}{p_t(v)}
    \gtrsim \max_{1\leq i \leq d}  \frac{|\langle w, S^{-1}e_i \rangle|}{p_t(S^{-1}e_i)}
    = \max_{1\leq i \leq d}  \frac{|b_i|}{\lambda_i^{1-t}\mu_i^t}.
    \end{equation}
    
Similarly, since $\{S^{-1}e_i\}$ is a basis, we can write any $v\in \R^d$ as
  \[ v = \sum_{i=1}^d c_i S^{-1}e_i.  \]
Since
\[
p_t(v) \approx \bigg(\sum_{i=1}^d |c_i|^2 \lambda_i^2 \bigg)^{(1-t)/2}
\bigg(\sum_{i=1}^d |c_i|^2 \mu_i^2 \bigg)^{t/2}
\]
we have $p_t(v) \gtrsim |c_i| \lambda_i^{1-t} \mu_i^t$ for any $v\in \R^d$ such that $|\langle v,S^* e_i \rangle | = |c_i|$. Thus, 
\begin{equation}\label{no2}
  p_t^*(w) \le \sum_{i=1}^d |b_i| p_t^*(S^* e_i)  \le \sum_{i=1}^d \frac{|b_i|}{\lambda_i^{1-t}\mu_i^t}.
    \end{equation}
Combining  \eqref{no1} and \eqref{no2}  yields 
\begin{equation}\label{no3}
p_t^*(w) \approx \bigg(\sum_{i=1}^d \bigg(\frac{|b_i|}{\lambda_i^{1-t}\mu_i^t}\bigg)^2\bigg)^{1/2} = |D_{A}^{-(1-t)/2}D_B^{-t/2} (S^*)^{-1}w|.
\end{equation}

 Define $q(w)= p_0^*(w)^{1-t}p_1^*(w)^t$.  By Lemma~\ref{lemma:norm-matrix},
 \[ p_0^*(w) \approx |D_A^{-1/2}(S^*)^{-1}w| \quad\text{and}\quad
  \qquad p_1^*(w) \approx
   |D_A^{-1/2}(S^*)^{-1}w|.  
 \]
 Hence, applying \eqref{no3} for $p_0^*$ and $p_1^*$ yields its dual analogue
\begin{equation}\label{no4}
q^*(v) \approx \bigg(\sum_{i=1}^d ({|c_i|}{\lambda_i^{1-t}\mu_i^t})^2\bigg)^{1/2} = |D_{A}^{(1-t)/2}D_B^{t/2} Sv|.
\end{equation}
Combining \eqref{no3} and \eqref{no4} with Lemma \ref{lemma:norm-matrix} yields for any $v\in \R^d$,
\begin{equation}\label{no5}
p_t^{**}(v) \approx |D_{A}^{(1-t)/2}D_B^{t/2} Sv| \approx q^*(v). \qedhere
\end{equation}
\end{proof}

Finally we connect the concept of a weighted geometric mean of norms with that of matrices \cite{bhatia}.

\begin{definition}\label{geme}
Let $A$ and $B$ be two symmetric positive definite matrices. For $0<t<1$ define the weighted geometric mean of $A$ and $B$ by
\[
A\#_tB = A^{1/2} (A^{-1/2} B A^{-1/2} )^t A^{1/2}.
\]
\end{definition}

Lemma \ref{wgm} gives an equivalent definition of the weighted geometric mean.

\begin{lemma}\label{wgm}
Let $A, B \in \Sd$. Suppose that for some
invertible matrix $S$, and diagonal matrices $D_A$ and $D_B$ we have
\[
A= S^*D_A S, \qquad B= S^* D_B S.
\]
Then, the weighted geometric mean of $A$ and $B$ satisfies 
\begin{equation}\label{wgm2}
 A\#_t B = S^* (D_A)^{1-t} (D_B)^t S.
\end{equation}
\end{lemma}

\begin{proof} We need to recall some useful facts about the set $\Sd$
  of symmetric positive $d\times d$ matrices from
  \cite[Chapter~6]{bhatia}.  The set $\Sd$ is an open subset of the
  space of all $d\times d$ symmetric matrices, which is equipped with
  the inner product $ \langle A,B \rangle = \operatorname{tr} A^*
  B$. Hence, $\Sd$ is a differentiable manifold equipped with a
  natural Riemannian metric. By \cite[Theorem 6.1.6]{bhatia}, there
  exists a unique geodesic path joining any two points $A,B \in \Sd$,
  which has a parametrization $A \#_t B$, $0\le t \le 1$. For each
  $d\times d$ invertible matrix $X$, define the congruence transformation
\[
\Gamma_X: \Sd \to \Sd, \qquad \Gamma_X(A)= X^* A X,\ A \in \Sd.
\]
By \cite[Lemma 6.1.1]{bhatia}, $\Gamma_X$ preserves lengths of
differentiable paths in $\Sd$. Hence, if the $\gamma: [0,1] \to \Sd$
is a geodesic path in $\Sd$, so is $\Gamma_X \circ \gamma$.

Let $\gamma(t)=A \#_t B$, $0\le t \le 1$, be a geodesic path between $A$ and $B$. Then,
\[
\Gamma_{S^{-1}} (\gamma(t)) = (S^{-1})^* A \#_t B S^{-1}
\]
is a geodesic path between the diagonal matrices $D_A$ and
$D_B$. Since the matrices $D_A$ and $D_B$ commute, by
\cite[Proposition 6.1.5 {\em et seq.}]{bhatia}, their geodesic path is given by
$t \mapsto (D_A)^{1-t}(D_B)^t$. Hence, since geodesic paths are
unique, we have
\[ (S^{-1})^* A \#_t B S^{-1} = (D_A)^{1-t}(D_B)^t, \qquad 0 \le t \le 1,
\]
which yields \eqref{wgm2}.
\end{proof}

As a consequence Proposition \ref{prop:geometric-double-dual} and Lemma \ref{wgm} we have the following corollary.

\begin{corollary}\label{dwgm} Suppose that $A,B \in \Sd$ and the norms $p_0$ and $p_1$ are given by 
\[
p_0(v)=|A^{1/2}v|\quad\text{and}\quad p_1(v)=|B^{1/2}v|
\qquad\text{for }v\in \R^d.
\]
Then the double dual of the weighted geometric mean $p_t$ \eqref{eqn:geo-mean} satisfies
\[
(p_t)^{**}(v) \approx |(A\#_t B)^{1/2} v| \qquad\text{for }v\in \R^d.
\]
\end{corollary}

\begin{proof}
This is an immediate consequence of \eqref{no5} and \eqref{wgm2} since
\[
|(A\#_t B)^{1/2} v|^2 = \langle (A\#_t B)v, v\rangle = \langle (D_A)^{1-t} (D_B)^t Sv, Sv \rangle = |D_{A}^{(1-t)/2}D_B^{t/2} Sv|^2.\qedhere
\]
\end{proof}

\section{Convex-set valued functions}
\label{section:convex-set-valued-functions}

\subsection*{Measurable convex-set valued functions}
In this section we develop the properties of measurable convex-set valued
functions.  Recall 
our standing assumption  that $(\Omega,\mathcal A,\mu)$ is a positive,
$\sigma$-finite, and complete measure space. We start with a definition of measurability of functions taking values in closed sets $\K(\R^d)$.

\begin{definition} \label{defn:measurable}
 Given a function $F: \Omega \to \K(\R^d)$, we say that $F$ is
 measurable if for every open set $U\subset \R^d$,  $F^{-1}(U)= \{x \in \Omega: F(x) \cap U \ne \emptyset\} \in \mathcal
 A$.
\end{definition}

We shall employ the following characterization of closed-set valued measurable
functions.  For a proof, see~\cite[Theorems~8.1.4,~8.3.1]{MR2458436}.
The equivalence of {\em (i)} and {\em (iv)} is known as the Castaing representation theorem.

\begin{theorem}\label{char}
Given $F: \Omega \to \K(\R^d)$, the following are equivalent:

\begin{enumerate}[(i)]
\item $F$ is measurable;
  
\item the graph of $F$, given by
\[
\operatorname{Graph}(F)=\{(x,v) \in \Omega \times \R^d: v \in F(x) \},
\]
belongs to the product $\sigma$-algebra $\mathcal A \otimes \mathcal B$, where $\mathcal B$ is the Borel $\sigma$-algebra of $\R^d$;

\item for any $v\in \R^d$, the distance map $x\mapsto d(v,F(x))$ is measurable;

\item there exists a sequence of measurable selection functions
  $f_k:\Omega \to \R^d$, $k \in \N$, of $F$ such that for all $x\in \Omega$,
\begin{equation}\label{char2}
F(x) = \overline{\{f_k(x): k \in \N\}}.
\end{equation}
\end{enumerate}
\end{theorem}

We will denote the set of all measurable selection functions for a
convex-set valued function $F$, that is all measurable functions $f$ such
that $f(x)\in F(x)$ a.e.,  by $S^0(\Omega, F)$. Note that by  
{\em (iv)}, this set is non-empty.

Measurability is preserved by taking the intersection or the convex
hull of the union of a sequence of measurable closed-set valued functions.
For a proof, see~\cite[Theorems~8.2.2,~8.2.4]{MR2458436}.

\begin{theorem}\label{cup}
Given a family $F_k: \Omega \to \K(\R^d)$, $k\in\N$, of measurable
maps,  the convex hull union map $G: \Omega \to \K(\R^d)$ defined by
\[
G(x) = \clconv\bigg(\bigcup_{k\in \N} F_k(x) \bigg)
\]
is measurable. Likewise, the intersection map $H: \Omega \to \K(\R^d)$ defined by
\[
H(x) = \bigcap_{k\in \N} F_k(x)
\]
is measurable.
\end{theorem}

As a consequence  of Theorem~\ref{cup} we can prove that the polar of a
measurable map with values in convex symmetric sets $\K_{cs}(\R^d)$ is again measurable.

\begin{theorem}\label{pol-measure}
  Given a measurable map $F: \Omega \to \K_{cs}(\R^d)$, the polar map
  $F^\circ: \Omega \to \K_{cs}(\R^d)$, defined by
  $F^\circ(x) = F(x)^\circ$, $x\in \Omega$, is also measurable.
\end{theorem}

\begin{proof}
By Theorem \ref{char}({\em iv}), there exists measurable selection functions
$f_k\in S^0(\Omega,F)$, $k \in \N$, such that \eqref{char2} holds.
For each $k\geq 1$, define $F_k: \Omega \to \K_{cs}(\R^d)$ by
\[ F_k(x)=\{v\in \R^d: |\langle v, f_k(x) \rangle| \le 1 \}. \]
Since the mapping from $\Omega \times \R^d$ to $\R$ given by
$(x,v) \mapsto \langle v, f_k(x) \rangle$ is measurable on the product
$\sigma$-algebra $\mathcal A \otimes \mathcal B$, the graph of $F_k$
is measurable, so by Theorem~\ref{char}({\em ii}) $F_k$ is a measurable
convex-set valued function.  But 
by Definition \ref{polar},
\[
F(x)^\circ = \bigcap_{k\in \N} F_k(x),
\]
so by Theorem \ref{cup}, $F^\circ $ is measurable as well.
\end{proof}

When a convex-set valued map $F: \Omega \to \K_b(\R^d)$ takes values in
compact sets, we have yet another equivalent definition of
measurability.  Recall that the
Hausdorff distance between two nonempty compact sets
$K_1,K_2 \subset \R^d$ is defined by
\begin{equation}\label{hd}
d(K_1,K_2)=\max\{ \sup_{v\in K_1} \inf_{w\in K_2} |v-w|, \sup_{v\in K_2} \inf_{w\in K_1} |v-w| \}.
\end{equation}
It is well-known that the collection of nonempty compact sets
$\K_b(\R^d)$ equipped with the Hausdorff distance is a complete and
separable metric space; see, for instance,~\cite[Theorem II.8]{CV}.
Both $\K_{bc}(\R^d)$ and $\bcs$ are closed subsets of $\K_b(\R^d)$:
see \cite[Theorem 1.8.5]{MR1216521}.  We can characterize the
measurability of compact-set valued mappings in terms of this
topology; this result is due to Castaing and Valadier \cite[Theorem
III.2]{CV}.

\begin{theorem}\label{CV}
  Given $F:\Omega \to \mathcal K_b(\R^d)$, $F$ is measurable in the
  sense of Definition~\ref{defn:measurable} if and only if $F$ is measurable as a
  function into $ \K_b(\R^d)$ with the Hausdorff topology. That is, if
  $U \subset \K_b(\R^d)$ is open in the Hausdorff topology, then $F^{-1}(U)$ is measurable.
\end{theorem}

We also need a characterization of the measurability of
functions taking values in the set of subspaces of $\R^d$.  

\begin{theorem}\label{BR}
  Let $F:\Omega \to \mathcal K(\R^d)$ be such that $F(x)$ is a
  (linear) subspace of $\R^d$ for all $x\in \Omega$. For each
  $x\in \Omega$, let $P(x)\in \Md$ be the matrix of the orthogonal
  projection of $\R^d$ onto $F(x)$. Then $F$ is measurable in the
  sense of  Definition~\ref{defn:measurable} if and only if the matrix-valued mapping
  $P: \Omega \to \Md$ is measurable.
\end{theorem}

Theorem~\ref{BR} is actually a special case of the theory of range
functions, which were introduced and studied by Helson in the context
of shift-invariant subspaces \cite{He1}. In general, a range function
takes values in the set of closed subspaces of a separable Hilbert
space.  A range function is defined to be measurable precisely if the
projection map is measurable.  Theorem~\ref{BR} can be proved using
results about multiplication-invariant spaces in \cite{BR} although
the assumption that $L^2(\Omega,\mu)$ is a separable Hilbert space is
needed. To avoid this extra assumption, we give a short direct
proof.

\begin{proof}
Suppose first that  $F$ is measurable in the sense of Definition~\ref{defn:measurable}.   Then there
  exists a sequence of measurable selection  functions $\{f_k\}_{k\in \N}$
such that  \eqref{char2} holds. For each $x\in \Omega$,  apply
  Gram-Schmidt orthogonalization to the vectors $\{f_k(x)\}_{k\in \N}$ to obtain a
  collection of orthogonal vectors $\{g_k(x)\}_{k\in \N}$ that span
  $F(x)$ and whose norms are either $0$ or $1$. Since each $g_k$ is a
  finite linear combination of the functions $f_k$, we have that each
  $g_k: \Omega \to \R^d$ is measurable.   The  orthogonal projection
  $P$ is given by
  $P(x)v=\sum_{k\in \N} \langle v,g_k(x)\rangle v$ for $v\in
  \R^d$, and so $P$ is a measurable matrix-valued function.

  Conversely, suppose the function 
  $P: \Omega \to \Md $ is measurable and takes its  values in the set of orthogonal
  projections.   Define a countable collection of measurable selection
  functions $f_q(x) = P(x)q$ which are indexed by $q\in\mathbb Q^d$.
  Since
  \[ F(x)=P(x)(\R^d) = \overline{ \{ P(x)q: q\in\mathbb Q^d \} },  \]
by Theorem~\ref{char}  the mapping $F$ is measurable in the sense of Definition~\ref{defn:measurable}.
\end{proof}

\medskip

It is well known that given a set $K \in \bcs$, there exists a unique
ellipsoid $E$ of maximal volume such that $E \subset K \subset
\sqrt{d}E$.  This is referred to as the John ellipsoid~\cite[Theorem
3.13]{TV}.    Given a convex-set valued function $F$ taking values in $\bcs$,
we can define an associated function $G$ such that $G(x)$ is the John
ellipsoid of $F(x)$.  It turns out that this mapping is measurable in
the sense of Definition~\ref{defn:measurable}:
see Lemma~\ref{je} below.  For our purposes, we state this result in a
slightly different form.  We note in passing that the measurability of
the John ellipsoid has been implicitly assumed in the literature; see,
for instance,~\cite[Proposition~1.2]{MR2015733}.

\begin{theorem}\label{mj}
  Suppose that $F : \Omega \rightarrow \bcs$ is measurable in the
  sense of Definition~\ref{defn:measurable}. Then
  there exists a measurable matrix-valued mapping $W: \Omega \to \Md$
  such that:
\begin{enumerate}[(i)]

\item the columns of the matrix $W(x)$ are mutually orthogonal; 

\item for all $x\in \Omega$,
  \[ W(x) \overline{{\mathbf B}} \subset F(x) \subset \sqrt{d} W(x) \overline{{\mathbf B}}.  \]
\end{enumerate}
\end{theorem}

The proof of Theorem \ref{mj} requires three lemmas.  First, let $\E$
be the set of all ellipsoids in $\R^d$ (possibly lower
dimensional):
\[
\E = \{P \overline{{\mathbf B}}: P \in \Md \} \subset \bcs.
\]
We note that by a compactness argument (see the proof of
Lemma~\ref{je} below) we have that $\mathcal E$
is a closed subset of $\mathcal K_b(\R^d)$ with respect to the
Hausdorff distance \eqref{hd}.

\begin{lemma}\label{je}
  Given a measurable convex-set valued function
  $F : \Omega \rightarrow \abcs$, there exists a measurable mapping
  $G: \Omega \rightarrow \abcs$ such that $G(x) \in \mathcal E$ and for
  all $x\in \Omega$,
\begin{equation}\label{je0}
G(x) \subset F(x) \subset \sqrt{d} G(x).
\end{equation}
\end{lemma}

\begin{proof}
For each $x\in \Omega$, define $G(x)$ to be the John ellipsoid; as
noted above, this is the unique ellipsoid of maximal volume that
satisfies~\eqref{je0}.  To complete the proof we only have to show
that $G: \Omega \to \abcs$ is measurable in the sense of
Definition~\ref{defn:measurable}.  

Let $P_1,P_2,\ldots$ be a dense collection of invertible matrices in
$\Md$. In particular, for any ellipsoid
$E=P \overline{{\mathbf B}} \in \mathcal E$ of positive volume and any
$\epsilon>0$ there exists $i\in \N$ such that
\begin{equation}\label{je2}
 P_i \overline{{\mathbf B}} \subset E \subset  (1+\epsilon) P_i \overline{{\mathbf B}}.
\end{equation}
We now define a sequence of measurable functions $G_i: \Omega \to
\mathcal E$, $i\in \N$,  by induction. Let
\[
G_1(x) = \begin{cases}
P_1 \overline{{\mathbf B}} & \text{if } P_1 \overline{{\mathbf B}}  \subset F(x),\\
\{0\} & \text{otherwise.}
\end{cases}
\]
For any invertible $P\in \Md$ we have that 
\[
\{x\in \Omega: P\overline{{\mathbf B}} \not \subset F(x) \} = \{x\in \Omega: F(x) \cap P(\R^d \setminus \overline{{\mathbf B}}) \ne \emptyset \} 
\]
is measurable by Definition~\ref{defn:measurable}.  Therefore, $G_1$
is a measurable function.  Suppose for some $i\geq 1$, we have defined
measurable functions $G_1, \ldots, G_i$.   Define
\[
G_{i+1}(x) = \begin{cases}
P_{i+1} \overline{{\mathbf B}} & \text{if } P_{i+1} \overline{{\mathbf B}} \subset F(x) \text{ and } \ld(P_{i+1}\overline{{\mathbf B}}) > \ld(G_{i}(x)),\\
G_i(x) & \text{otherwise.}
\end{cases}
\]
(Recall that $\ld$ denotes Lebesgue measure on $\R^d$.)
To show that $G_{i+1}$ is measurable, first note that the volume
functional $ K \mapsto \ld(K)$ is a continuous mapping of
$\mathcal K_b(\R^d)$ to $[0,\infty)$ \cite[Theorem
1.8.16]{MR1216521}. Hence, since $G_i: \Omega \to \bcs$ is measurable,
so is $\ld(G_i): \Omega \to [0,\infty)$ by Theorem \ref{CV}.  But then for any open set~$U$,
\begin{align*}
&   \{ x\in \Omega : G_{i+1}(x) \cap U \neq \emptyset \} \\
&  \qquad     =
  \big(\{ x\in \Omega : \ld(G_i(x)) \geq \ld(P_{i+1} \overline{{\mathbf B}}) \}
  \cap \{ x \in \Omega : G_i(x) \cap U \neq \emptyset \} \big)\\
&  \qquad  \qquad \cup
  \big(\{ x\in \Omega : \ld(G_i(x)) < \ld(P_{i+1} \overline{{\mathbf B}}) \}\\
& \qquad \qquad \qquad   \cap \{ x \in \Omega : P_{i+1}\overline{\mathbf{B}} \subset F(x) \text{ and }
  P_{i+1}\overline{\mathbf{B}} \cap U \neq \emptyset \}\big).
\end{align*}
Since each set on the right-hand side is measurable, we conclude that
$G_{i+1}$ is a measurable function.

To complete the proof, we need to show that $G_i(x)$ converges to
$G(x)$ in the Hausdorff distance~\eqref{hd}.  Because then $G$ is a
measurable function with respect to the Hausdorff topology, and so by
Theorem~\ref{CV} is measurable in the sense of
Definition~\ref{defn:measurable}.  We will prove this by
contradiction.  Fix $x\in \Omega$ and suppose to the contrary that
$G_i(x)$ does not converge to $G(x)$.  Since $G(x)$ is the maximal
ellipsoid contained in $F(x)$, by \eqref{je2} and the
definition of the $G_i$'s we have that as $i\rightarrow \infty$, 
\[
\ld(G_i(x)) \to \ld(G(x)).
\]
By the Blaschke selection theorem \cite[Theorem
1.8.6]{MR1216521},  $\{E\in \mathcal E: E \subset F(x)\}$ is a compact subset of
$\mathcal K_b(\R^d)$. Hence, some subsequence $G_{i_j}(x)$ converges as
$j\to \infty$ to an ellipsoid $E' \in\mathcal E$, and by assumption
$E'\neq G(x)$.  But we have $\ld(E')=\ld(G(x))$, and this contradicts the
fact that the John ellipsoid $G(x)$ is unique.  Thus $G_i(x)\rightarrow
G(x)$ and our proof is complete.
\end{proof}

\begin{lemma}\label{jem}
  Let $F: \Omega \rightarrow \mathcal K_b(\R^d)$ be a measurable
  mapping. Then there exists a measurable mapping $v: \Omega \to \R^d$
  such that for all $x\in \Omega$
\[
v(x) \in F(x) \qquad\text{and}\qquad |v(x)|= \sup\{ |v|: v\in F(x) \}.
\]
\end{lemma}

\begin{proof}
  Let $\{f_k\}_{k\in \N}$ be a sequence of measurable selection functions such
  that \eqref{char2} holds.   Define 
  $g_0: \Omega \to [0,\infty)$ by
  \[ g_0(x) = \sup \{ |v| : v \in F(x) \}
    = \sup  \{  |f_k(x)| : k\in \N \}.  \]
  Then $g_0$ is measurable.  Now define  $F_0: \Omega \rightarrow \mathcal K_b(\R^d)$ by
\[
F_0(x) = \{v \in F(x): |v| = g_0(x) \} = F(x) \cap g_0(x) \mathbf S,
\]
where $\mathbf S = \{u\in \R^d: |u|=1\}$. Then by Theorem \ref{cup},
$F_0$ is measurable since $F$ and $g_0\mathbf{S}$ are.

We  now show that we can choose $v(x)$ from $F_0(x)$ in such a way
that $v(x)$ is a measurable function.  We  do this iteratively by
choosing the vectors $v$ that are maximal in each coordinate.
For $v=(v_1,\ldots,v_d) \in \R^d$, let $P_1(v)=v_1$ be the projection
onto the first coordinate.  Since $P_1$ is continuous, if we define
$g_1: \Omega \to [0,\infty)$  by
\[ g_1(x)= \sup\{ v_1: (v_1,\ldots,v_d)\in F(x) \} = \sup_k
  P_1(f_k(x)), \]
then $g_1$ is measurable.   
 Define
$F_1: \Omega \rightarrow \mathcal K_b(\R^d)$ by
\[
F_1(x) = \{(v_1,\ldots,v_d) \in F_0(x): v_1 = g_1(x) \} = F(x) \cap (\{g_1(x)\}\times \R^{d-1}).
\]
By Theorem \ref{cup}, $F_1$ is measurable.

We repeat this argument:  by induction, for each $i\geq 1$ we  define
$F_{i+1}: \Omega \rightarrow \mathcal K_b(\R^d)$ such that
$F_{i+1}(x)$ consists of all points in $F_i(x)$ that have a maximal
$i+1$ coordinate.  Here we mean maximal in norm:  in $g_1$ we fixed
$v_1$ to be positive in order to be explicit, but in the subsequent steps the maximal coordinate could be
negative.  After $d$ steps this yields a measurable function
$F_d: \Omega \to \K_b(\R^d)$.  By the maximality of each coordinate
we must have that  $F_d(x)$ is a singleton: i.e., $F_d(x)=\{v(x)\}$
for some measurable  function $v: \Omega \to \R^d$.
\end{proof}

\begin{lemma}\label{jes}
  Let $G: \Omega \rightarrow \bcs$ be a measurable mapping such that
  $G(x) \in \mathcal E$ for all $x\in \Omega$. Then there exists a
  measurable mapping $W: \Omega \to \Md$ such that:
\begin{enumerate}[(i)]
\item
the columns of the matrix $W(x)$ are mutually orthogonal;
\item 
$G(x)=W(x) \overline{{\mathbf B}}$ for all $x\in \Omega$.
\end{enumerate}
\end{lemma}

\begin{proof}
  We will construct the columns $v_1,\ldots v_d: \Omega \to \R^d$ of $W$
  inductively. Let $v_1$ be the vector-valued function given by Lemma \ref{jem} corresponding
  to $G$. Define the
  mapping $J_1: \Omega \to \K(\R^d)$ by
  $J_1(x)= \operatorname{span}\{v_1(x)\}$.  By Theorem~\ref{char} it
  is measurable, since the collection of linear multiples of $v_1(x)$
  by rational numbers forms a countable collection of measurable
  selection functions.
Then by Theorem
  \ref{BR}, $J_1$ is a measurable
  range function: i.e., the associated projection matrix is a
  measurable function. Hence, so is orthogonal projection, and thus the range function
  $J_1^\perp: \Omega \to \K(\R^d)$, defined as the orthogonal complement
  $J_1^\perp(x) = (J_1(x))^\perp$, is measurable.

  We now proceed by induction.  If for some $i\geq 1$ we have defined
  measurable, 
  vector-valued functions $v_1,\ldots, v_i$, then we can define a
  mapping 
  \[J_i: \Omega \to \K(\R^d),
  \qquad J_i(x)= \operatorname{span}\{v_1(x),\ldots,v_i(x)\}.
  \]
  The map
  $J_i$ is measurable by Theorem~\ref{char} since linear combinations
  of the vectors $v_1,\ldots,v_i$ with rational coefficients form a
  countable family of measurable selection functions.
Define
the measurable mapping  $G_{i}:\Omega \to \bcs$ by
\[
G_{i}(x) = G(x) \cap J_i(x)^\perp.
\]
We can now apply Lemma \ref{jem} to get a vector-valued function
$v_{i+1}$ which is orthogonal to $v_1,\ldots,v_i$. For every
$x\in \Omega$ the vectors $v_1(x),\ldots, v_d(x)$ define semi-axes of an
ellipsoid; since at every step we chose $v_i$ to be maximal, they are
given in decreasing order and the
ellipsoid must equal $G(x)$.   Hence, if $W(x)$ is the $d\times d$
matrix with columns $v_i$, then
we have $G(x) = W(x) \overline{{\mathbf B}}$.
\end{proof}

\begin{proof}[Proof of Theorem \ref{mj}]
If the function $F$ in our hypothesis is absorbing for all $x$, then
the desired conclusion follows from Lemmas~\ref{je} and~\ref{jes}.
Since this need not be the case, we need to consider the ``dimension''
of $F$ at each point.  More precisely, we argue as follows.  Given  an arbitrary measurable
mapping $F: \Omega \rightarrow \bcs$, define the new function
\[
  J(x) = \operatorname{span} F(x) = \bigcup_{r>0} r F(x) = \clconv\bigg(
  \bigcup_{\substack{r>0 \\ r\in \Q}} rF(x) \bigg);
\]
then by Theorem~\ref{cup}, $J$ is measurable.  
For $k=0,\ldots,d$ define the sets
\[
\Omega_k = \{x\in \Omega: \dim J(x) = k \}.
\]
Equivalently, if we let $P$ be the measurable projection matrix in
Theorem~\ref{BR}, then $\Omega_k$ is the set where $P(x)$ has rank
$k$. 
Since the
rank can be computed by taking the determinant of all the $k\times k$
minors, it is a measurable mapping, and so $\Omega_k$ is a measurable
set. 
Therefore, to complete the proof it will  suffice to show the
conclusion for each restriction $F|_{\Omega_k}$, $k=1,\ldots,d$.

Fix $k$.  By \cite[Theorem 2 in Section 1.3]{He2} we can find
measurable functions $w_1, \ldots,w_k: \Omega_k \to \R^d$ such that
$w_1(x),\ldots,w_k(x)$ form an orthonormal basis of $J(x)$ for
$x\in \Omega_k$. This follows from the Gram-Schmidt process as in the
proof of Theorem~\ref{BR}.  Denote the collection of $s\times t$
matrices by $\M_{s\times t}$.  Let $M_k(x) \in \M_{d\times k}$ be the matrix
whose columns are the vectors $w_1(x),\ldots,w_k(x)$. Then $M_k(x)$ is
an isometry of $\R^k$ onto $J(x)$ and the transpose
$M_k^*(x) \in \M_{k\times d}$ is its inverse.
Consequently, $F_k: \Omega_k \to \K_{bcs}(\R^k)$, defined by
$F_k(x)=M_k^*(x)F(x)$, $x\in\Omega_k$, is a measurable convex-set valued
mapping such that $F_k(x)$ is absorbing. Therefore, we can apply Lemmas~\ref{je} and~\ref{jes} to get a measurable mapping
$W_k: \Omega_k \to \M_{k}$ such that the columns of
$W_k(x)$ are mutually orthogonal and $F_k(x)=W_k(x) \overline{\mathbf B}_k$,
where $\overline{\mathbf B}_k$ is the closed unit ball in $\R^k$.

Finally, define
$W(x)=M_k(x)\circ W_k(x)\circ P_k \in \M_{d}$ for
$x\in \Omega_k$, where $P_k \in \Md$ is the coordinate
projection of $\R^d$ onto $\R^k$. Then the columns of $W(x)$ are
orthogonal and
\[
  W(x)\overline{\mathbf B} = (M_k(x)\circ W_k(x))\overline{\mathbf
    B}_k = M_k(x) F_k(x)=F(x) \qquad x\in \Omega_k.
\]
This defines the required mapping $W: \Omega_k \to \M_{d}$;
combining these functions we get the desired mapping on $\Omega$.
\end{proof}

\subsection*{Integrals of convex-set valued maps}
In this section we define the integral of convex-set valued functions
using the Aumann integral.  We follow the treatment given in~
\cite[Section 8.6]{MR2458436}.  As before, the underlying measure
space is $(\Omega,\A,\mu)$.

\begin{definition}\label{au}
 Suppose $F: \Omega \to \K(\R^d)$ is a measurable map. Define the {\it
   set of all integrable selection functions} of $F$ by
\[
S^1(\Omega,F)  = \{ f\in L^1(\Omega,\R^d): f \in  S^0(\Omega,F) \}.
\]
The Aumann integral of $F$ is the set of integrals of integrable
selection functions of $F$, i.e.,
\[
  \int _\Omega F\, d\mu = \bigg\{ \int_\Omega f \,d\mu: f\in S^1(\Omega,F)
  \bigg\}.
\]
\end{definition}

{\em A priori} a measurable map $F$ may not have any integrable
selection functions.  We therefore introduce a class of maps for which
this set is non-empty. 

\begin{definition}
  A measurable closed-set valued function $F: \Omega \to \K(\R^d)$ is 
    integrably bounded if there exists a non-negative function
  $k\in L^1(\Omega,\R)$ such that
\begin{equation}\label{au3}
F(x) \subset k(x)\mathbf B \qquad\text{for a.e. }x\in\Omega.
\end{equation}
If $\Omega$ is a metric space (in particular if $\Omega=\R^n$) we say
$F$ is locally integrably bounded if this holds for $k\in
L^1_{loc}(\Omega,\R)$. 
\end{definition}

Below, we will want to treat the integral of a vector-valued function
as the integral of a convex-set valued function.  We will be able to
do this using the following lemma.

\begin{lemma} \label{lemma:vector-int}
Let $f \in L^1(\Omega ,\R^d)$. Then, the convex-set
  valued map
\begin{equation}\label{ints1}
F(x) = \conv\{ f(x), - f(x) \}, \qquad x\in \Omega,
\end{equation}
is measurable and integrably bounded. Moreover, its Aumann integral satisfies
\begin{equation}\label{ints2}
\int_\Omega F d\mu = \bigg\{ \int_\Omega k f d\mu: k \in L^\infty(\Omega; \R), ||k||_\infty \le 1 \bigg\}.
\end{equation}
\end{lemma}

\begin{proof}
Let $\{\alpha_i\}_{i\in\N}$ be a dense subset of the interval $[-1,1]$. Then,
\[
F(x) = \overline{\{ \alpha_i f(x): i\in \N\}}.
\]
Hence, by Theorem \ref{char} (see also \cite[Theorem
8.2.2]{MR2458436}) $F$ is measurable as a convex-set valued
mapping. Moreover, it is clear that if $f\in S^1(\Omega, F)$, then it 
must be of the form $g(x)=k(x) f(x)$, where $|k(x)|\le 1.$ Hence,
\eqref{ints2} follows from the definition of the Aumann integral.
\end{proof}

When the measure $\mu$ is non-atomic, the integral of any closed-set valued
map $F$ is convex, even when the values of $F$ are not necessarily
convex.  For proof of this highly non-trivial result,
see~\cite[Theorem 8.6.3]{MR2458436}.

\begin{theorem}\label{ci}
Suppose that the measure $\mu$ is nonatomic.
Given a measurable mapping  $F: \Omega \to \K(\R^d)$, let
$K=\int_\Omega F\, d\mu$ be the Aumann integral of $F$.  Then $K$ is
a convex, though not necessarily closed, subset of $\R^d$. In addition, if
$F$ is integrably bounded, then $K \subset  \K_{bc}$.
\end{theorem}

In this paper we are primarily interested in convex-set valued
mappings. In this case, the assumption that $\mu$ is nonatomic can be
dropped and we have the following result.

\begin{theorem}\label{cci}
  Given a measurable mapping $F: \Omega \to \bcs$, let
  $K=\int_\Omega F \,d\mu$ be the Aumann integral of $F$. Then $K$ is a
  convex, symmetric set in $\R^d$, and so
  $\overline{K}\in\mathcal K_{cs}(\R^d)$. In addition, if $F$ is
  integrably bounded, then $K=\overline{K}\in \bcs$.
\end{theorem}

\begin{proof}
  This result follows from the corresponding properties of the
  integrable selection functions.  Since $F(x) \in \bcs$ for all $x\in \Omega$, if
  $f,\,g\in S^1(\Omega,F)$, then $-f \in S^1(\Omega,F)$ and for
  $0<\lambda <1$, $\lambda f+(1-\lambda)g\in S^1(\Omega,F)$.  Hence, it
  follows from Definition~\ref{au} that $K$ is a convex, symmetric set
  in $\R^d$.

Now suppose that $F$ is integrably bounded by $k\in L^1(\Omega)$.   If
$f\in S^1(\Omega,F)$, then $f(x)\in F(x)$, so $|f(x)|\leq k(x)$ a.e.
In particular,
\[ \bigg|\int_\Omega f(x)\,d\mu \bigg| \leq \int_\Omega k(x)\,dx, \]
and so $K$ is bounded.

Finally, to show that $K$ is closed, first note that $S^1(\Omega,F)$
is closed in $L^1(\Omega)$.  For if $f\in L^1(\Omega)$ is a limit
point, there exists a sequence $\{f_n\}_{n\in\N}$ that converges to $f$ in $L^1$
and (by passing to a subsequence) pointwise almost everywhere.  Since
$f_n(x)\in F(x)$ and $F(x)$ is closed, $f(x)\in F(x)$, so $f\in
S^1(\Omega,F)$.

Second, given any $\epsilon>0$, there exists $\delta>0$ such that if
$E\subset \Omega$ satisfies $|E|<\delta$, then
$\int_E k\,d\mu <\epsilon$.  Therefore, if we replace $\Omega$ by $E$
in the above argument, then we have that that $S^1(\Omega,F)$ is
equi-integrable.  Hence, by the Dunford-Pettis theorem (see
\cite[Theorem~IV.8.9,  p.~292]{MR1009162} or \cite[Theorem III.C.12]{Woj}), $S^1(\Omega,F)$ is weakly compact in
$L^1(\Omega)$.  Let $v$ be a limit point of $K$.  Then there exists a sequence
$\{f_n\}_{n\in\N}$ in $S^1(\Omega,F)$ such that
$\int_\Omega f_n\,d\mu \rightarrow v$ as $n\rightarrow \infty$.  By weak compactness, if we pass to a
subsequence, there exists $f\in L^1(\Omega)$ such that
$f_n\rightarrow f$ weakly; in particular, $\int_\Omega f\,d\mu = v$.
Moreover, by Mazur's lemma~\cite[Corollary~3.8]{MR2759829}, there
exists a sequence $\{g_k\}_{k\in \N}$, where each $g_k$ is a convex combination
of the functions $f_n$, that converges to $f$ in $L^1$ norm.  However,
as we noted above, $S^1(\Omega,F)$ is convex, and so each $g_k$ is
contained in it.  Therefore, since $S^1(\Omega,F)$ is closed,
$f\in S^1(\Omega,F)$, and so $v\in K$.  Thus $K$ is closed, which
completes our proof.
\end{proof}
%


We now show that integrable selection functions are additive.
Our proof is adapted from~\cite[Theorem~1.4]{MR0507504}.

\begin{theorem} \label{prop:selection-additive}
Suppose that  $F_i : \Omega \rightarrow \bcs$, $i=1,\,2$, are measurable and
integrably bounded.  Then
\[ S^1(\Omega, F_1+F_2)= S^1(\Omega, F_1)+S^1(\Omega,F_2). \]
\end{theorem}

\begin{proof}
One direction is immediate:  if $f_i\in S^1(\Omega,F_i)$, $i=1,\,2$, then $f_1+f_2\in S^1(\Omega,F_1+F_2)$.  To prove the converse, note
that by Theorem~\ref{char}, we have sequences of selection functions
$\{f_k\}_{k\in\N} \subset S^1(\Omega, F_1)$, $\{g_j\}_{j\in\N} \subset
S^1(\Omega, F_2)$ such that for each $x\in \Omega$,
\[ F_1(x) + F_2(x) = \overline{\{f_k(x)+g_j(x): j,k \in \N\}}. \]
Therefore, if we fix $f \in S^1(\Omega,F_1+F_2)$, there exist
sequences of measurable functions $\{h_n\}_{n\in\N}$ and $\{k_n\}_{n\in\N}$ such that
for almost every $x\in \Omega$ and $n\in \N$,
$h_n(x) \in \{f_1(x),\ldots,f_n(x)\}$,
$k_n(x) \in \{g_1(x),\ldots,g_n(x)\}$, and $h_n(x) + k_n(x) \to f(x)$
as $n\to \infty$.   The construction of these functions follows the
argument in \cite[Lemma~1.3]{MR0507504} with $p=1$, which yields a
sequence that converges in norm; by passing to a subsequence we get a
sequence that converges pointwise almost everywhere.  

We now argue as in the proof of
Theorem~\ref{cci}. There we showed that
$S^1(\Omega,F_i)$, $i=1,2$, is closed and weakly compact subset of
$L^1(\Omega)$. Hence, by passing to a subsequence, there exists
$h , k \in L^1(\Omega)$ such that $h_n \to h$ and $k_n \to k$ weakly
in $L^1(\Omega)$ as $n\to \infty$. By Mazur's lemma we can replace
the functions $h_n$ and $k_n$ by convex combinations of them to get
sequences that converge in $L^1(\Omega)$ norm. Since the sets $S^1(\Omega,F_i)$,
$i=1,2$, are convex and closed, we have $h\in S^1(\Omega,F_1)$ and
$k\in S^1(\Omega, F_2)$.  Therefore,
$f = h+ k \in S^1(\Omega, F_1)+S^1(\Omega,F_2)$ and this completes our
proof.
\end{proof}

As a consequence of Theorem~\ref{prop:selection-additive} we can prove that the Aumann
integral is linear and monotonic.  

\begin{theorem}\label{add}
 Suppose that $F_i: \Omega \to \bcs$, $i=1,2$, are measurable and integrably bounded. Then, for any $\alpha_i \in \R $, $i=1,2$, we have
\[
\int_{\Omega} (\alpha_1 F_1+\alpha_2 F_2)\, d\mu = 
\alpha_1 \int_{\Omega} F_1 \,d\mu +\alpha_2 \int_\Omega F_2 \,d\mu.
\]
Moreover, if $F_1(x) \subset F_2(x)$ for all $x\in \Omega$, then
\[ \int_\Omega F_1\,d\mu \subset \int_\Omega F_2\,d\mu.  \]
\end{theorem}

\begin{proof}
The monotonicity of the Aumann integral follows at once from
Definition~\ref{au}  and from the fact
that if $F_1\subset F_2$, then $S^1(\Omega, F_1) \subset
S^1(\Omega,F_2)$. 
  
To show that it is linear, note first that it is immediate from Definition~\ref{au} that 
\[ \int_{\Omega} \alpha_i F_i \,d\mu = \alpha_i \int_\Omega F_i
  \,d\mu.  \]
Finally, by Theorem~\ref{prop:selection-additive} we have that
\begin{multline*}
\int_{\Omega} F_1+ F_2\, d\mu 
= \bigg\{ \int_{\Omega} f_1+f_2\,d\mu : f_i \in S^1(\Omega,F_i), i=1,2
\bigg\} 
=
\int_{\Omega} F_1  \,d\mu + \int_{\Omega}
  F_2  \,d\mu. 
\end{multline*}
\end{proof}

\begin{corollary} \label{cor:integral-additive}
  Given a locally integrably bounded function $F  :\R^n \rightarrow
  \bcs$ and bounded sets $A\subset B$,
  \[ \int_ AF(x)\,dx \subset \int_B F(x) \,dx. \]
\end{corollary}

\begin{proof}
Since $F(x)\chi_A(x) \subset F(x)\chi_B(x)$, this follows at once from Theorem~\ref{add}.
\end{proof}

Finally, we prove versions of H\"older's inequality and Minkowski's
inequality for the Aumann integral.    The proof
requires one lemma.

\begin{lemma} \label{lemma:zero-int}
  Given $F : \Omega \rightarrow \bcs$ measurable,
  \begin{equation}\label{zint}
  \int_\Omega  F(x)\,d\mu = \{0 \} 
  \end{equation}
  if and only if $F(x)=\{0\}$ a.e.
\end{lemma}

\begin{proof}
Suppose that \eqref{zint} holds. Let $f\in S^1(\Omega,F)$. 
Take any $v\in \R^d$.
 Then 
 \[f^+=f\chi_{\{ x\in \Omega :
   \langle f(x), v \rangle \geq 0\}}
   \quad\text{and}\quad 
   f^+=f\chi_{\{ x\in \Omega :
  \langle  f(x),v \rangle < 0\}}
  \]
   are also selection functions since $0\in F(x)$.
  By Definition~\ref{au} we deduce that
  \[
   \int_\Omega  \langle f^+(x), v \rangle \,d\mu =
   \int_\Omega  \langle f^-(x), v \rangle \,d\mu =0.
   \]
  Hence, $ \langle f^+(x), v \rangle =  \langle f^-(x), v \rangle=0$
  for a.e. $x\in \Omega$, and hence  $ \langle f(x), v \rangle = 0$
  for a.e. $x\in \Omega$. Since $v\in \R^d$ is arbitrary we have that
  all the integrable selection functions of $F$ are trivial. If $f$ is an arbitrary selection of $F$, then there exists a strictly positive function $k$ on $\Omega$ such that $kf \in S^1(\Omega,F)$. By the previous argument $k(x)f(x)=0$ a.e. Hence, Theorem~\ref{char} yields that $F(x)=\{0\}$ a.e. The converse implication is trivial.
\end{proof}

\begin{prop} \label{prop:convex-Holder}
 Let $\rho$ be a norm on $\R^d$ and fix $1<p<\infty$.  Suppose $H :
 \Omega \rightarrow \bcs$  and 
 $f,\,g : \Omega \rightarrow [0,\infty)$ are measurable, and $f^pH$
 and $g^{p'}H$ are integrably bounded.  
Then
 \[ \rho\bigg(\int_\Omega f(x)g(x)H(x)\,d\mu \bigg)
   \leq \rho\bigg(\int_\Omega f(x)^pH(x)\,d\mu
   \bigg)^{\frac{1}{p}}
   \rho\bigg(\int_\Omega g(x)^{p'}H(x)\,d\mu
   \bigg)^{\frac{1}{p'}}. \]
\end{prop}

\begin{proof}
  The proof is an adaptation of the standard proof of H\"older's inequality for scalar functions.  Since $f^pH$
 and $g^{p'}H$ are integrably bounded, the integrals on the righthand
 side are finite.  If either is equal to $0$, then by
 Lemma~\ref{lemma:zero-int}, either $fH=\{0\}$ or $gH=\{0\}$ a.e., so
 the lefthand side is $0$ as well.  Therefore, since the desired
  inequality is homogeneous, we
  may assume without loss of generality that
  \[ \rho\bigg(\int_\Omega f(x)^pH(x)\,d\mu
   \bigg) = 
   \rho\bigg(\int_\Omega g(x)^{p'}H(x)\,d\mu
   \bigg) = 1. \]
 Then by Young's inequality and Theorem~\ref{add},
 \begin{multline*}
   \rho\bigg(\int_\Omega f(x)g(x)H(x)\,d\mu \bigg)
   \leq
   \rho\bigg(\int_\Omega \frac{1}{p}f(x)^pH(x)\,d\mu
   + \int_\Omega \frac{1}{p'}g(x)^{p'}H(x)\,d\mu
   \bigg) \\
   \leq
   \frac{1}{p}\rho\bigg(\int_\Omega f(x)^pH(x)\,d\mu
   \bigg) + 
   \frac{1}{p'}\rho\bigg(\int_\Omega g(x)^{p'}H(x)\,d\mu
   \bigg)
   = 1.\qedhere
 \end{multline*}
\end{proof}

\begin{prop} \label{prop:convex-minkowski}
Let $\rho$ be a norm on $\R^d$ and fix $1<p<\infty$.  Suppose $H :
 \Omega \rightarrow \bcs$  and 
 $f,\,g : \Omega \rightarrow [0,\infty)$ are measurable, and $f^pH$
 and $g^{p}H$ are integrably bounded.    Then
  \[ \rho\bigg(\int_\Omega [f(x)+g(x)]^pH(x)\,d\mu \bigg)^{\frac{1}{p}}
   \leq \rho\bigg(\int_\Omega f(x)^pH(x)\,d\mu
   \bigg)^{\frac{1}{p}} 
   + \rho\bigg(\int_\Omega g(x)^{p}H(x)\,d\mu
   \bigg)^{\frac{1}{p}}. \]
\end{prop}

\begin{proof}
  The proof is again an adaptation of the standard proof of Minkowski's inequality for scalar
  functions. First note that since
  \[ [f+g]^pH \subset 2^{p-1}(f^p+g^p)H = 2^{p-1}f^pH+2^{p-1}g^pH, \]
  $[f+g]^pH$ is integrably bounded, so the left hand side of the
  inequality is finite.  We may also assume without generality that it
  is positive since otherwise there is nothing to prove.  But then,
  by Proposition~\ref{add},
  \begin{align*}
    & \rho\bigg(\int_\Omega [f(x)+g(x)]^pH(x)\,d\mu
    \bigg) \\
    & \qquad \quad = \rho\bigg(\int_\Omega f(x)[f(x)+g(x)]^{p-1} H(x)\,d\mu
      + \int_\Omega g(x)[f(x)+g(x)]^{p-1}
      H(x)\,d\mu\bigg)\\
    & \qquad \quad \leq  \rho\bigg(\int_\Omega f(x)[f(x)+g(x)]^{p-1}
      H(x)\,d\mu\bigg)
      + \rho\bigg(\int_\Omega g(x)[f(x)+g(x)]^{p-1}
      H(x)\,d\mu\bigg) \\
     & \qquad \quad \leq  \rho\bigg(\int_\Omega f(x)^pH(x)\,d\mu
       \bigg)^{\frac{1}{p}}\rho\bigg(\int_\Omega [f(x)+g(x)]^pH(x)\,d\mu
       \bigg)^{\frac{1}{p'}} \\
    & \qquad \qquad \qquad +
      \rho\bigg(\int_\Omega g(x)^pH(x)\,d\mu
       \bigg)^{\frac{1}{p}}\rho\bigg(\int_\Omega [f(x)+g(x)]^pH(x)\,d\mu
       \bigg)^{\frac{1}{p'}}.
  \end{align*}
The last step follows from  Proposition \ref{prop:convex-Holder} since $(p-1)p'=p$.
  The desired inequality now follows immediately.
\end{proof}
\section{Seminorm functions}
\label{section:seminorm}

In this section we introduce seminorm functions, which we will use
below to
define $L^p$ spaces of convex-set valued functions.   We will show an
equivalence between the Aumann integral of such a function and the
seminorm associated with the convex bodies.    We begin with a definition.

\begin{definition}\label{normf}
A seminorm function $\rho$ on $\Omega$ is a mapping $\rho: \Omega \times \R^d \to [0,\infty)$ such that:
\begin{enumerate}[(i)]
\item $x\mapsto \rho_x(v)=\rho(x,v)$ is a
  measurable function for any $v\in \R^d$, 
\item for all $x\in \Omega$, $\rho_x(\cdot)$ is a seminorm on $\R^d$.
\end{enumerate}
\end{definition}

Our first result shows that there is a one-to-one correspondence
between seminorm functions and measurable convex-set valued maps
$\Omega \to \bcs$. For a variant of Theorem \ref{mc} for bounded
convex-set valued mappings into separable Banach space, see \cite[Theorem
8.2.14]{MR2458436}, and a version for compact convex-set valued mappings
into a locally convex, metrizable, separable space, see \cite[Theorem
III.15]{CV}.

\begin{theorem}\label{mc} Suppose that $\rho: \Omega \times \R^d \to
  [0,\infty)$ is a seminorm function. Then the convex-set valued mapping
  $F: \Omega \to \bcs$ defined for each $x\in \Omega$ by 
\begin{equation}\label{mc0}
F(x)=\{v\in \R^d: \rho_x(v) \le 1\}^\circ
\end{equation}
is measurable. Conversely, given a measurable mapping $F: \Omega \to
\bcs$, define a function $\rho: \Omega \times \R^d \to [0,\infty)$ by
\begin{equation}\label{mc1}
\rho_x(v)=p_{F(x)^\circ}(v)
\qquad (x,v) \in \Omega \times \R^d.
\end{equation}
Then $\rho$ is a seminorm function. Moreover, the correspondence
between seminorm functions $\rho$ and convex-set valued mappings $F$ is
one-to-one.
\end{theorem}

\begin{proof}
  Suppose first that $\rho: \Omega \times \R^d \to [0,\infty)$ is a seminorm
  function. To show that $F$ is measurable, we will first prove that
  $\rho$ satisfies a stronger version of ({\em i}) in Definition \ref{normf}:
\begin{enumerate}
\item[({\em i-a})]
$\rho: \Omega \times \R^d \to [0,\infty)$ is measurable with respect
to the product $\sigma$-algebra $\mathcal A \otimes \mathcal B$, where
$\mathcal B $ is the Borel $\sigma$-algebra on $\R^d$.
\end{enumerate}
We could derive this from the fact that $\rho$ is a Carath\'eodory
map, using \cite[Lemma 8.2.6]{MR2458436}.  Instead, however, we will
give a direct proof of ({\em i-a}).
Let $\mathcal D = \{v_i\}_{i \in \N}$ be a countable dense subset of
$\R^d$. For any functional $l\in (\R^d)^*$, the set
\begin{equation}\label{mc2}
A_l= \{x\in \Omega: |l(v)| \le \rho_x(v) \quad\text{ for all }v\in \R^d\}
\end{equation}
is a measurable subset of $\Omega$. To see this, note that $A_l$ can
be written as a countable intersection of measurable sets in $\mathcal A$:
\[
A_l = \bigcap_{l=1}^\infty \{ x\in \Omega: |l(v_i)| \le \rho_x(v_i) \}.
\]

Now let $\mathcal D' = \{l_i\}_{i\in \N}$ be a countable dense subset
of functionals on $\R^d$ given by $l_i(v) = \langle v, v_i \rangle$,
$v\in \R^d$. For any norm $p$ on $\R^d$ we claim that
\begin{equation}\label{mc4}
p(v) = \sup \{ |l_i(v)|: i\in \N \text{ is such that }|l_i(w)| \le p(w) \text{ for all }w\in \R^d\}.
\end{equation}
To see this, fix $v\in \R^d$ and let $E=\{\alpha v : \alpha \in \R\}$
  be the subspace generated by $v$.  Define the linear functional
  $\lambda$ on $E$ by $\lambda(\alpha v)=\alpha p(v)$.  Then by the
  Hahn-Banach theorem, $\lambda$ extends to an element of $(\R^d)^*$
  such that $|\lambda(w)|\leq p(w)$ for all $w\in \R^d$.  The
  identity~\eqref{mc4} now follows from the density of $\mathcal D'$. 

  To prove ({\em i-a}), assume for the moment that for all $x\in \Omega$,
  $\rho_x(\cdot)$ is a norm on $\R^d$. Then, if we combine \eqref{mc2} and
  \eqref{mc4} we get that
\begin{equation}\label{mc6}
  \rho_x(v) = \sup_{i\in \N} |l_i(v)| \chi_{A_{l_i}}(x)
  \qquad\text{for all }(x,v) \in \Omega \times \R^d.
\end{equation}
Hence, $\rho$ is measurable with respect to the $\sigma$-algebra
$\mathcal A \otimes \mathcal B$. Finally, if $\rho$ is an arbitrary
seminorm function, define the sequence of norm
functions
\[
\rho^i(x,v) = \rho_x(v) + \tfrac{1}{i} |v| \qquad (x,v) \in \Omega \times \R^d.
\]
Each $\rho^i$ is measurable with respect to the $\sigma$-algebra
$\mathcal A \otimes \mathcal B$, and so their limit $\rho$ is
measurable as well. This proves ({\em i-a}).

We can now prove that $F: \Omega \to \bcs$  defined by \eqref{mc0} is
measurable.   By ({\em i-a})
\begin{equation}\label{gr}
\operatorname{Graph}(F^\circ)=\rho^{-1}([0,1]) = \{ (x,v) \in \Omega
\times \R^d : \rho_x(v)\leq 1 \}
\end{equation}
is a measurable set in $\mathcal A \otimes \mathcal B$. By
Theorem~\ref{char},  $F^\circ$ is measurable. Consequently, $F=(F^\circ)^\circ$ is a measurable
convex-set valued mapping by Theorem \ref{pol}.

\medskip

The converse is much easier to prove. Suppose $F: \Omega \to \bcs$ is
measurable and  define $\rho$
by \eqref{mc1}. By Theorem \ref{pol-measure},
$F^\circ: \Omega \to \K_{cs}(\R^d)$ is measurable. Hence, by
\eqref{gr}, $\rho^{-1}([0,t])$ is a measurable subset in
$\mathcal A \otimes \mathcal B$ for $t=1$. By scaling, the same is
true for any $t>0$. Since the $\sigma$-algebra of open sets in
$[0,\infty)$ is generated by sets of the form $[0,t]$, by a standard
measure theory argument 
$\rho^{-1}(U)$ is measurable for any open set $U \subset
[0,\infty)$. Thus, $\rho$ is measurable in the sense of ({\em i-a}) and
so $\rho$ is a seminorm function. Finally, the one-to-one
correspondence is a consequence of Theorems \ref{eqfun} and \ref{pol-measure}.
\end{proof}

As a corollary of Theorem \ref{pol-measure} and Theorem~\ref{mc}
we have the following.

\begin{corollary} \label{cor:dual-norm-measurable}
  If $\rho: \Omega \times \R^d \to
  [0,\infty)$ is a norm function, then $\rho^*: \Omega \times \R^d \to
  [0,\infty)$, defined by $\rho_x^*(v) = (\rho_x)^*(v)$, is a measurable
norm function.
\end{corollary}

The following lemma, whose proof makes use of seminorm functions, will
be used below.  

\begin{lemma}\label{pl}
  Every measurable mapping $F: \Omega \to \bcs$ is the pointwise limit
  of simple measurable mappings with respect to the Hausdorff distance
  on $\bcs$.
\end{lemma}

\begin{proof} Suppose first that $F: \Omega \to \abcs$. Then the
  corresponding seminorm function
  $\rho: \Omega \times \R^d \to [0,\infty)$ from Theorem \ref{mc} is
  actually a norm function. With the same notation as in the proof of
  Theorem \ref{mc},  for each $n\in\N$ we define the seminorm function
\[
\rho^n(x,v) = \sup_{1\leq i \leq n} |l_i(v)| \chi_{A_{l_i}}(x)
  \qquad\text{for all }(x,v) \in \Omega \times \R^d.
\]
Let $F_n: \Omega \to \bcs$ be the corresponding convex-valued
function. Clearly, $F_n$ is a simple measurable function. Since
$\rho^n(x,v) \nearrow \rho_x(v)$ as $n\to\infty$ for all
$(x,v)\in \Omega \times \R^d$, we have that
\[
F_1(x) \subset F_2(x) \subset \cdots \qquad\text{and}\qquad F(x)= \bigcup_{n\in \N} F_n(x).
\]
For any $x\in \Omega$, $F_n(x) \in \abcs$ for sufficiently large $n$. 
By the characterization of the convergence of convex bodies in~\cite[Theorem~1.8.7]{MR1216521},
we have that $F_n(x) \to F(x)$ as $n\to \infty$ with
respect to the Hausdorff distance in $\bcs$.

Finally, let $F:\Omega \to \bcs$ be any measurable function.  Define
the measurable functions $G_n: \Omega \to \abcs$, $n\in \N$, by
$G_n(x)=F_n(x)+\frac{1}{n}\overline {\mathbf B}$. Since each $G_n$ is a
pointwise limit of simple measurable mappings, a Cantor diagonalization argument
shows that so is $F$.
\end{proof}

The next result extends the correspondence between convex-set valued
mappings and seminorm functions in Theorem \ref{mc} to their
respective integrals.

\begin{theorem}\label{sn} Let $F: \Omega \to \bcs$ be a measurable
  mapping, and  let $\rho$ be the corresponding seminorm function
  given by  \eqref{mc1}. Then $F$ is integrably bounded if and only if
  for all $v\in \R^d$,
\begin{equation}\label{sn0}
p(v):= \int_\Omega \rho_x(v) d\mu(x)<\infty.
\end{equation}
In this case $p$ is a seminorm  which coincides with the Minkowski
functional of the polar set of $\int_\Omega F d\mu$. In other words,
\begin{equation}\label{sn2}
  \int_\Omega F d\mu
  =
  \bigg\{v\in\R^d: \int_\Omega \rho_x(v) d\mu(x) \le 1 \bigg\}^\circ.
\end{equation}
\end{theorem}

\begin{proof}
  We first consider the special case when $F: \Omega \to \bcs$ is a
  simple mapping: i.e., $F$ takes only finitely many values
  $K_1,\ldots,K_m \in \bcs$. Hence, it can be written in the form
\[
F(x) = \sum_{i=1}^m \chi_{A_i}(x) K_i \qquad x\in \Omega,
\]
where $A_1,\ldots, A_m$ are disjoint measurable sets such that
$\bigcup_{i=1}^m A_i = \Omega$. The corresponding seminorm function
$\rho$ also takes on finitely many values and satisfies
\[
\rho_x(v)= \sum_{i=1}^m \chi_{A_i}(x) p_{(K_i)^\circ}(v), \qquad (x,v) \in \Omega \times \R^d.
\]

The mapping $F$ is integrably bounded if and only if
$\mu(A_i) <\infty$ for any $i$ such that $K_{i} \ne \{0\}$. In this case the
Aumann integral of $F$ equals
\[
K=\int_\Omega F d\mu = \sum_{i=1}^m \mu(A_i) K_i,
\]
where we use the convention that $\mu(A_i) K_i=\{0\}$ if
$\mu(A_i)=\infty$ and $K_i=\{0\}$.  Hence, by Theorem
\ref{eqnfun-bis}({\em a})({\em b}) the seminorm $p$ given by \eqref{sn0} satisfies
\[
p= \sum_{i=1}^m \mu(A_i) p_{(K_i)^\circ}= p_{K^\circ}.
\]
Thus, \eqref{sn0} and \eqref{sn2} hold exactly when $F$ is integrably
bounded.  This proves Theorem \ref{sn} for simple mappings $F: \Omega
\to \bcs$.

\medskip

Now fix a general measurable mapping
$F: \Omega \to \mathcal K_{bcs}(\R^d)$ and let $\rho$ be the
associated seminorm function.  Suppose first that $F$ is integrably
bounded.  We need to show that for all $v\in \R^d$, $x\mapsto
\rho_x(v)$ is in $L^1(\Omega)$.    Since $F$ is integrably bounded, there
exists $k\in L^1$ such that for all $x$, $F(x)\subset k(x)\mathbf{B}$.  But
then by Lemma~\ref{sf},
\begin{equation} \label{eqn:sn3}
  \rho_x(v) = p_{F(x)^\circ}(v) = h_{F(x)}(v)
  = \sup_{w\in F(x)} \langle v, w \rangle 
  \leq \sup_{w\in k(x)\mathbf{B}} \langle v, w \rangle
  = k(x)|v|. 
\end{equation}
It is immediate that $x \mapsto \rho_x(v)$ is in $L^1$.

Conversely,  suppose that \eqref{sn0} holds for all $v\in \R^d$. Define 
$k:\Omega \to [0,\infty)$ by
\[
k(x):=\sup_{v\in\R^d, \ |v|=1} \rho_x(v), \qquad x\in\Omega.
\]
Then,  since $v$ is a convex combination of the standard basis vectors $\pm e_i$, $i=1,\ldots, d$, by the
triangle inequality
\[
\int_\Omega k(x)d \mu(x) 
\le 
\int_{\Omega} \sum_{i=1}^d \rho_x(e_i) d\mu(x) 
= 
\sum_{i=1}^d p(e_i) <\infty.
\]
Thus, $k\in L^1(\Omega)$.  Furthermore, if we let $\rho_k$ be the
seminorm function defined by $x\mapsto k(x)|\cdot|$, then
$\rho_x(v)\leq \rho_k(v)$, and arguing as we did in~\eqref{eqn:sn3} we
conclude that $F$ is integrably bounded.

\medskip

We now prove \eqref{sn2}. We will first prove the special case
where 
$F(x) \subset \R^d$ is absorbing for all $x\in \R^d$.
If we argue as we did in the proof of Lemma~\ref{pl}, then we have
that there exists a sequence of simple, convex-set value mappings
$F_n : \Omega \rightarrow \K_{bcs}(\R^d)$ such that $F_n(x)\subset
F_{n+1}(x)$ for all $n\in \N$ and
\[ F(x) = \bigcup_{n\in \N} F_n(x). \]
We  now apply
the Lebesgue dominated convergence theorem for
convex-set valued mappings \cite[Theorem 8.6.7]{MR2458436}  to the
sequence $\{F_n\}_{n\in\N}$  to get that
\begin{equation}\label{sn16}
  K = \int_\Omega F \,d\mu
  = \lim_{n\to \infty} \int_\Omega F_n \,d\mu 
  = \overline{\bigcup_{n\in \N} \int_\Omega F_n \,d\mu }
  = \clconv\bigg(\bigcup_{n \in \N} K_n\bigg), 
\end{equation}
where   $K_n = \int_\Omega F_n d\mu$.
Here we interpret the limit in the middle term as the Kuratowski
limit of closed sets~\cite[Section~1.1]{MR2458436}. Similarly, if we
let $\rho^n$  be the seminorm associated with $F_n$ (by Theorem~\ref{mc}), and if we apply the 
monotone convergence theorem to the sequence $\{\rho^n\}_{n\in\N}$, we
get that for all $v\in \R^d$,
\begin{equation}\label{sn18}
  p(v)=\int_\Omega \rho_x(v) \,d\mu(x)
  = \lim_{n\to \infty} \int_\Omega \rho_x^n(v) \,d\mu(x)
  = \sup_{n\in\N} p_n(v),
\end{equation}
where $p_n(v)=\int_\Omega \rho_x^n(v)$.
As we proved above for simple functions, for each
$n\in\N$,
\begin{equation}\label{sn20}
K_n  = \bigg\{v\in\R^d: p_n(v) \le 1 \bigg\}^\circ.
\end{equation}
Therefore, 
by Theorems \ref{eqfun} and~\ref{eqnfun-bis}(c),
\[ p(v)  = \sup_{n\in\N} p_n(v) = \sup_{n\in\N} p_{K_n^\circ}(v) =
  p_{K^\circ}(v). \]
It follows at once that~\eqref{sn2} holds.

\medskip

Finally, we consider the general case where $F: \Omega \to \bcs$ is an
arbitrary integrably bounded mapping.  For each $j\in \N$,  define a new
convex-set valued mapping $F_j = F+  \tfrac{k}{j} \mathbf B$ and its
corresponding seminorm function $\rho^j$.  Then $F_j$ is
integrably bounded and absorbing, and so $\rho^j_x(\cdot)$ is a norm for all
$x\in \Omega$.   Therefore, by the previous case,
\begin{equation}\label{sn22}
  K_j = \int_\Omega F_j \,d\mu
  = \bigg\{v\in\R^d: \int_\Omega \rho^j_x(v) \,d\mu(x) \le 1 \bigg\}^\circ.
\end{equation}
The convex sets $K_j$ form a nested, decreasing sequence, so again 
by the Lebesgue dominated convergence theorem for convex-set valued mappings
\cite[Theorem 8.6.7]{MR2458436} we have that
\[
K=\int_\Omega F d\mu = \lim_{j \to \infty} \int_\Omega F_j \,d\mu
=\bigcap_{j\in \N} \int_\Omega F_j \,d\mu =\bigcap_{j\in \N} K_j,
\]
where the limit is the Kuratowski limit.
By the dominated convergence theorem applied to the sequence
$\rho^j$, we have that for
all $v\in \R^d$, 
\[
  p(v)
  = \lim_{j \to \infty} \int_\Omega \rho^j_x(v) \,d\mu(x)
  = \inf_{j\in \N} \int_\Omega \rho^j_x(v)\, d\mu(x)
  = \inf_{j\in \N}p_{(K_j)^\circ}(v). \]
Therefore, by Theorems \ref{eqfun} and \ref{eqnfun-bis}({\em d}), the identity~\eqref{sn2} follows at once.
\end{proof}

\subsection*{$L^p$ spaces of convex-set valued functions}
In this section we define a natural generalization of the space $L^p(\Omega,\rho)$ of vector-valued functions $f: \Omega \to \R^d$ equipped with the norm
\[
  \|f\|_{L^p(\Omega,\rho)}
  = \bigg( \int_\Omega \rho_x(f(x))^p d\mu(x) \bigg)^{\frac{1}{p}} <\infty.\]
Recall our standing assumption that $(\Omega,\mathcal A,\mu)$ is a
positive, $\sigma$-finite, and complete measure space. Given a fixed
seminorm function $\rho: \Omega \times \R^d \to [0,\infty)$ we will
define the space $L^p_{\K} (\Omega,\rho)$ of convex-set valued mappings
$F:\Omega \to \bcs$. 
To do so, we first prove a basic measurability lemma. For any seminorm
$p$ on $\R^d$ and $K\in \bcs$, define $p(K)=\sup\{ p(v): v\in K\}$.

\begin{lemma} \label{lemma:norm-measure}
  Let $\rho: \Omega \times \R^d \to [0,\infty)$ be a seminorm function
  and let $F: \Omega \to \bcs$ be a measurable convex-set valued
  mapping. Then
\[ 
x\mapsto \rho_x(F(x)) = \sup\{ \rho_x(v) : v \in F(x) \} \]
is a measurable  function from  $\Omega$ to $[0,\infty)$.
\end{lemma}

\begin{proof}
If $f=\sum v_i\chi_{A_i}$ is a simple, vector-valued function, then
the map $x\mapsto \rho_x(f(x))= \sum_i \rho_x(v_i)\chi_{A_i}(x)$, is
measurable by the definition of seminorm functions.  
By Theorem \ref{char} there exists a sequence of measurable
selection functions
  $\{f_k\}_{k\in \N}$ of $F$ such that~\eqref{char2} holds.
Since for each $k\ge 1$, the function $f_k: \Omega \to \R^d$ is a
pointwise limit of simple measurable functions,  so by the above
observation we have that
\[
x\mapsto \rho_x(F(x)) = \sup_{k\in \N} \rho_x(f_k(x))
\]
is measurable.
\end{proof}

\begin{definition}
Suppose that $\rho$ is a seminorm function on $\Omega$. For each $p$,
$0<p<\infty$, define the Lebesgue space of convex-set valued mappings
$L^p_{\K}(\Omega,\rho)$ to be the set of measurable mappings
$F:\Omega \to \bcs$ such that
\[ \|F\|_{L_\K^p(\Omega,\rho)}= \|F\|_p
  = \bigg(\int_\Omega \rho_x(F(x))^p d\mu(x) \bigg)^{\frac{1}{p}}<\infty.
\]
 When $p=\infty$, define $L^\infty(\Omega,\rho)$ to be the set of all
 such $F$ that satisfy
 \[
 \|F\|_{L_\K^\infty(\Omega,\rho)}= \|F\|_\infty = \esssup_{x\in \Omega} \rho_x(F(x))<\infty.
 \]
\end{definition}

A straightforward argument shows that $\|\cdot \|_p$, $1\le p \le
\infty$, satisfies the usual properties of a seminorm:
\begin{enumerate}
\item
if $F(x)=\{0\}$ for a.e.~$x \in \Omega$, then $\|F\|_p =0$, and if $\rho$ is
a norm for almost every $x$, then the converse holds;
\item
$\|\alpha F\|_p = |\alpha| \|F\|_p$ for any $F\in L^p_{\K}$ and $\alpha\in \R$;
\item
$\|F+G\|_p \le \|F\|_p + \|G\|_p$ for any $F,\,G\in L^p_{\K}$.
\end{enumerate}
However, unlike its classical vector-valued analog,
$L^p_{\K}(\Omega,\rho)$ is {\bf not} a vector space because $\bcs$
equipped with the Minkowski addition is only a semigroup: the additive
inverse does not exist. Nevertheless, we have that
$L^p_{\K}(\Omega,\rho)$ is a complete metric space.

Recall that as we noted above, the set of nonempty, compact, convex
sets equipped with the Hausdorff distance \eqref{hd} is a complete
metric space. Given a norm $\rho_x$ on $\R^d$ and two compact sets
$K_1,K_2 \subset \R^d$,  define the corresponding Hausdorff distance function
\begin{equation} \label{eqn:hausdorff-dist-func}
d_{H,x}(K_1,K_2)=\max\{ \sup_{v\in K_1} \inf_{w\in K_2} \rho_{x}(v-w), \sup_{v\in K_2} \inf_{w\in K_1} \rho_{x}(v-w) \}
\end{equation}
If in \eqref{eqn:hausdorff-dist-func} the sets $K_1$ and $K_2$ are replaced by  countable dense subsets, this function is measurable. 
By Lemma~ \ref{pl},  any
measurable convex-set valued mapping $F:\Omega \to \bcs$ is a pointwise
limit of simple measurable mappings with respect to the Hausdorff topology
on $\bcs$.  Hence, given any $F,\,G \in L^p_{\K}(\Omega,\rho)$, we have that
$x\mapsto  d_{H,x}(F(x), G(x))$ is measurable, so we can define
the distance function
\begin{equation} \label{metric}
d_p(F,G)= \bigg(\int_\Omega d_{H,x}(F(x), G(x))^p d\mu(x) \bigg)^{\frac{1}{p}}.
\end{equation}
Note that for any $F \in L^p_{\K}(\Omega,\rho)$, $\|F\|_p = d_p(F,\{0\})$.
Moreover, we have that $d_p$ is a metric and we have the following
analogue of the classical result for vector-valued $L^p_\K(\Omega,\rho)$
spaces.

\begin{theorem} \label{thm:complete-metric-space}
Given  a norm function $\rho :\Omega \times [0,\infty) \rightarrow
\R^d$, 
  the space $L^p_{\K}(\Omega,\rho)$, $1\le p\le \infty$, equipped with
  $d_p$ is a complete metric space (after identifying functions that
  are equal to $\{0\}$ a.e). In addition, this metric is invariant and
  homogeneous: that is, for $F,\,G,\,H \in L^p_{\K}$ and $\alpha \in \R$, 
\begin{gather}\label{lp1}
d_p(F+H,G+H) =d_p(F,G) \\
\label{lp2}
d_p(\alpha F, \alpha G)  = |\alpha|d_p (F,G).
\end{gather}
\end{theorem}

\begin{proof}
  The proof that $d_p$ is a metric is straightforward:  it follows from the
  triangle inequality for the Hausdorff distance $d_{H,x}$ and Minkowski's
  inequality on the scalar-valued spaces $L^p(\Omega)$.  Properties
  \eqref{lp1} and \eqref{lp2} then follow immediately from the analogous
  properties for the Hausdorff distance $d_{H,x}$.

  It remains to prove that $L^p_{\K}(\Omega,\rho)$ is complete.  We
  will do this in the case $1\le p<\infty$ by adapting the proof of
  the classical Riesz-Fischer theorem. The case $p=\infty$ is
  much easier: the proof is similar to that of the  completeness of
  $L^\infty(\Omega)$ and we leave the details to the reader.

  Let $\{F_n\}_{n\in\N}$ be a Cauchy sequence in
  $L^p_{\K}(\Omega,\rho)$. Then there exists a strictly increasing
  sequence $\{n_i\}_{i\in \N}$ such that for $i\geq 1$, 
$d_p(F_{n_{i+1}},F_{n_i})< 2^{-i}$. 
Let $\mathbf B_x=\{v\in \R^d: \rho_x(v)\le 1\}$.  For each $k\in \N$,
define $g_k:\Omega \to [0,\infty)$ by 
\[
g_k(x) = \sum_{i=1}^k d_{H,x}(F_{n_{i+1}}(x),F_{n_i}(x)), \qquad x\in \Omega,
\]
and define $g:\Omega \to [0,\infty]$ by 
\[
g(x) = \sum_{i=1}^\infty d_{H,x}(F_{n_{i+1}}(x),F_{n_i}(x)), \qquad x\in \Omega
. 
\]
We claim that  $g(x)<\infty$ for a.e.~$x\in
\Omega$.   To see this, note that 
by  Minkowski's inequality,
\begin{multline*}
  \|g_k\|_p
  = \bigg(\int_\Omega \bigg(\sum_{i=1}^k
  d_{H,x}(F_{n_{i+1}}(x),F_{n_i}(x)) \bigg)^p d\mu(x) \bigg)^{\frac{1}{p}}  \\
\le 
\sum_{i=1}^k d_p(F_{n_{i+1}},F_{n_i}) 
\le
\sum_{i=1}^k 2^{-i} <1.
\end{multline*}
Then by Fatou's lemma we have that
\[
  \|g\|^p_p
  \le \liminf_{k\to\infty} \|g_k\|_p^p
  \le 1.
\]

For each $k\in \N$,
define $h_k:\Omega \to [0,\infty]$ by 
\[
h_k(x)= \sum_{i=k}^\infty d_{H,x}(F_{n_{i+1}}(x),F_{n_i}(x)), \qquad x\in \Omega. 
\]
For any $i,j \ge k$, the triangle inequality implies that
\[
d_{H,x}(F_{n_{i}}(x),F_{n_j }(x)) \le h_k(x) \leq g(x).
\]
Since $g(x)<\infty $ for a.e.~$x\in \Omega$, we have $h_k(x) \to 0$ as
$k\to \infty$. Hence, for a.e.~$x\in \Omega$, the sequence
$\{F_{n_i}(x)\}_{i\in \N}$ is Cauchy in $\bcs$ with respect to the
Hausdorff distance $d_{H,x}$ given by \eqref{eqn:hausdorff-dist-func}.
Since $\bcs$ is a closed subset of $\K_b(\R^d)$ in the Hausdorff
topology, the sequence $\{F_{n_i}(x)\}_{i\in \N}$ converges to some
set $F(x)\in \bcs$ for a.e.~$x\in\Omega$. Since $F: \Omega \to \bcs$
is the pointwise a.e.~limit of measurable functions $F_{n_i}$,
$i\in \N$, $F$ is measurable as well.  Finally, we have that $F$ is
the limit of $\{F_n\}_{n\in\N}$ in $L^p_{\K}$.  Since
$F_{n_i}(x)\rightarrow F(x)$ in $d_{H,x}$ distance, by the triangle
inequality,
\[ d_{H,x}(F(x),F_{n_k}(x))
  \leq \lim_{i\rightarrow \infty} d_{H,x}(F_{n_i}(x),F_{n_k}(x)) \leq
  h_k(x). \]
Therefore, as $k\rightarrow \infty$,  by the Lebesgue dominated convergence theorem
we have that
\begin{multline*}
\lim_{k\rightarrow \infty}  d_p(F,F_{n_k})^p
  = \lim_{k\rightarrow \infty}  \int_\Omega d_{H,x}(F(x), F_{n_k}(x))^p d\mu(x)  \\
  \le \lim_{k\rightarrow \infty}  \int_\Omega \bigg(\sum_{i=k}^\infty
  d_{H,x}(F_{n_{i+1}}(x),F_{n_i}(x)) \bigg)^p d\mu(x)
  = 0. 
\end{multline*}
Finally, since $\{F_n\}_{n\in\N}$ is a Cauchy sequence in
$L^p_{\K}(\Omega,\rho)$, by a standard argument we have that
$d_p(F,F_n) \to 0$ as $n\to \infty$. This completes the proof.
\end{proof}

\begin{remark}
  While $L^p_{\K}(\Omega,\rho)$ is not a Banach space, one can show that it
  is a convex cone of some Banach space. By the R\aa dstr\"om embedding
  theorem \cite[Theorem 1]{MR0045938}, the collection of
  all nonempty, compact convex subsets of a normed, real vector space
  (endowed with the Hausdorff distance) can be isometrically embedded
  as a convex cone in a normed real vector-space. Thus, for
  a.e.~$x$, the Hausdorff distance $d_{H,x}$ comes from a
  certain norm $\|\cdot \|_x$ on a vector space $V$, which
  consists of equivalence classes of pairs of compact convex sets
  under the relation:
\[
(K_1,K_2) \sim (K_3,K_4) \quad \text{if and only if} \quad  K_1+K_4=K_2+K_3.
\]
The space $V$ with norm $\|\cdot\|_x$ is a complete separable normed
space; all the norms $\|\cdot\|_x$ are mutually equivalent since they
come from equivalent norms $\rho_x$ on a finite dimensional space
$\R^d$.  Thus, $L^p_{\K}(\Omega,\rho)$ can be identified with a
weighted vector-valued space $L^p(\Omega, \|\cdot\|_x)$ consisting of
all measurable functions $f:\Omega \to V$ such that
\[
\|f\|_p = \bigg(\int_\Omega \|f(x)\|_x^p d\mu(x) \bigg)^{\frac{1}{p}}<\infty \}.
\]
Since we will not use this fact elsewhere, we leave the details to the interested reader.
\end{remark}

\subsection*{Matrix weights and seminorms}
Let $A : \Omega \rightarrow \Md$ be a measurable matrix mapping.
Then we can define a seminorm function $\rho_A$ by
\[ \rho_A(x,v) = |A(x)v|, \qquad  x\in \Omega, \, v \in \R^d. \]
Clearly, for each $x$, $\rho_A(x,\cdot)$ is a seminorm, and since $A$
is measurable, the map $x\mapsto \rho_A(x,v)$ is measurable for all
$v$.  Moreover, in defining seminorms it suffices to restrict
ourselves to measurable, positive semidefinite matrix  mappings $W : \Omega \rightarrow \Sd$.

\begin{theorem} \label{thm:Sd-weight}  Given a measurable matrix mapping $A : \Omega \rightarrow \Md$,
  there exists a measurable matrix mapping $W : \Omega \rightarrow
  \Sd$ such that for all $x\in \Omega$ and $v\in \R^d$,
  $\rho_A(x,v)=\rho_W(x,v)$.  If $A$ is invertible for a.e. $x\in
  \Omega$, then $W$ is positive definite almost everywhere.
\end{theorem}

\begin{proof}
  Given a matrix $A\in \Md$, it is well-known that if we form the polar
  decomposition of $A$ we can write $A=UW$, where $U$ is orthogonal
  and $W\in \Sd$.  Further, if $A$ is invertible, then $W$ is positive
  definite.  But then, for any $v\in \R^d$,
  \[ |Av|= |UWv|= |Wv|. \]
  Therefore, it suffices to show that we can take $W$ to be a
  measurable function.   We can define $W$ by $W=(A^tA)^{1/2}$, so we
  need to show that we can measurably define the square root of a
  postive semidefinite matrix.

  Let $V : \Omega \rightarrow  \Sd$ be a measurable mapping, then
  by~\cite[Lemma~2.3.5]{MR1350650} there exists a measurable matrix
  mapping $U$ such that $U(x)$ is orthogonal and $U^t(x)V(x)U(x)$ is
  diagonal.  Denote this matrix by $D=\diag(\lambda_1,\ldots,\lambda_d)$,
  and define its square root to be the diagonal matrix
  $D^{1/2}=\diag(\lambda_1^{1/2},\ldots,\lambda_d^{1/2})$.   If we now
  define $V^{1/2}= UD^{1/2}U^t$, then  $V^{1/2}$ is measurable and
  $V^{1/2}V^{1/2}=V$. 
\end{proof}

Conversely, given a norm function $\rho_x$, we can associate to it a
matrix norm $\rho_W$.  This result was proved
in~\cite[Proposition~1.2]{MR2015733}.  For completeness, and to
emphasize the role of measurability, we include the short proof.

\begin{theorem} \label{thm:matrix-norm}
Let $\rho$ be a norm function.  Then there exists a measurable matrix
mapping $W : \Omega \rightarrow \Sd$ such that for a.e. $x\in \Omega$,
$W$ is positive definite, and for every $v\in \R^d$,
\[ \rho_W(x,v) \leq \rho(x,v) \leq \sqrt{d} \rho_W(x,v). \]
\end{theorem}

\begin{proof}
Let
\[ K(x) = \{ v\in \R^d : \rho_x(v)\leq 1 \} \]
be the unit ball of $\rho_x$, $x\in \Omega$.  Then by Corollary~\ref{cor:min-norm},
$K(x) \in \abcs$.  Further,
$K : \Omega \rightarrow \abcs$ is a measurable mapping.  To see this,
note that by Theorem~\ref{mc}, $K^\circ$ is measurable, so by
Theorem~\ref{pol-measure}, $K$ is measurable.  Therefore, by Theorem~\ref{mj} there exists a measurable
matrix mapping $A : \Omega \rightarrow \Md$ such that
\begin{equation}\label{mno}
A(x)\overline{\mathbf{B}} \subset K(x) \subset \sqrt{d} A(x)\overline{\mathbf{B}}.
\end{equation}
We therefore have that for $x\in \Omega$
and $v\in \R^d$,
\[ p_{A(x)\mathbf{B}}(v) \leq \rho(x,v) \leq \sqrt{d} p_{A(x)\mathbf{B}}(v),  \]
where $p_{A(x)\overline{\mathbf{B}}}$ is the Minkowski functional of
$A(x)\overline{\mathbf{B}}$. (See Definition~\ref{mink}.)
It follows from \eqref{mno} that $A$ is invertible.  Thus,
\begin{multline*}
  p_{A(x)\mathbf{B}}(v) = \inf\{ r>0 : \frac{v}{r}\in
  A(x)\overline{\mathbf{B}} \} \\
  = \inf \{ r>0 : A^{-1}(x)v \in r\overline{\mathbf{B}} \} = |A^{-1}(x)v| =
  \rho_{A^{-1}}(x,v). 
\end{multline*}
  Finally, by Theorem~\ref{thm:Sd-weight}, there exists a measurable,
  positive definite
  matrix mapping $W : \Omega \rightarrow \Sd$ such that
  $\rho_W(x,v)=\rho_{A^{-1}}(x,v)$.  This completes the proof.
  \end{proof}

  \begin{prop} \label{prop:dual-matrix-norm}
    If $W : \Omega \rightarrow \Md$ is invertible a.e., then
    for a.e. $x\in \Omega$ and every $v\in \R^d$,
    $\rho_W^*(x,v) = \rho_{(W^*)^{-1}}(x,v)$.  In particular, if $W$
    is symmetric a.e., then $\rho_W^*=\rho_{W^{-1}}$. 
  \end{prop}

  \begin{proof}
    Arguing as in the proof of Theorem~\ref{thm:matrix-norm}, we have
    that the unit ball of $\rho_W$ is $K(x)=W^{-1}(x)\overline{\mathbf{B}}$, and
    so by Theorem~\ref{pol}, the unit ball of $\rho_W^*$ is
    \begin{multline*}
      K(x)^\circ
      =
      \{ v\in \R^d : |\langle v, W^{-1}(x)y \rangle|\leq 1, y \in
      \overline{\mathbf{B}} \} \\
      =
      \{ v\in \R^d : |\langle (W^*)^{-1}(x)v, y \rangle|\leq 1, y \in
      \overline{\mathbf{B}} \}
      =
      W^*(x)\overline{\mathbf{B}}. 
    \end{multline*}
    As above, $W^*(x)\overline{\mathbf{B}}$ is the unit ball of $\rho_{(W^*)^{-1}}$.  
    By Corollary~\ref{cor:min-norm}, if two norms have the same unit ball, they
    are the same norm, so   $\rho_W^*(x,v) = \rho_{(W^*)^{-1}}(x,v)$.
  \end{proof}

  \medskip

  We refer to the matrix $W$ in Theorem~\ref{thm:matrix-norm} as the
  matrix weight associated with the norm function $\rho$.   Then we
  have that a function $F \in L^p_\K(\Omega,\rho)$ if and only if $F\in
  L^p_\K(\Omega,\rho_W)$, and
  \[ \|F\|_{L^p_\K(\Omega,\rho_W)}
    \leq \|F\|_{L^p_\K(\Omega,\rho)}
      \leq \sqrt{d} \|F\|_{L^p_\K(\Omega,\rho_W)}. \]
    We will thus be able to pass between these spaces depending on
    which is most convenient.  The spaces $L^p_\K(\Omega,\rho_W)$ are
    referred to as matrix weighted spaces; for simplicity we will
    often denote them by $L^p_\K(\Omega, W)$.   Closely connected to
    these spaces are the matrix-weighted spaces of vector-valued
    functions, which we will denote $L^p(\Omega, W)$.  This space can
    be identified with a subset of $L^p_\K(\Omega, W)$ using the
    mapping defined in Lemma~\ref{lemma:vector-int}.

\section{The maximal operator on convex-set valued functions}
\label{section:maximal}

In this section we generalize the Hardy-Littlewood maximal operator to
the setting of convex-set valued functions.  Throughout this section, we
will take our underlying measure space to be $\R^n$ equipped with
Lebesgue measure.  We will use the standard Euclidean norm on $\R^d$,
and given a set $K\subset \R^d$, we define the norm of a set by
\begin{equation}\label{norm}
 |K| = \sup\{ |v| : v \in K \}.
\end{equation}
Hereafter, by a cube $Q$ we will always mean a cube whose sides are
parallel to the coordinate axes.  Unless we indicate otherwise, all
integrals are taken with respect to the Lebesgue measure $\ln$ of
$\R^n$. The volume of the cube $Q$ is denoted by $\ln(Q)$ rather than
the customary $|Q|$ to avoid ambiguity with the norm of a set given by
\eqref{norm}.

\subsection*{Averaging operators}
We first consider the simpler case of averaging operators.  
Given a function $F :
\R^n \rightarrow \bcs$ that is locally integrably bounded, if we fix a cube
$Q$,  then we define the averaging operator  $A_Q$ by
\[ A_Q F(x) =  \avgint_Q F(y)\,dy \cdot \chi_Q(x)
  = \frac{1}{\ln(Q)}\int_Q F(y)\,dy \cdot \chi_Q(x). \]

Therefore, $A_QF(x)$ is the ``average'' of $F$ on $Q$ if $x\in Q$, and
is the set $\{0\}$ otherwise.  However, the associated convex set can
be quite different from $F$, even if it is the convex-set valued
function associated to a vector-valued function (as in
Lemma~\ref{lemma:vector-int}).  For example, let $f : \R \rightarrow
\R^2$ be defined by
\begin{equation} \label{eqn:vec-function}
 f(x) = \begin{cases}
    (1,1)^t, & x\geq 0; \\
    (-1,1)^t, & x < 0. 
  \end{cases}
\end{equation}
and let $F(x) = \clconv\{ f(x), -f(x)\}$.  Let $Q=[-1,1]$.  Then for
$x\in Q$
\[ A_QF(x) = \bigg\{ \avgint_Q k(y)f(y)\,dy :  k \in L^\infty(Q),
  \|k\|_\infty \leq 1 \bigg\}. \]
Fix any $k \in L^\infty(Q)$; then
\[ \avgint_Q k(y)f(y)\,dy
  =
  \frac{1}{2}\int_{-1}^0 k(y)\,dy \begin{pmatrix}-1 \\ 1 \end{pmatrix}
  +
  \frac{1}{2}\int_{0}^1 k(y)\,dy \begin{pmatrix}1 \\
    1 \end{pmatrix}. \]
The integrals are constants with values in $[-1,1]$ so without loss of
generality we may assume that $k$ is constant on $[-1,0)$ and $[0,1]$;
denote these values by $a$ and $b$.  Hence,
\[  A_QF(x) = \bigg\{  \frac{a}{2}\begin{pmatrix}-1 \\ 1 \end{pmatrix}
  + \frac{b}{2}\begin{pmatrix}1 \\ 1 \end{pmatrix} :  |a|,\,|b| \leq 1\bigg\}. \]
It follows immediately that $A_QF(x)$ is equal to the square with
vertices $(\pm 1,0)$, $(0,\pm 1)$.

\medskip

By Theorem~\ref{cci},
$A_Q : \R^n \rightarrow \bcs$ and is a measurable mapping. The
averaging operators are linear operators in the sense of Lemma
\ref{lemma:avg-op-linear} below. We use this terminology, even
though $\bcs$ is not a vector space,  because of the compelling form of
the identities \eqref{lin1} and \eqref{lin2}. Lemma
\ref{lemma:avg-op-linear} is an immediate consequence of the linearity
of the Aumann integral, Theorem~\ref{add}.

\begin{lemma} \label{lemma:avg-op-linear}
  Given any cube $Q$, the averaging operator $A_Q$ is linear:
  if $F, \, G  :
\R^n \rightarrow \bcs$  are  locally integrably bounded  mappings, and $\alpha \in
\R$, then 
  \begin{gather}
    A_Q(F+G)(x) = A_QF(x) + A_QG(x),\label{lin1} \\
    A_Q(\alpha F)(x) =\alpha A_QF(x).\label{lin2}
  \end{gather}
\end{lemma}

For $1\leq p < \infty$, if $F\in L^p_\K(\R^d,|\cdot|)$, then $F$ is
locally integrably bounded:  if we define $k(x)=|F(x)|$, then by
definition, $k\in L^p(\R^n)$, and so $k\in L^1_{loc}(\R^n)$.  Since
$F(x) \subset k(x)\mathbf{B}$, $F$ is locally
integrably bounded.  In particular,  averaging operators are
well-defined on $L^p_\K(\R^d,|\cdot|)$.

\begin{prop} \label{prop:avg-op}
  Given a cube $Q$, for $1\leq p\leq \infty$, $A_Q : L^p_\K(\R^n,|\cdot|)
  \rightarrow L^p_\K(\R^n,|\cdot|)$, and $\|A_QF\|_p \leq \|F\|_p$. 
\end{prop}

\begin{proof}
Fix $p$, $1\leq p<\infty$.  By the definition of the norm in $L^p_\K(\R^n,|\cdot|)$,
  \[ \|A_QF\|_{L^p_\K(\R^n,|\cdot|)}
    = \bigg(\int_\subRn \bigg| \avgint_Q F(y)\,dy \cdot
    \chi_Q(x)\bigg|^p \,dx\bigg)^{\frac{1}{p}}
    = \bigg|\avgint_Q F(y)\,dy \bigg| \ln(Q)^{\frac{1}{p}}. \]
  Since $F$ is locally integrably bounded, by Lemma~\ref{jem} there
  exists a selection function $v_F \in S^1(Q,F)$ such that
  $|v_F(x)|=|F(x)|$ for all $x\in Q$.  In particular, given any
  selection function $f\in S^1(Q,F)$, $|f(x)|\leq |v_F(x)|$. Therefore,
  \begin{multline*}
    \bigg|\avgint_Q F(y)\,dy \bigg|
    = \sup\bigg\{ \bigg|\avgint_Q f(y)\,dy \bigg|
    : f\in S^1(Q,F) \bigg\} \\
    \leq \avgint_Q |v_F(y)|\,dy
    \leq \bigg(\avgint_Q |v_F(y)|^p\,dy\bigg)^{\frac{1}{p}}
    = \|F\|_{L^p(\R^n,|\cdot|)} \ln(Q)^{-\frac{1}{p}}. 
  \end{multline*}
  If we combine these estimates we get the desired inequality.

  When $p=\infty$, the proof is similar but simpler.  By the
  definition of the $\|\cdot\|_\infty$ norm, for a.e. $x$, $|v_F(x)|=|F(x)|\leq
  \|F\|_{L^\infty(\R^n,|\cdot|)}$.  Hence, arguing as above,
  \[  \|A_QF\|_{L^\infty(\R^n,|\cdot|)}
    \leq  \avgint_Q |v_F(y)|\,dy  \leq
    \|F\|_{L^\infty(\R^n,|\cdot|)}. \]
\end{proof}

\begin{remark}
  We can also define the averaging operator by taking averages over
  balls $B$ instead of cubes.  Every result above remains true for
  these averaging operators.
\end{remark}

\subsection*{The convex-set valued maximal operator} 
We now extend the definition of the Hardy-Littlewood maximal
operator to convex-set valued functions.

\begin{definition}
  Given a locally integrably bounded function $F  :\R^n \rightarrow
  \bcs$,  define the maximal operator acting on $F$ by
  \[ MF(x) = \clconv\bigg(\bigcup_Q A_Q F(x) \bigg), \]
where the union is taken over all cubes $Q$ whose sides are parallel to
the coordinate axes.
\end{definition}

It is immediate from the definition that since $F(x)\in \bcs$,
$MF(x)\in \K_{cs}$.   The set $MF(x)$ can be a considerably larger
set than $F(x)$.  For example, if we let $f$ be the vector-valued
function~\eqref{eqn:vec-function} and define $F$ as before, then for
$x>0$, arguing as we did above, we can show that
\[ MF(x) = \clconv\bigg\{ \frac{-as}{t-s}\begin{pmatrix}-1 \\ 1 \end{pmatrix}
  + \frac{bt}{t-s}\begin{pmatrix}1 \\ 1 \end{pmatrix} :  |a|,\,|b| \leq
  1; t,\, s \in \R, s<0<x<t \bigg\}. \]
(The case $0\leq s<x<t$ should also be included, but it is easy to check that
it does not add anything to the set.)  If we reparameterize by setting
$s=-rt$, $0<r<\infty$, and then making the change of variables
$v=\frac{1}{1+r}$, we get that
\[ MF(x) = \clconv\bigg\{ \begin{pmatrix}-a \\ a \end{pmatrix}
  + v\begin{pmatrix}a+b \\ -a+b \end{pmatrix} :  |a|,\,|b| \leq
  1; 0< v <1\bigg\}. \]
By varying the parameters, it is straightforward to see that we get
all points in the square with vertices $(\pm 1, \pm 1)$.

\medskip

\begin{lemma} \label{lemma:Mf-sublinear}
  The maximal operator is sublinear:  for any  locally integrably bounded mappings $F, \, G  :
\R^n \rightarrow \bcs$ and $\alpha \in
\R$,
\[ 
  M(F+G)(x)  \subset MF(x) + MG(x), \qquad
  M(\alpha F)(x) = \alpha MF(x). 
\]
Further, the maximal operator is monotone: if $F(x) \subset G(x)$ for
all $x$, then $MF(x) \subset MG(x)$.
\end{lemma}

\begin{proof}
  Sublinearity follows from Lemma~\ref{lemma:avg-op-linear} and the
  linearity of the convex hull with respect to Minkowski sum:
  \begin{multline*}
    M(F+G)(x) = \clconv\bigg(\bigcup_Q \big[A_QF(x) +A_QG(x)
    \big]\bigg) \\
    \subset \clconv\bigg( \bigcup_Q A_QF(x) + \bigcup_Q A_QG(x) \bigg)
    = MF(x) + MG(x).
  \end{multline*}
  Similarly,
  \[ M(\alpha F)(x) = \clconv\bigg(\bigcup_Q \alpha A_QF(x) \bigg)
    = \alpha MF(x). \]
Monotonicity follows from the definition of the maximal
  operator and Theorem~\ref{add}. 
\end{proof}

Below we will also need a version of sublinearity that generalizes the
fact that in the scalar case, for $1<p<\infty$, the operator
$M_pf(x)=M(|f|^p)(x)^{\frac{1}{p}}$ is sublinear.  

\begin{lemma} \label{lemma:M_p-sublinear}
  Given $1<p<\infty$,  a locally integrably bounded mapping $H  :
\R^n \rightarrow \bcs$, non-negative functions $f,\,g\in
L^p_{loc}(\R^n) $, and a norm $\rho$, 
\[ \rho\big( M((f+g)^pH)(x)\big)^{\frac{1}{p}}
\leq \rho\big(M(f^pH)(x)\big)^{\frac{1}{p}}
+ \rho\big(M(g^pH)(x)\big)^{\frac{1}{p}}. \]
\end{lemma}

\begin{proof}
We introduce an auxiliary operator that simplifies our
  argument.  Given a locally integrably bounded, convex-set valued
  function $F$, define
  \[ \widehat{M} F(x) = \overline{ \bigcup_{Q} A_Q F(x) }.  \]
  Note that in contrast to the maximal operator $M$, the operator
  $\widehat M$ may not be convex-set valued since the convex hull is not
  present in the definition of $\widehat{M}$. This is not a problem
  since the operators $M$ and $\widehat{M}$ share the same boundedness
  characteristics. Indeed, we claim that
  $\rho(M F(x))=\rho(\widehat{M} F(x))$.
  Clearly, $\rho(M F(x))  \geq \rho(\widehat{M}F(x))$.  To see the reverse inequality, fix $v$ in
  \[ \conv\bigg( \bigcup_{Q} A_Q F(x) \bigg).  \]
  Then we can write $v$ as the finite sum $v=\sum \alpha_i v_i$,
  where $v_i \in A_{Q_i}F(x)$, $Q_i \in \D$, $\alpha_i \ge0$, and $\sum\alpha_i=1$.
  But then it is immediate that
  \[ \rho(v) \leq \sum \alpha_i \rho(v_i) \leq \rho(\widehat{M} F(x)), \]
  and the desired inequality follows at once.

  Given this equality we can argue as follows:  by Proposition~\ref{prop:convex-minkowski},
\begin{align*}
\rho\big(\widehat{M}((f+g)^pH)(x)\big)^{\frac{1}{p}}
& =   \rho\bigg(\bigcup_{x\in Q} A_Q((f+g)^pH)(x)\bigg)^{\frac{1}{p}}
  \\
  & = \sup_Q  \rho\big( A_Q((f+g)^pH)(x)\big)^{\frac{1}{p}} \\
  &\leq \sup_Q \rho\big(A_Q(f^pH)(x)\big)^{\frac{1}{p}}
    + \sup_Q \rho\big(A_Q(g^pH)(x)\big)^{\frac{1}{p}} \\
  & = \rho\big(\widehat{M}(f^pH)(x)\big)^{\frac{1}{p}}
+ \rho\big(\widehat{M}(g^pH)(x)\big)^{\frac{1}{p}}. \qedhere
\end{align*}
 \end{proof}

We claim that $MF$ is a measurable function.  To
show this, let $\mathcal{Q}$ denote the countable set of all cubes with
edges parallel to the coordinate axes, all of whose vertices have
rational coordinates.

\begin{prop} \label{prop:MF-measurable}
  Given a locally integrably bounded function $F  :\R^n \rightarrow
  \bcs$,
  \[ MF(x) = \clconv\bigg(\bigcup_{P\in  \mathcal{Q}} A_P F(x)
    \bigg). \]
  Consequently, $MF : \R^n \rightarrow \K_{cs}(\R^d)$ is a measurable
  function.
\end{prop}

\begin{proof}
  Fix $x\in \R^n$; then it is immediate that
  \[ \clconv\bigg(\bigcup_{P\in  \mathcal{Q}} A_P F(x) \bigg)
    \subset MF(x). \]
  To prove the reverse inclusion, fix a cube $Q$ containing $x$.  Then
  for any $\epsilon>0$, there exists a cube $P\in \mathcal{Q}$
  containing $Q$ such that $\ln(P)\leq (1+\epsilon)\ln(Q)$.  Hence, by
  Corollary~\ref{cor:integral-additive},
  \[ \avgint_Q F(y)\,dy \subset \frac{\ln(P)}{\ln(Q)} \avgint_P F(y)\,dy
    \subset (1+\epsilon)\avgint_P F(y)\,dy. \]
  Therefore,
  \[ MF(x) \subset (1+\epsilon)  \clconv
    \bigg(\bigcup_{P\in  \mathcal{Q}} A_P F(x) \bigg).  \]
  Since $\epsilon>0$ is arbitrary, we get that equality holds.

  Finally, since each averaging operator $A_P$, $P\in \mathcal{Q}$, is
  measurable, by Theorem~\ref{cup}, $MF$ is a measurable function.
\end{proof}

\begin{lemma} \label{lemma:Mf-dominates-F}
   Given a locally integrably bounded function $F  :\R^n \rightarrow
  \bcs$, for almost every $x\in \R^n$, $F(x) \subset MF(x)$. 
\end{lemma}

\begin{proof}
  By Theorem~\ref{char}, we can write
  \[ F(x) = \overline{\{ f_k(x): k\in \N\} }, \]
    where $f_k \in S^0(\R^n,F)$.  Since $F$ is locally integrably
    bounded, $f_k \in L^1_{loc}(\R^n)$, so the restriction $f_k|_Q \in S^1(Q,F)$ for any
    cube $Q$.  Since the
    collection $\{f_k\}_{k\in \N}$ is countable, by the Lebesgue
    differentiation theorem, for almost every $x\in \R^n$,
    \[ f_k(x) = \lim_{\substack{x\in Q\\ \ln(Q)\rightarrow 0}}
      \avgint_Q f_k(y)\,dy.  \]
    By the definition of the Aumann integral,
    \[ \avgint_Q f_k(y)\,dy \in \avgint_Q F(y)\,dy, \]
    and therefore $f_k(x)$ is a limit point of
    \[ \bigcup_Q A_QF(x).  \]
    Since
    $MF$ has values in closed sets, $f_k(x) \in MF(x)$ and the desired inclusion follows.
  \end{proof}
  
\medskip

There are alternative definitions of the maximal operator
that are analogous to the ones from the classical theory.   Let
$Q(x,r)$ be the cube centered at $x$ with side length $r$. Then we can
define the centered maximal operator
\[ M^cF(x) = \clconv\bigg(\bigcup_{r>0}  A_{Q(x,r)}F(x)
\bigg).  \]
We can also define a maximal operator where the averages are over balls containing $x$ instead of cubes,
\[  \widebar{M}F(x) = \clconv\bigg(\bigcup_{B}  A_{B}F(x)
  \bigg),  \]
where
\[ A_BF(x) = \avgint_B F(y)\,dy \cdot \chi_B(x) \]
is the averaging operator defined with respect to balls.  Similarly we
can restrict to  balls centered at $x$,
\[ \widebar{M}^cF(x) = \clconv\bigg(\bigcup_{r>0}  A_{B(x,r)}F(x)
\bigg).  \]
All of these maximal operators are equivalent to the maximal operator
$M$ as originally defined.  This follows from
Corollary~\ref{cor:integral-additive}, using the fact that given a
point $x$ and cube $Q$, then $Q\subset Q(x,2\ell(Q))$, and the fact
that given a ball $B(x,r)$,
\[ Q(x,n^{-n/2}r) \subset B(x,r) \subset Q(x,2r). \]
Since we will not use this result, we leave the details to the
interested reader. 

\medskip

More important is a dyadic version of the convex-set valued maximal
operator.  Given the collection of dyadic cubes
\[ \D = \{ 2^{k}([0,1)^n+m) : k\in \Z, m\in \Z^n \}, \]
 we can define the dyadic maximal operator
\[ M^d F(x) = \clconv\bigg(\bigcup_{Q\in \D} A_{Q}F(x)
  \bigg).  \]
It is immediate that $M^d$ has all the same properties as the maximal
operator $M$.  Moreover, from the definition we  have that
for any locally integrably bounded convex-set valued function $F$,
$M^dF(x) \subset MF(x)$.

The converse inclusion is not true, but if we define a larger family
of dyadic operators, a closely related inclusion is true.  For $\tau
\in \{0,\pm 1/3\}^n$, define the translated dyadic grid
\[ \D^\tau = \{ 2^{k}([0,1)^n+m +(-1)^{k}\tau) : k\in \Z, m\in \Z^n \}. \]
Then $\D^0=\D$; moreover, all of the dyadic grids $\D^\tau$
have the same essential properties as $\D$.  (See~\cite{CruzUribe:2016ji,MR3092729}.)  We
define the generalized dyadic maximal operator
\[ M^\tau F(x) = \clconv\bigg(\bigcup_{Q\in \D^\tau} A_{Q}F(x)
  \bigg).  \]

\begin{lemma} \label{lemma:dyadic}
  Given a locally integrably bounded, convex-set valued function $F :
  \R^n \rightarrow \K_{bcs}(\R^d)$,
  \[ MF(x) \subset C\sum_{\tau\in \{0,\pm 1/3\}^n} M^\tau F(x), \]
  where the constant $C$ depends only on the dimension $n$.
\end{lemma}

\begin{proof}
Fix $x\in \R^n$ and a cube $Q$ containing $x$.  Then there exists
$\tau\in \{0,\pm 1/3\}^n$ and a cube $P\subset \D^\tau$ such that
$Q\subset P$ and $\ell(P)\leq
3\ell(Q)$~\cite[Theorem~3.1]{CruzUribe:2016ji}.   Therefore,
\[ \avgint_Q F(y)\,dy \subset 3^n \avgint_P F(y)\,dy. \]
Since $0 \in F(y)$, $0\in \avgint_P F(y)\,dy$, and so
\[ \bigcup_Q A_Q F(x) 
\subset
3^n \bigcup_{\tau\in \{0,\pm 1/3\}^n} \bigcup_{P\in \D^\tau} A_PF(x)
\subset 3^n \sum_{\tau\in \{0,\pm 1/3\}^n} \bigcup_{P\in \D^\tau}
A_PF(x).  \]
By the linearity of the convex hull,
  \[ MF(x) \subset C\sum_{\tau\in \{0,\pm 1/3\}^n} M^\tau F(x).  \qedhere\]
\end{proof}

\subsection*{$L^p$ norm inequalities for the convex-set valued maximal operator}
In this section we prove strong and weak-type norm inequalities for
the convex-set valued maximal operator.

\begin{theorem} \label{thm:M-norm-ineq}
 For $1< p \leq \infty$, $M : L^p_\K(\R^n, |\cdot|) \rightarrow L^p_\K(\R^n,
 |\cdot|)$ is bounded.  When $p=1$, $M : L^1_\K(\R^n, |\cdot|) \rightarrow L^{1,\infty}_\K(\R^n,
 |\cdot|)$ is bounded. That is, for all $\lambda>0$ and $F \in L^1_\K(\R^n, |\cdot|)$, 
\[ \ln(\{x\in \R^n : |MF(x)|>\lambda \}) \leq 
\frac{C}{\lambda} \int_{\R^n} |F(x)|\,dx. \]
\end{theorem}

\begin{proof}
  Our proof adapts the classic proof of the boundedness of the dyadic
  maximal operator, which uses the Calder\'on-Zygmund cubes, to the
  convex-set valued maximal operator.  For the theory of the scalar
  maximal operator, which extends to vector-valued functions without
  change, see~\cite{duoandikoetxea01,garcia-cuerva-rubiodefrancia85}.
  We begin with several reductions.  First, by
  Lemma~\ref{lemma:dyadic},
  \[ |MF(x)| \leq C\sum_{\tau\in \{0,\pm 1/3\}^n} |M^\tau F(x)|, \]
  and so it will suffice to prove the strong and weak-type
  inequalities for $M^\tau$.  In fact, given that all of the dyadic
  grids $\D^\tau$ have the same properties as the standard dyadic grid
  $\D$, it will suffice to prove them for the dyadic convex-set valued
  maximal operator, $M^d$.  Moreover, arguing as we did in the proof
  of Lemma~\ref{lemma:M_p-sublinear}, it will suffice
  to prove our estimates for the auxiliary operator
  $\widehat{M}^d$ with omitted convex hull, defined like $\widehat{M}$ in Lemma \ref{lemma:M_p-sublinear}, but only using dyadic cubes.

  First note that by Proposition~\ref{prop:avg-op}, for a.e. $x\in
  \R^n$,  $|A_QF(x)|\leq
  \|F\|_{L^\infty_\K(\R^n,|\cdot|)}$, and so we have that
  $\|\widehat{M}^d F\|_{L^\infty_\K(\R^n,|\cdot|)} \leq
  \|F\|_{L^\infty_\K(\R^n,|\cdot|)}$.

 We will now prove the weak $(1,1)$ inequality by adapting the
 Calder\'on-Zygmund decomposition to convex-set valued functions.  Fix
 $\lambda>0$ and define
 \[ \Omega_\lambda^d = \{ x\in \R^n : |\widehat{M}^dF(x)|>\lambda \}. \]
 If $\Omega_\lambda^d$ is empty, there is nothing to prove.
 Otherwise, given $x\in \Omega_\lambda^d$, there must exist a cube
 $Q\in \D$ such that $x\in Q$ and 
 \[ \bigg| \avgint_Q F(y)\,dy \bigg| > \lambda.  \]
 We claim that among all the dyadic cubes containing $x$, there must
 be a largest one with this property.  Arguing as we did above, we
 have that
 \[ \bigg| \avgint_Q F(y)\,dy \bigg| \leq \avgint_Q |F(y)|\,dy
   \leq \ln(Q)^{-1}\|F\|_{L^1_\K(\R^n,|\cdot|)}.  \]
 Since the right-hand side goes to $0$ as $\ln(Q)\rightarrow \infty$, we
 see that such a maximal cube must exist.   Denote this cube by
 $Q_x$.  Since the set of dyadic cubes is countable, we can enumerate
 the set $\{Q_x : x \in \Omega_\lambda^d \}$ by
 $\{Q_j\}_{j\in \N}$.   The cubes $Q_j$ must be disjoint, since if
 one was contained in the other, it would contradict the maximality.
By our choice of these cubes, $\Omega_\lambda^d \subset \bigcup_j
Q_j$.  Hence, we have that
 \begin{multline*}
  \ln(\Omega_\lambda^d) \leq \sum_j \ln(Q_j)
 \leq
     \frac{1}{\lambda} \sum_j \ln(Q_j) \bigg| \avgint_{Q_j} F(y)\,dy \bigg| \\
    \leq \frac{1}{\lambda} \sum_j\int_{Q_j} |F(y)|\,dy
   \leq \lambda^{-1}\|F\|_{L^1_\K(\R^n,|\cdot|)}.
 \end{multline*}
 
To complete the proof, fix $1<p<\infty$.  For each $\lambda>0$ we can
decompose $F=F^\lambda_1+F^\lambda_2$, where
\[  F^\lambda_1(x) = F(x)\chi_{ \{x\in \R^n : |F(x)|>\lambda/2\}},
    \quad
    F^\lambda_2(x) = F(x)\chi_{ \{x\in \R^n :
      |F(x)|\leq\lambda/2\}}.  \]
    Since the operator $\widehat{M}^d$ is bounded on $L^\infty_\K(\R^n,|\cdot|)$,
    by Lemma~\ref{lemma:Mf-sublinear},
    \[ |\widehat{M}^dF(x)| \leq |\widehat{M}^dF^\lambda_1(x)|+|\widehat{M}^dF^\lambda_2(x)|
      \leq |\widehat{M}^dF^\lambda_1(x)| +\lambda/2. \]
    Therefore, by the weak $(1,1)$ inequality and Fubini's theorem,
    \begin{align*}
      \|\widehat{M}^dF(x)\|_{L^p_\K(\R^n,|\cdot|)}^p
  &    \le p \int_0^\infty \lambda^{p-1} \ln(\{ x\in\R^n : 
      |\widehat{M}^dF^\lambda_1(x)|>\lambda/2\})\,d\lambda \\
&      \lesssim
      p \int_0^\infty \lambda^{p-2} \int_{\{ x\in \R^n :
        |F(x)|>\lambda/2\}} |F(x)|\,dx\,d\lambda \\
&      = p \int_{\R^n} |F(x)| \int_0^{2|F(x)|}
      \lambda^{p-2}\,d\lambda\,dx \\
&      = 2^{p-1} p' \|F\|_{L^p_\K(\R^n,|\cdot|)}^p,
    \end{align*}
where $1/p+1/p'=1$.
 \end{proof}

 \medskip

 Even though $L^p_\K(\R^n,|\cdot|)$ is not a normed vector space, the
maximal operator is still continuous.

\begin{corollary} \label{cor:max-continuous}
 For $1<p\leq \infty$, the maximal operator is continuous on $L^p_\K(\R^n,|\cdot|)$ with respect to the metric \eqref{metric}.
\end{corollary}

\begin{proof}
Let $d_H$ denote the Hausdorff distance defined
by~\eqref{eqn:hausdorff-dist-func} with respect to the Euclidean
metric.  Given compact sets $F,\,G\in \R^d$, if $d_H(F,G)\leq r$, then
it follows at once from the definition that
$F \subset G + r\overline{{\mathbf B}}$ and  $G \subset F
  +r\overline{{\mathbf B}}$.

Fix a sequence $\{F_n\}_{n\in\N}$ that converges to $F$ in
$L^p_\K(\R^n,|\cdot|)$.   For each $n\in \N$ define
\[ H_n(x) = d_{H}(F_n(x),F(x))\overline{{\mathbf B}}.  \]
Then $\|H_n\|_{L^p_\K(\R^n,|\cdot|)}\rightarrow 0$ as $n\rightarrow
\infty$, and
\[ F(x) \subset F_n(x) + H_n(x), \qquad F_n(x) \subset F(x) +
  H_n(x). \]
by Proposition~\ref{lemma:Mf-sublinear}  the maximal operator is sublinear,
so we have that
\[ MF(x) \subset MF_n(x) + MH_n(x), \qquad MF_n(x) \subset MF(x) +
  MH_n(x). \]
Therefore, by Theorem~\ref{thm:M-norm-ineq},
\[ d_p(MF,MF_n) \leq d_p(MF_n+MH_n,MF_n) =
  \|MH_n\|_{L^p_\K(\R^n,|\cdot|)} \leq C
  \|H_n\|_{L^p_\K(\R^n,|\cdot|)}.  \]
The desired conclusion follows at once.
\end{proof}

\begin{remark}
 The proof of Corollary~\ref{cor:max-continuous} is not specific to
 the maximal operator:  in fact, we have that any linear or sublinear
 operator that is bounded on $L^p_\K(\R^n,|\cdot|)$ is continuous.
\end{remark}

 \section{Matrix $\A_p$ weights and weighted norm inequalities}
 \label{section:matrix-weights}
 
In this section we extend Theorem~\ref{thm:M-norm-ineq} to the spaces
$L^p_\K(\R^n,\rho)$, where the norm function $\rho$ satisfies a
generalized Muckenhoupt $\A_p$ condition.  To prove our results, we first need to
develop the theory of $\A_p$ norms.  Throughout this
section, let $\rho : \R^n \times \R^d \rightarrow [0,\infty)$ be a
norm function, such that if $\rho_x(v)=\rho_{F(x)^\circ}(v)$ as in
\eqref{mc1}, then $F$ is locally integrably bounded.   Hence, by
Theorem~\ref{sn}, given any cube $Q$ and $v\in \R^d$,
\[ \int_Q \rho_x(v)\,dx < \infty.  \]

\subsection*{$\A_p$ norms and matrix $\A_p$ weights}
The classical Muckenhoupt $A_p$ condition~\eqref{eqn:Ap-defn} is defined in terms of
averages of scalar weights.  Here we will first define the
corresponding ``average'' of a norm.
Fix $1\leq p<\infty$ and suppose $\rho(\cdot,v)\in L^p_{loc}$ for all
$v\in \R^d$.  We define
$\rho_{p,Q} : \R^d \rightarrow
[0,\infty)$ by
\[ \langle \rho \rangle_{p,Q}(v)
  = \|\rho(\cdot,v)\|_{p,Q}=\bigg(\avgint_Q \rho_x(v)^p\,dx\bigg)^{\frac{1}{p}}. \]
Similarly, if $\rho(\cdot,v)\in L^\infty$ for all
$v\in \R^d$, we define
\[ \langle \rho \rangle_{\infty,Q}(v)
  = \|\rho(\cdot,v)\|_{\infty,Q}= \esssup_{x\in
    Q} \rho_x(v). \]
Since $\|\cdot\|_{p,Q}$, $1\leq p \leq \infty$,  is a norm, it follows
that $\langle \rho \rangle_{p,Q}$ is a norm.  Let $\rho_x^*$ be the dual norm function
and let $\langle\rho^*\rangle_{p,Q}$ be the average of the dual norm
(see Corollary~\ref{cor:dual-norm}).  These are related by the
following inequality.  When $1<p<\infty$,  this was
proved in~\cite[Proposition~1.1]{MR2015733}; for completeness we include the
short proof which immediately extends to $p=1$ and $p=\infty$. 

\begin{lemma} \label{lemma:average-dual}
Given a norm function $\rho : \R^n \times \R^d \rightarrow
[0,\infty)$ and $1 \leq p \leq \infty$, then for every cube $Q$ and $v\in \R^d$, 
\begin{equation} \label{eqn:avg-dual1}
  \langle\rho \rangle_{p,Q}^*(v) \leq \langle \rho^* \rangle_{p',Q}(v). 
\end{equation}
\end{lemma}

\begin{proof}
  Fix $1<p<\infty$.   By H\"older's inequality,
  given two vectors $v,\,w\in \R^d$,
  \begin{multline*}
    |\langle v, w \rangle |
    \leq
    \avgint_Q \rho_x^*(v) \rho_x(w)\,dx \\
    \leq
    \bigg(\avgint_Q \rho_x^*(v)^{p'}\,dx\bigg)^{\frac{1}{p'}}
    \bigg(\avgint_Q \rho_x(w)^{p}\,dx\bigg)^{\frac{1}{p}}
    =
    \langle \rho^*\rangle_{p',Q}(v)\langle \rho\rangle_{p,Q}(w). 
  \end{multline*}
  The desired inequality now follows by the definition of the dual
  norm.  When $p=1$ or $p=\infty$, we repeat this argument but use the
  $L^\infty$ norm in place of the $L^{p'}$ or $L^p$ norm.  
\end{proof}

An $\A_p$ norm is one for which the reverse of
inequality~\eqref{eqn:avg-dual1} holds.  The following definition and
lemma first appeared in the work of Nazarov, Treil, and Volberg \cite{{MR1428988}, Vol} when
$1<p<\infty$.  Note that our definition of an
$\A_\infty$ norm is different from the one that is given there.

\begin{definition} \label{defn:Ap-defn}
  Given a norm function $\rho : \R^n \times \R^d \rightarrow
  [0,\infty)$, then for $1\leq p \leq \infty$ we say that $\rho\in\A_p$ if
  for every cube $Q$ and $v\in \R^d$, 
  \begin{equation} \label{eqn:Ap-defn1}
    \langle\rho^*\rangle_{p',Q}(v) \lesssim \langle\rho\rangle_{p,Q}^*(v). 
  \end{equation}
  The infimum of the constants which make this inequality true is
  denoted by $[\rho]_{\A_p}$.
\end{definition}

\begin{lemma} \label{lemma:Ap-duality}
  Given $1\leq p \leq \infty$ and norm function $\rho : \R^n \times \R^d \rightarrow
  [0,\infty)$, if $\rho\in \A_p$, then $\rho^* \in \A_{p'}$ and
  $[\rho^*]_{\A_{p'}} =[\rho]_{\A_p}$. 
\end{lemma}

\begin{proof}
It is immediate from the definition of the dual norm that if $p_1$ and
$p_2$ are two norms, and $p_1(v)\leq p_2(v)$ for all $v\in \R^d$, then
$p_2^*(v) \leq p_1^*(v)$.    But then from \eqref{eqn:Ap-defn1} we
have that $\langle \rho\rangle_{p,Q}(v) \leq  [\rho]_{\A_p}\langle\rho^*\rangle_{p',Q}^*(v)$, and since
$\rho^{**}=\rho$, it follows that $\rho^* \in \A_{p'}$ and $[\rho^*]_{\A_{p'}} =[\rho]_{\A_p}$. 
\end{proof}

We can also characterize $\A_p$ norms in terms of their associated
matrices; in doing so, we also give our definition of matrix $\A_p$.
As we noted in the Introduction, our definition is different from, but
equivalent to, the definition used previously when $1<p<\infty$, and
corresponds to replacing the matrix $W$ by $W^p$ in that definition; see \cite{Bow}.
We  give two
characterizations.  To do so, we first define the notion of a reducing
operator.  These were first introduced in~\cite{Vol} for norms; here
we will follow the definition in~\cite{MR2015733} in terms of
matrices.
Given a norm function $\rho$, by Theorem~\ref{thm:matrix-norm} there
exists a positive definite matrix mapping $W : \R^n \rightarrow \Sd$
such that $\rho_x(v) \approx |W(x)v|$.  By
Proposition~\ref{prop:dual-matrix-norm} we have that $\rho^*_x(v)
\approx |W^{-1}(x)v|$.    In both cases the implicit constants depend
only on $d$.  Given a cube $Q$ and $1\leq p\leq \infty$,
$\langle\rho \rangle_{p,Q}$ is also a norm and by the John ellipsoid theorem there exists a matrix $\W_Q^p$ such that for all $v\in
\R^d$,
\[ \langle \rho \rangle_{p,Q}(v)
  \approx \|W(\cdot)v\|_{p,Q} \approx |\W_Q^pv|. \]
The matrix $\W_Q^p$ is referred to as the reducing operator associated
to $\rho$ on $Q$.   For the reducing operators associated to the dual
norm we will use the notation $\barW_Q^p$:  i.e.,
\[ \langle\rho^*\rangle_{p,Q}(v)
  \approx \|W^{-1}(\cdot)v\|_{p,Q} \approx |\barW_Q^pv|. \]

\begin{prop} \label{prop:Ap-reducing} Given a norm function
  $\rho : \R^n \times \R^d \rightarrow [0,\infty)$ with associated
  matrix mapping $W$, and given $1\leq p\leq \infty$, $\rho \in \A_p$
  if and only if 
 \begin{equation} \label{eqn:Ap-reducing1}
[W]^R_{\A_p} = \sup_Q |\barW_Q^{p'} \W_Q^p|_{\op} < \infty. 
\end{equation}
Moreover, $[W]^R_{\A_p} \approx [\rho]_{\A_p}$ with implicit constants
that depend only on $d$. 
\end{prop}

\begin{proof}
  Suppose first that \eqref{eqn:Ap-reducing1} holds.  Then given any
cube $Q$ and  vector $v\in \R^d$,
\[ \langle\rho^* \rangle_{p',Q}(v)
  \approx
  |\barW_Q^{p'}v|
  =
  |\barW_Q^{p'}\W_Q^p (\W_Q^p)^{-1}v|
  \leq
  |\barW_Q^{p'}\W_Q^p|_{\op}|(\W_Q^p)^{-1}v|
  \leq
  [W]^R_{\A_p} \langle\rho\rangle_{p,Q}^*(v). \]
Hence, $\rho \in \A_p$.  

Conversely, if $\rho \in \A_p$, then given any vector $v\in \R^d$,
\[  |\barW_Q^{p'}\W_Q^pv|
  \approx
  \langle \rho^*\rangle_{p',Q}(\W_Q^p v)
  \leq [\rho]_{\A_p}
  \langle \rho \rangle_{p,Q}^*(\W_Q^p v)
  \approx
  |(W_Q^p)^{-1}\W_Q^p v|
  =
  |v|. \]
It follows at once that \eqref{eqn:Ap-reducing1} holds and the
constants are comparable.
\end{proof}

If a matrix mapping $W$ is such that $\rho_W$ is an $\A_p$ norm, we
say that  $W$ is in matrix $\A_p$, and write $W\in \A_p$.  Note that
it follows immediately from Proposition~\ref{prop:Ap-reducing}, analogous to
Lemma~\ref{lemma:average-dual}, that $W\in \A_p$ if and only if $W^{-1}
\in \A_{p'}$.

We can give another characterization of matrix $\A_p$ using integral averages
that strongly resembles the Muckenhoupt $A_p$ condition for scalar
weights.  When $1<p<\infty$, this condition is due to
Roudenko~\cite{MR1928089}; when $p=1$ it was used as the definition of
$\A_1$ by Frazier and Roudenko~\cite{MR2104276}.    Here we 
give the proof when $p=1$ (equivalently, when $p=\infty$) and refer
the reader to~\cite{MR1928089} for the case $1<p<\infty$.

\begin{prop} \label{prop:Ap-matrix-defn}
Given a norm function
  $\rho : \R^n \times \R^d \rightarrow [0,\infty)$ with associated
  matrix mapping $W$, and given $1< p< \infty$, $\rho \in \A_p$
  if and only if
  \[ [W]_{\A_p} = \sup_Q \bigg(\avgint_Q \bigg( \avgint_Q
    |W(x)W^{-1}(y)|_{\op}^{p'}\,dy\bigg)^{\frac{p}{p'}}\,dx\bigg)^{\frac{1}{p}} <
    \infty. \]
  When $p=1$, $\rho\in \A_1$ if and only if
  \begin{equation} \label{eqn:A1-matrix1}
    [W]_{\A_1} = \sup_Q \esssup_{x\in Q} \avgint_Q |W^{-1}(x)W(y)|_{\op}\,dy <
    \infty.  
  \end{equation}
  When $p=\infty$, $\rho\in \A_\infty$ if and only if
  \[ [W]_{\A_\infty} =\sup_Q \esssup_{x\in Q} \avgint_Q |W(x)W^{-1}(y)|_{\op}\,dy <
    \infty.  \]
  For all $p$, $[W]_{\A_p} \approx [W]_{\A_p}^R \approx [\rho]_{\A_p}$
  with constants that depend only on $d$.  
\end{prop}

\begin{proof}
Recall that if $A$ and $B$ are two matrices in $\Sd$, then
  \[ |AB|_{\op} = |(AB)^t|_{\op} = |B^tA^t|_{\op} = |BA|_{\op}. \]

  Suppose first that $\rho \in \A_1$.  Let $\{e_i\}_{i=1}^d$ be
the standard basis in $\R^d$.  Fix a cube $Q$; then for almost
  every $x\in Q$,
  \begin{multline*}
    \avgint_Q |W^{-1}(x)W(y)|_{\op} \,dy
     =
      \avgint_Q |W(y)W^{-1}(x)|_{\op}\,dy 
     \approx
      \sum_{i=1}^d \avgint_Q |W(y)W^{-1}(x)e_i|\,dy \\
     \approx \sum_{i=1}^d |\W_Q^1 W^{-1}(x)e_i| 
     \approx |\W_Q^1 W^{-1}(x)|_{\op} 
         = |W^{-1}(x)\W_Q^1 |_{\op} \\
 \lesssim \sum_{i=1}^d  \esssup_{x\in Q} |W^{-1}(x)\W_Q^1 e_i| 
     \lesssim \sum_{i=1}^d |\barW_Q^\infty \W_Q^1 e_i| 
       \approx |\barW_Q^\infty \W_Q^1|_{\op} < \infty;
     \end{multline*}
  the last inequality follows from Proposition~\ref{prop:Ap-reducing}.
This gives us  inequality~\eqref{eqn:A1-matrix1}.

  \medskip

  Conversely, suppose ~\eqref{eqn:A1-matrix1} holds.   If we fix a cube $Q$, then there exists a
  vector $v\in \R^d$, $|v|=1$, and $x\in Q$ such that
  \begin{multline*}
    |\barW_Q^\infty \W_Q^1|_{\op}
     \lesssim |\barW_Q^\infty \W_Q^1v| 
     \lesssim |W^{-1}(x)\W_Q^1v| 
     \leq |W^{-1}(x)\W_Q^1|_{\op} 
     = |\W_Q^1 W^{-1}(x)|_{\op} \\
     \lesssim \sum_{i=1}^d |\W_Q^1 W^{-1}(x)e_i| 
     \approx \sum_{i=1}^d \avgint_Q |W(y) W^{-1}(x)e_i|\,dy \\
     \lesssim \avgint_Q |W(y) W^{-1}(x)|_{\op}\,dy 
     \lesssim \avgint_Q |W^{-1}(x) W(y) |_{\op}\,dy < \infty. 
   \end{multline*}
  So again by Proposition~\ref{prop:Ap-reducing}, $\rho\in \A_1$ and
  the constants are comparable.
\end{proof}

    \subsection*{Weighted norm inequalities for
      averaging and maximal operators}
    In this section we generalize Proposition~\ref{prop:avg-op} and
    Theorem~\ref{thm:M-norm-ineq} to the case of matrix weights.  

    \begin{prop} \label{prop:avg-op-wts}
     Given $1\leq p \leq \infty$ and a matrix weight $W$, the
     following are equivalent:
     \begin{enumerate}

     \item $W \in \A_p$.
       
      \item Given any cube $Q$, $A_Q : L^p_\K(\R^n,W) \rightarrow L^p_\K(\R^n,W)$, and
        $\|A_Q\|_{L^p_\K(\R^n,W)} \leq K$.

      \end{enumerate}

      Moreover, we have that $[W]_{\A_p} \approx \sup_Q \|A_Q\|_{L^p_\K(\R^n,W)}$.  
    \end{prop}

    \begin{proof}
      We first assume $W\in \A_p$.  When $1\leq p<\infty$, this result
      was 
    originally proved for vector-valued functions
    in~\cite[Proposition~4.7]{MR3544941}, but
    the proof readily extends to convex-set valued functions.  Here we
    prove the case when $p=\infty$.
      Fix $W\in \A_\infty$ and $F \in L_\K^\infty(\R^n, W)$.  Given a cube $Q$,
      for almost every $x\in Q$,
      \begin{align*}
        |W(x)A_QF(x)|
        & = \sup\bigg\{ \bigg|W(x) \avgint_Q f(y)\,dy \bigg| : f \in
          S^1(Q,F) \bigg\} \\
        & \leq  \sup\bigg\{ \avgint_Q |W(x)W^{-1}(y)W(y)f(y)|\,dy  : f \in
          S^1(Q,F) \bigg\} \\
        & \lesssim \sup \big\{ [W]_{\A_\infty} \|Wf\|_\infty :  f \in
          S^1(Q,F) \big\} \\
        & = [W]_{\A_\infty}\|F\|_{L_\K^\infty(\R^n,W)}. 
      \end{align*}

      To prove necessity, first note that it follows at once from the
      mapping from vector-valued functions to convex-set valued
      functions given in
      Lemma~\ref{lemma:vector-int}, that to prove necessity it
      suffices to prove it for vector-valued functions.  This was
      proved when $1\leq p<\infty$ in~\cite[Theorem~1.18]{MR3803292}.
      The proof for $1<p<\infty$ immediately extends to the case
      $p=\infty$, using the fact that the dual of $L^1$ is $L^\infty$.
    \end{proof}

    As a corollary to Proposition~\ref{prop:avg-op-wts} we deduce that if
    $W\in \A_p$, then the operator norm $|W|_{op}$ is a scalar weight in
    $\A_p$.  This was proved by
    Goldberg~\cite[Corollary~2.3]{MR2015733} for $p<\infty$ and the same
    proof holds for $p=\infty$.  We omit the details.

    \begin{corollary} \label{cor:eigenvalues-scalar-Ap}
      For $1\leq p \leq \infty$, if $W\in \A_p$ and $w=|W(\cdot)|_{op}$ is an
      operator norm of $W$, then $w\in \A_p$ with $[w]_{\A_p} \lesssim
      [W]_{\A_p}$. 
    \end{corollary}

    \medskip

    To prove norm inequalities for the convex-set valued maximal operator,
    we need an auxiliary weighted maximal operator first introduced by
    Christ and Goldberg~\cite{CG,MR2015733}.  Given a matrix weight
    $W$,  for any function $f\in L^1_{loc}(\R^n, \R^d)$ define
    \[ M_W f(x) = \sup_Q \avgint_Q |W(x)W^{-1}(y)f(y)|\,dy \cdot
      \chi_Q(x). \]

     \begin{prop} \label{prop:max-p-infty-weights}
      Fix $1< p\leq \infty$.  Given a matrix weight $W\in \A_p$, $M_W : L^p(\R^n,\R^d)
      \rightarrow L^p(\R^n)$.   Moreover,
      \[ \|M_Wf \|_{L^p(\R^n)}
        \leq C(n,d,p)[W]_{\A_p}^{p'} \|f\|_{L^p(\R^n,\R^d)}. \]
  \end{prop}

  \begin{proof}
    For $1<p<\infty$, this inequality, without a quantitative estimate of the
    constant, was proved in~\cite{CG,MR2015733}.  The given estimate
    was proved by Isralowitz and Moen~\cite[Theorem~1.3]{MR4030471}.  
We will prove the case when $p=\infty$. 

Give a vector function $f\in  L^\infty(\R^n,\R^d)$, then for  any
    cube $Q$ and a.e.~$x\in Q$, we have by Proposition~\ref{prop:Ap-matrix-defn} that
    \[ \avgint_Q |W(x)W^{-1}(y)f(y)|\,dy
      \leq
      \avgint_Q |W(x)W^{-1}(y)|_{\op} |f(y)|\,dy
      \lesssim [W]_{\A_\infty} \|f\|_\infty. \]
    If we now fix $x$ and take the supremum over all cubes containing $x$,
    we get the desired estimate.
  \end{proof}

  \begin{theorem} \label{thm:wtd-convex-max-bounds}
    Given $1<p\leq \infty$ and a matrix weight $W\in \A_p$, then the
    convex-set valued maximal operator satisfies $M : L^p_\K(\R^n,W) \rightarrow
    L^p_\K(\R^n,W)$.   Moreover, 
 \[ \|MF \|_{L^p_\K(\R^n,W)}
        \leq C(n,d,p)[W]_{\A_p}^{p'} \|F\|_{L^p_\K(\R^n,W)}. \]
  \end{theorem}

  \begin{proof}
    We will prove this by reducing to the corresponding inequalities
    for $M_W$.  First note that by replacing $F$
    by $W^{-1}F$, we have that $M : L^p_\K(\R^n,W) \rightarrow
    L^p_\K(\R^n,W)$ is bounded if and only if 
    \[
    ||WM(W^{-1}F)||_{L^p_\K (\R^n,|\cdot|)} \lesssim ||F||_{L^p_\K (\R^n,|\cdot|)}.
    \]

Given $F\in L^p_\K(\R^n,
|\cdot|)$, by Theorem~\ref{mj} there exists a measurable matrix
map $A : \R^n \rightarrow \Md$ such that 
\[ A(x)\overline{{\mathbf B}} \subset F(x) \subset \sqrt{d} A(x) \overline{{\mathbf B}}. \]
Let $a_i(x)$, $1\leq i \leq d$, be the columns of $A(x)$.  Then
$a_i(x)\in F(x)$, and conversely, if $v\in F(x)$, 
\[ v = \sum_{i=1}^d \lambda_i a_i(x), \quad \text{where} \quad
\sum_{i=1}^d |\lambda_i| \leq \sqrt{d}  \bigg( \sum_{i=1}^d |\lambda_i|^2 \bigg)^{1/2} \le d.  \]
Since $F \in L^p_\K(\R^n,|\cdot|)$, $F$ is locally integrably bounded,
and so $a_i \in L^1_{loc}(\R^d)$.  Define 
\[ F_i(x) =\clconv\{ a_i(x),-a_i(x)\}.  \]
Then by Lemma~\ref{lemma:vector-int}, $F_i$ is measurable and locally
integrably bounded, and by the above estimate,
\begin{equation} \label{eqn:max-vec0}
 F(x) \subset C(d) \sum_{i=1}^d F_i(x). 
\end{equation}
Hence, by Lemma~\ref{lemma:Mf-sublinear},
\begin{equation} \label{eqn:max-vec1}
W(x) M(W^{-1}F)(x) \subset C(d)\sum_{i=1}^d W(x)M(W^{-1}F_i)(x).  
\end{equation}

Again by Lemma~\ref{lemma:vector-int}, for $1\leq i \leq d$ and any
cube $Q$ containing $x$, 
\begin{multline} \label{eqn:max-vec2}
 |W(x)A_Q(W^{-1}F_i)(x)| =
  \sup\bigg\{ \bigg|W(x)\avgint_Q k(y)W^{-1}(y)a_i(y)\,dy \bigg| :  \|k\|_\infty
  \leq 1 \bigg\} \\
  \leq \avgint_Q |W(x)W^{-1}(y)a_i(y)|\,dy
  \leq M_Wa_i(x). 
\end{multline}
Therefore, we have that $|W(x)M(W^{-1}F_i)(x)|\leq M_Wa_i(x)$.  But then by
Proposition~\ref{prop:max-p-infty-weights}, 
\begin{align}\label{eqn:max-vec3}
  \|WM(W^{-1}F)\|_{ L^p_\K(\R^n,|\cdot|)}
&  \leq C(d) \sum_{i=1}^d \|WM(W^{-1}F_i) \|_{ L^p_\K(\R^n, |\cdot|)} \\
&  \leq C(d) \sum_{i=1}^d \|M_Wa_i\|_{ L^p(\R^n,\R)} \notag \\
&  \leq C(n,d,p) [W]_{\A_p}^{p'}\sum_{i=1}^d \|a_i\|_{ L^p(\R^n,\R^d)} \notag \\
&  \leq C(n,d,p) [W]_{\A_p}^{p'}\|F\|_{ L^p_\K(\R^n,|\cdot|)}. \notag 
\end{align}
\end{proof}

  \begin{remark}
    Recently, it was shown in~\cite{DCU-JI-KM-SP-IRR} that for $p=1$,
    if $W\in \A_1$, then
    $M_W : L^1 (\R^n,\R^d) \rightarrow L^{1,\infty}(\R^n)$ with a
    constant proportional to $[W]_{\A_1}^2$.  The
    above proof can be modified to show that
    \begin{equation} \label{eqn:wtd-convex-max-bound1}
 \ln(\{ x\in \R^n : |W(x)M(W^{-1}F)(x)|> \lambda \})
      \leq \frac{C}{\lambda} [W]_{\A_1}^2\int_{\R^n} |F(x)|\,dx.
    \end{equation}
  \end{remark}
  
  
\section{Convex-set valued $\A_1$ and the Rubio de Francia iteration
  algorithm}
\label{section:rubio}

In this section we give a new characterization of the matrix $\A_1$ condition that is a
close analog of the classical Muckenhoupt $A_1$ condition.  We then
use this to define a convex-set valued version of the Rubio de Francia iteration algorithm.  

\subsection*{Convex-set valued $\A_1$}
Given a locally integrably bounded function $F: \R^n \to \bcs$, we
showed in Lemma~\ref{lemma:Mf-dominates-F} that $F(x)\subset MF(x)$
almost everywhere.   This motivates the following definition which
adapts the definition of $A_1$ in the scalar case.

\begin{definition} \label{defn:convex-A1}
Given a locally integrably bounded function $F: \R^n \to \bcs$, we say
that $F$ is in convex-set valued $ \A_1^\K$,  if there exists a constant $C$ such that for almost every
$x$,
\[  MF(x) \subset C F(x).  \]
Denote the infimum of all such constants by $[F]_{\A_1^\K}$.
\end{definition}

There is an alternative characterization of convex-set valued $\A_1$ in
terms of averaging operators.

\begin{lemma} \label{lemma:alt-A1-defn}
  Given a locally integrably bounded function $F: \R^n \to \bcs$,
  $F\in \A_1^\K$ if and only if there exists a constant $C$ such that
  for every cube $Q$ and almost every $x\in Q$,
  \begin{equation} \label{eqn:alt-A1-defn1}
\avgint_Q F(y)\,dy \subset CF(x). 
\end{equation}
The infimum of all such constants equals $[F]_{\A_1^\K}$.
\end{lemma}

\begin{proof}
  One direction is immediate:  if $F\in \A_1^\K$, then  for every cube
  $Q$ and almost every $x\in Q$,
  \[ \avgint_Q F(y)\,dy \subset MF(x)\subset [F]_{\A_1^\K} F(x). \]

  Conversely, suppose that \eqref{eqn:alt-A1-defn1} holds.  Recall
  that $\Q$ is the countable collection of  cubes whose vertices have
  rational coordinates.  For each $P\in \Q$, let $E_P$ be the set of
  $x\in P$ such that~\eqref{eqn:alt-A1-defn1} does not hold.  If we
  define
  \[ E= \bigcup_{P\in \Q} E_P, \]
  then $\ln(E)=0$.  Fix $x\not\in E$. Then
  \[ \bigcup_{\substack{P\in \Q\\x\in P}} \avgint_P F(y)\,dy \subset
    CF(x). \]
  Since $F\in \K_{bcs}$, the closed convex hull of the left-hand side
  is also contained in $CF(x)$.  Therefore, by
  Proposition~\ref{prop:MF-measurable}, $MF(x) \subset CF(x)$.
 This, together with the above estimate, shows that the infimum of all
 such constant $C$ must be $[F]_{\A_1^\K}$. 
\end{proof}

There is a one-to-one correspondence between $\A_1^\K$ and matrix
$\A_1$ weights.  To prove this  we 
use the characterization of (locally) integrably bounded convex-set valued
mappings from Theorem~\ref{sn}.

\begin{theorem} \label{thm:matrix-convex-A1}
  Given a convex-set valued function $F: \R^n \to \abcs$ that is locally
  integrably bounded,  $F\in \A_1^\K$ if and only if the norm function
  $\rho(x,v)=p_{F(x)^\circ}(v)$ satisfies $\rho \in \A_1$.   Moreover,
  $[F]_{\A_1^\K} \approx [\rho]_{\A_1}$ with implicit constants that
depend only on $d$. 
\end{theorem}

\begin{proof}
  Let $\rho : \R^n \times \R^d \rightarrow [0,\infty)$ be a norm
  function; then by Theorem~\ref{mc}, $\rho(x,v)=p_{F(x)^\circ}$,
  where  $F : \R^n \rightarrow \bcs$ is the measurable mapping
  \[ F(x) = \{ x\in \R^n : \rho_x(v) \leq 1\}^\circ.  \]
  Conversely, given $F$ we can define the norm function $\rho$ in this
  way.
  By Theorem~\ref{sn}, $F$ is locally integrably bounded if and only if
  for every cube $Q$ and $v\in \R^d$, the norm
  \[ \langle \rho \rangle_{1,Q}(v) = \avgint_Q \rho_x(v)\,dx < \infty. \]

  By Theorem~\ref{thm:matrix-norm}, there exists a
  measurable matrix  $W : \R^n \rightarrow \Sd$, positive
  definite almost everywhere, such that $\rho_x(v) \approx |W(x)v|$,
  with constants depending only on $d$.  By
  Proposition~\ref{prop:Ap-matrix-defn}, $\rho$ is an $\A_1$ norm if
  and only if $W$ is in matrix $\A_1$, and
  satisfies~\eqref{eqn:A1-matrix1}, which is equivalent to the
  existence of a constant $C_0$ such that
  \[ \avgint_Q \frac{|W(y)W^{-1}(x)v|}{|v|}\,dy \lesssim C_0<
    \infty \]
  for any cube $Q$, almost every $x\in Q$, and every $v\in \R^d \setminus \{0\}$.  By the
  change of variables $v\mapsto W(x)v$, this is equivalent to
 \begin{equation}\label{fr2}
 \avgint_Q |W(y)v|\,dy \leq C_0|W(x)v|
\end{equation}
for all $v\in \R^d$, which in turn is equivalent to saying that for
every cube $Q$, almost every $x\in Q$, and every $v\in \R^d$,
\begin{equation} \label{eqn:matrix-convex1}
 \langle \rho \rangle_{1,Q}(v) \leq C_1\rho_x(v),
\end{equation}
where $C_1=c(d)C_0$.

We will now show that \eqref{eqn:matrix-convex1} is equivalent to the
$\A_1^\K$ condition for $F$.  By Theorem~\ref{sn}, for every cube
$Q$, we have that $\langle \rho \rangle_{1,Q}(v)= p_{K_Q}(v)$, where
\[ K_Q = \bigg(\avgint_Q F(y)\,dy \bigg)^\circ.  \]
On the other hand, we have that
\[ K_Q =
  \bigg \{v \in \R^d:  \avgint_Q \rho_y(v) dy \le 1 \bigg\}, \]
and so \eqref{eqn:matrix-convex1} is equivalent to the inclusion
\[ \bigg(\avgint_Q F(y)\,dy \bigg)^\circ = K_Q \supset  \{v \in \R^d:
  C_1 \rho_x(v) \le 1 \} = C_1^{-1}F(x)^\circ = \big(C_1F(x)\big)^\circ. \]
If we take the polar of the sets we reverse the inclusion, so this is
equivalent to
\[ \avgint_Q F(y) dy \subset C_1 F(x), \]
and by Lemma~\ref{lemma:alt-A1-defn} this is equivalent to $F\in
\A_1^\K$. 
Therefore, we have that $\rho$ is an $\A_1$ norm if and only if $F$ is
locally integrably bounded and in $\A_1^\K$.   By taking the infima of
the respective constants we see that $[F]_{\A_1^\K} \approx [\rho]_{\A_1}$.
\end{proof}

\begin{corollary} \label{cor:matrix-norm-A1}
  Given a matrix weight $W$, $W\in \A_1$ if and only if $W\overline{\mathbf{B}}
  \in \A_1^\K$. 
\end{corollary}

\begin{proof}
  Let $F=W\mathbf{B}$.  By Theorem~\ref{thm:matrix-convex-A1}, $F \in
  \A_1^\K$ if and only if $\rho\in A_1$, where $\rho(x,v)=
  p_{F(x)^\circ}(v)$.   We compute $\rho$ explicitly: if we  argue as in
  the proof of Proposition~\ref{prop:dual-matrix-norm}, $F(x)^\circ =
  W^{-1}(x)\overline{\mathbf{B}}$, and so $ p_{F(x)^\circ}(v)=\rho_W(x)$.  
\end{proof}

\subsection*{The Rubio de Francia iteration algorithm}
Our goal now is to show that the Rubio de Francia iteration
algorithm~\cite[Chapter~2]{MR2797562}
can be extended to the convex-set valued maximal operator.  Given
$1<p\leq \infty$, suppose that $\rho \in \A_p$.  Let
$\|M\|_\rho= \|M\|_{L^p_\K(\R^n,\rho)}$ denote the norm of the
convex-set valued maximal operator on $L^p_\K(\R^n,\rho)$: that is,
the infimum of all constants $C$ such that
$\|MF\|_{L^p_\K(\R^n,\rho)} \leq C\|F\|_{L^p_\K(\R^n,\rho)}$.

Given $G \in L^p_\K(\R^n,\rho)$ we formally define the Rubio de
Francia iteration algorithm to be the sum
\begin{equation}  \label{eqn:rdf0}
\Rdf G(x) = \sum_{k=0}^\infty 2^{-k}\|M\|_\rho^{-k} M^kG(x), 
\end{equation}
where $M^kG = M\circ M \circ \cdots \circ MG$ for $k\geq 1$ and
$M^0G(x)=G(x)$.   We can show that this series converges to a
convex-set valued function that has exactly the same properties as in
the scalar setting.

\begin{theorem} \label{thm:Rdf-iteration}Suppose that $\rho$ is an
  $\A_p$ norm for some $1<p \leq\infty$.  Fix
  $G \in L^p_\K(\R^n,\rho)$ and define $\Rdf G$ by \eqref{eqn:rdf0};
  then this series converges in $L^p_\K(\R^n,\rho)$ and
  $\Rdf G : \R^n \rightarrow \bcs$ is a measurable mapping.  Moreover,
  it has the following properties:
\begin{enumerate}

\item $G(x) \subset \Rdf G(x)$;

\item $\|\Rdf G\|_{L^p_\K(\R^n,\rho)} \leq 2\|G\|_{L^p_\K(\R^n,\rho)}$;

\item $\Rdf G \in \A_1^\K$, and
  $M(\Rdf G)(x) \subset 2\|M\|_\rho \Rdf G(x)$.

\end{enumerate}
\end{theorem}

Since variants of the iteration algorithm will play a central role in
subsequent sections, we are instead going to prove a more general
result which has Theorem~\ref{thm:Rdf-iteration} as an immediate
corollary using Lemma~\ref{lemma:Mf-sublinear} and
Theorem~\ref{thm:wtd-convex-max-bounds}.

\begin{theorem} \label{thm:gen-rdf-algorithm}
  Fix $1\leq p \leq \infty$ and a norm function $\rho$.  Suppose $T$
  is a convex-set valued operator with the following properties:
  \begin{enumerate}

    \item $T : L^p_\K(\R^n,\rho) \rightarrow L^p_\K(\R^n,\rho)$ with
      norm $\|T\|_\rho$.

      \item $T$ is sublinear and monotone in the sense of
        Lemma~\ref{lemma:Mf-sublinear}.

      \end{enumerate}
 Given $G \in L^p_\K(\R^n,\rho)$, define 
\begin{equation}  \label{eqn:rdf1}
S G(x) = \sum_{k=0}^\infty 2^{-k}\|T\|_\rho^{-k} T^kG(x), 
\end{equation}
where $T^kG = T\circ T \circ \cdots \circ TG$ for $k\geq 1$ and
$T^0G(x)=G(x)$.    Then this series converges in $L^p_\K(\R^n,\rho)$ and
  $S G : \R^n \rightarrow \bcs$ is a measurable mapping.  Moreover,
  it has the following properties:
\begin{enumerate}

\item $G(x) \subset S G(x)$;

\item $\|SG\|_{L^p_\K(\R^n,\rho)} \leq 2\|G\|_{L^p_\K(\R^n,\rho)}$;

\item   $T(S G)(x) \subset 2\|T\|_\rho S G(x)$.

\end{enumerate}
\end{theorem}

\begin{proof}
  For brevity, in this proof we will denote
  $\|\cdot\|_{L^p_\K(\R^n,\rho)}$ simply as
    $\|\cdot\|_p$.  
  To prove that the series~\eqref{eqn:rdf1} converges in norm, we apply
  Theorem~\ref{thm:complete-metric-space}.  Let
  \[ S_nG(x) = \sum_{k=0}^n 2^{-k}\|T\|_\rho^{-k} T^kG(x) \]
  denote the partial sums of the series $S G$.  We claim that this
  sequence is Cauchy with respect to the metric $d_p$ defined in \eqref{metric}.
  Indeed, if $n>m$, then by~\eqref{lp1} and
  the boundedness of $T$,
  \begin{multline} \label{eqn:rdf2}
    d_p(S_nG,S_mG)
    = d_p\bigg( \sum_{k=m+1}^n
    2^{-k}\|T\|_\rho^{-k} T^kG, \{0\}\bigg) \\
    = \bigg\| \sum_{k=m+1}^n
    2^{-k}\|T\|_\rho^{-k} T^kG \bigg\|_p
    \leq \sum_{k=m+1}^n 2^{-k}\|T\|_\rho^{-k} \|T^kG\|_p
    \leq 2^{-m}\|G\|_p. 
  \end{multline}
  By Theorem~\ref{thm:complete-metric-space},  $ L^p_\K(\R^n,\rho)$ is complete with respect to the metric
  $d_p$, so the sequence $\{S_nG\}_{n\in \N}$ converges.  Let $S
  G$ denote the limit.
  
  We now prove the desired properties.  To prove the first, since
  $d_p(S_n G,S G)\rightarrow 0$, there exists a subsequence such that
  for almost every $x\in \R^n$,
\[ \lim_{j\rightarrow \infty} d_{H,x}(S_{n_j} G(x), S G(x)) =
  0. \]
But for all $n\in \N$, $S_{n-1} G(x) \subset S_n
G(x)$.  Therefore,  we have that $S_n G(x)\rightarrow S G(x)$ in
the Hausdorff metric.  It then follows
from~\cite[Theorem~1.8.7]{MR1216521} that
\begin{equation} \label{eqn:nested-rdf}
 S G(x) = \overline{\bigcup_{n\in \N} S_n G(x)}.  
\end{equation}
Property (1) follows immediately.

  \medskip

To prove the second property, note that for every $n\geq 1$,
\[ \|S G\|_p = d_p(S G,\{0\})
  \leq d(S G, S_nG) + d(S_n G,\{0\})
  = d(S G, S_nG) + \|S_n G\|_p. \]
But then, if we take the limit as $n\rightarrow \infty$ and estimate
the second norm as we did above in \eqref{eqn:rdf2}, we have that
\[ \|S G\|_p
  \leq
  \limsup_{n\rightarrow 0} \big[ d(S G, S_nG) + \|S_n G\|_p \big]
  \leq 2\|G\|_p.  \]

Finally, we show the third property.  For each
$n\geq 1$, we can write
\[ S G(x)  = S_n G(x) + E_n(x), \qquad\text{where } E_n(x)=\sum_{k=n+1}^\infty 2^{-k}\|T\|_\rho^{-k} T^kG(x); \]
by assumption $T$ is sublinear, so
\[ T(S G)(x) \subset T(S_nG)(x) + TE_n(x).  \]
We estimate each term on the right separately.
To estimate the first, we argue as above.  
Since by \eqref{lp1},
\[ d_p(S G, S_nG) = d_p(E_n,\{0\}) = \|E_n\|_p, \]
we have that $\|E_n\|_p \rightarrow 0$ as $n\rightarrow \infty$.
Since $T$ is bounded, $\|TE_n\|_p \rightarrow 0$
as $n\rightarrow \infty$.   Therefore, there exists a subsequence
$\{E_{n_j}\}$ such that
\[ \rho_x (TE_{n_j}(x)) \rightarrow 0 \]
almost everywhere as $j\rightarrow \infty$.    However, the sets $E_n$
are nested, $E_{n+1}(x) \subset E_n(x)$;  since by assumption $T$ is monotone, $TE_{n+1}(x) \subset
TE_n(x)$.  Therefore, we have that for a.e.~$x$
\[ \rho_x (TE_{n}(x)) \rightarrow 0 \]
as $n\rightarrow \infty$.  Hence, given any
$\epsilon>0$, for all $n$ sufficiently large, $TE_n(x) \subset
B(\epsilon, 0)$.   

On the other hand, again since $T$ is sublinear,
\[ T(S_nG)(x) \subset
  \sum_{k=0}^n 2^{-k}\|T\|_\rho^{-k} T^{k+1}G(x)
  \subset 2\|T\|_\rho S_{n+1}G(x)
  \subset 2\|T\|_\rho S G(x).  \]
The last inclusion follows from \eqref{eqn:nested-rdf}.  
Combining these two estimates, we see that for every $\epsilon>0$, 
\[ T(S G)(x) \subset 2\|T\|_\rho S G(x) + B(\epsilon,0).  \]
Since $\epsilon>0$ is arbitrary and since $S G(x)$ is closed, it follows that $T(S G)(x) \subset
2\|T\|_\rho S G(x) $.
\end{proof}

\section{Factorization of matrix weights}
\label{section:factorization}

In this section we prove the Jones factorization
theorem~\cite{preprint-DCU} for matrix weights,
Theorem~\ref{thm:jones-factorization-intro} in the Introduction.  We
restate it here.

\begin{theorem} \label{thm:jones-factorization}
  Fix $1<p<\infty$.  Given a matrix weight $W$, we have $W\in \A_p$ if and
  only if
  \[  W = W_0^{1/p} W_1^{1/p'}, \]
for some commuting  matrix weights  $W_0\in
  \A_1$ and  $W_1 \in \A_\infty$.
\end{theorem}

To make clear the connection with the classical factorization theorem
for scalar $A_p$ weights,
recall that a scalar weight
$w\in \A_p$ if and only if $w^p\in A_p$.  Thus, we can restate the
Jones factorization theorem as $w\in \A_p$ if and only if
$w=w_0^{1/p}w_1^{-1/p'}$, where $w_0,\,w_1\in \A_1$ and so $w_1^{-1}\in \A_\infty$.
Any two scalar weights commute, hence the assumption of commutativity is moot.
In higher dimensions the situation is more complicated. For non-commuting matrix weights $W_0$ and $W_1$, it is necessary to replace the product $W_0^{1/p} W_1^{1/p'} $ by their weighted geometric mean $((W_0)^2 \#_{1/p'} (W_1)^2)^{1/2}$. For that reason Theorem \ref{thm:jones-factorization} splits into two more precise statements 
generalizing the scalar theorem.

\medskip

\subsection{Factorization}
We divide the proof of Theorem~\ref{thm:jones-factorization} into two
propositions.  In the first we prove factorization proper, which
in the scalar case is the more difficult half of the proof.   The
proof is a modification of the proof in the scalar
case~\cite[Theorem~4.2]{preprint-DCU} using the Rubio de Francia
iteration algorithm, which yields matrix weights $W_0$ and $W_1$, which are not only commuting, but also scalar multiples of one another.

\begin{prop} \label{prop:factorization}
     Fix $1<p<\infty$.  Given a matrix weight $W\in \A_p$, there exist
     matrix weights  $W_0$ and $W_1$ such that:
     \begin{itemize} 
     \item
     $W_0\in \A_1$ with $[W_0]_{\A_1} \lesssim  [W]^{p}_{\A_p}$,
\item $W_1 \in \A_\infty$ with  $[W_1]_{\A_\infty}  \lesssim [W]^{p'}_{\A_p}$,
 \item $W_0=rW$, $W_1=sW$ for some
  measurable scalar functions $r,\,s$, and
  \[  W = W_0^{1/p} W_1^{1/p'}. \]
  \end{itemize}
\end{prop}

The proof requires several lemmas which will also be used in the proof of 
extrapolation in Section~\ref{section:extrapolation}.   
The key technical idea is that we replace the
convex-set valued maximal operator with a slightly larger,
ellipsoid-valued maximal operator configured to the matrices.

\begin{definition} \label{defn:Nw}
Let $W$ be an invertible
matrix weight. Given a measurable function $H  :\R^n \rightarrow
  \bcs$,  define the exhausting operator $N_W$ with respect to $W$, which acts on $H$ by 
\[ N_WH(x) = |W(x)H(x)| W(x)^{-1}\overline{\mathbf{B}}.\]
\end{definition}

The following lemma shows that the exhausting operator is sublinear, monotone, and an isometry on $L^p_\K (\R^n,W)$.

\begin{lemma} \label{lemma:ellipsoid-max} Given $1<p<\infty$, a matrix
  $W$, and $H\in L^p_\K (\R^n,W)$, the operator $N_W$ satisfies the following:
  \begin{enumerate}
  \item $H(x) \subset N_WH(x)$. 
  \item $N_W$ is an isometry:
    $\|N_WH\|_{L^p_\K (\R^n,W)}=\|H\|_{L^p_\K (\R^n,W)}$.


    \item $N_W$ is sublinear and monotone in the sense of
      Lemma~\ref{lemma:Mf-sublinear}. 
    \end{enumerate}
  \end{lemma}

  \begin{proof}
    To prove the inclusion, note that if $v\in H(x)$, then $W(x)v\in
    W(x)H(x)$, and so $|W(x)v|\leq |W(x)H(x)|$.  Hence $v\in
    |W(x)H(x)|W^{-1}(x)\overline{\mathbf{B}} =N_WH(x)$.
    
To prove $N_W$ is an isometry, it is enough to observe that for almost every $x\in \R^n$,
    \[ |W(x)N_W H(x)|= \big||W(x)H(x)|\overline{\mathbf{B}} \big| = 
      |W(x)H(x)|. \]
    %


    Finally, to prove that $N_W$ is sublinear, fix $G,\,H \in   L^p_\K
    (\R^n,W)$.  If $w(x) \in
    W(x)G(x)+W(x)H(x)$,  then $|w(x)|\leq |W(x)G(x)|+|W(x)H(x)|$.  Hence,
    $|W(x)(G+H)(x)| \leq |W(x)G(x)|+|W(x)H(x)|$, and so  $N_W(G+H)(x) \subset N_WG(x)+N_WH(x)$.
    Similarly, if  $ \alpha\in \R$, then $N_W(\alpha H)(x)= |\alpha| N_WH(x)= \alpha N_WH(x)$.
  \end{proof}

  In the proof of factorization and extrapolation in the next section,
  we  consider two special classes of convex-set valued functions:  a
  function $G$ is ball-valued if there exists a non-negative scalar
  function $r$ such that $G(x)=r(x)\overline{\mathbf{B}}$.    Similarly, given a
  matrix $W$, $G$ is said to be 
  ellipsoid-valued with respect to $W$ if
  $G(x)=r(x)W(x)\overline{\mathbf{B}}$. 

  \begin{lemma} \label{lemma:ball-iteration} Fix $1\leq p \leq \infty$
    and a norm function $\rho$.  Let $T$ be a convex-set valued
    operator that satisfies the hypotheses of
    Theorem~\ref{thm:gen-rdf-algorithm}, and suppose that if $G\in L^p_\K(\R^n,\rho)$ is a
    ball-valued function, then $TG$ is as well.  If   $S$ is the associated
    iteration operator, then $SG$ is ball-valued.    More generally, if
    whenever $G$ is an ellipsoid-valued function with respect to a
    matrix $W$, then $TG$ is, we have
    that $SG$ is also ellipsoid-valued function with respect to $W$.
  \end{lemma}

  \begin{proof}
    Let $G=r_0W \overline{\mathbf{B}}$ for some scalar function $r_0$ and matrix $W$.   Then by
    induction, we have that for all $k\geq 0$, $T^kG$ is ellipsoid-valued,
    so we have that $T^kG=r_kW\overline{\mathbf{B}}$.   Since the Minkowski sum of
   two ellipsoids of the form $rW\overline{\mathbf{B}}$ and $sW\overline{\mathbf{B}}$ is again
   an ellipsoid of this form, we have, in the notation of
    Theorem~\ref{thm:gen-rdf-algorithm}, that for all $n\in \N$,
    $S^nG$ is an ellipsoid-valued function with respect to $W$.  But then it follows at once
    from~\eqref{eqn:nested-rdf} that $SG$ is ellipsoid-valued with
    respect to $W$.
  \end{proof}
  
  We now define the powers of a ball-valued function.  If
  $G(x)=r(x)\overline{\mathbf{B}}$ is a ball-valued function, for all $t>0$ 
  let $G^t=r^t\overline{\mathbf{B}}$.  The following lemma is an
  immediate consequence of this definition.

\begin{lemma} \label{lemma:ball-monotone}
  If $G$ is a ball-valued function, then for all $t>0$, the
  mapping $G\mapsto G^t$ is monotone.
\end{lemma}

\begin{proof}[Proof of Proposition~\ref{prop:factorization}]
To apply the Rubio de Francia iteration algorithm, we define two
auxiliary operators, $T_1$ and $T_2$.  Let $P_W=N_WM$, where $M$ is
the convex-set valued maximal operator.   Let $q=pp'>1$.  For $G\in
L^q_\K(\R^n, |\cdot|)$ define
\[ T_1G(x) = \big[ W(x) P_W (W^{-1}(N_IG)^{p'})(x)\big]^{1/p'}. \]
Here, $N_I$ is the exhausting operator with respect to the identity matrix $I$.
By the definition of $N_W$ and $N_I$, the two exponents appear on ball-valued
functions and so are well defined.  We claim that $T_1$ satisfies the
hypotheses of Theorem~\ref{thm:gen-rdf-algorithm}.  First, by
Lemma~\ref{lemma:ellipsoid-max} and
Theorem~\ref{thm:wtd-convex-max-bounds},
\begin{multline}\label{T1}
 \int_{\R^n} |T_1G(x)|^q\,dx 
   = \int_{\R^n} |W(x)P_W(W^{-1}(N_IG)^{p'})(x)|^{p}\,dx \\
   \leq C[W]_{\A_p}^{pp'} \int_{\R^n} |N_IG(x)^{p'}|^p\,dx 
   = C[W]_{\A_p}^{q} \int_{\R^n}|G(x)|^q\,dx. 
 \end{multline}
 
 We now prove that $T_1$ is monotone.   By
 Lemmas~\ref{lemma:Mf-sublinear},~\ref{lemma:ellipsoid-max},
 and~\ref{lemma:ball-monotone}, all its component functions are
 monotone, so $T_1$, their composition, is as well.

 To prove that it
 is sublinear, first note that since $T_1F$ is a ball-valued function,
 sublinearity is equivalent to showing that for $G,\,H\in L^q_\K(\R^n,
 |\cdot|)$ and $x\in \R^n$, 
 \[ |T_1(G+H)(x)| \leq |T_1G(x)| + |T_2H(x)|. \]
 To prove this, fix $x$ and define the norm $\rho(v)=|W(x)v|$.  For
 any locally integrably bounded function $F$, by the definition of
 $N_I$ and by (the proof of) Lemma~\ref{lemma:ellipsoid-max},
 \begin{multline*}
   |(W(x) N_W(MF)(x))^{1/p'}|
= \big| |W(x)MF(x)| ^{1/p'} \overline{\mathbf{B}} \big|
   = \big||W(x)MF(x)| \overline{\mathbf{B}} \big|^{1/p'} \\
=  \big|W(x) |W(x)MF(x)| W^{-1}(x)\overline{\mathbf{B}} \big|^{1/p'}
= \rho(N_W(MF)(x))^{1/p'}
   = \rho(MF(x))^{1/p'}. 
 \end{multline*}
Note that $N_I$ produces ball-valued functions and is
 sublinear by Lemma~\ref{lemma:ellipsoid-max}; hence,
 $|(G+H)(y)|\leq |G(y)|+|H(y)|$.  
 Therefore, if we combine these two observations, by Lemma~\ref{lemma:M_p-sublinear},
\begin{align*}
  |T_1(G+H)(x)|
  & = \rho(M(|G(x)+H(x)|^{p'}W^{-1}(x)\overline{\mathbf{B}}))^{1/p'}
  \\
  & \leq
    \rho(M((|G(x)|+|H(x)|)^{p'}W^{-1}(x)\overline{\mathbf{B}}))^{1/p'}
  \\
  & \leq
    \rho(M(|G(x)|^{p'}W^{-1}(x)\overline{\mathbf{B}}))^{1/p'}
    +
    \rho(M(|H(x)|^{p'}W^{-1}(x)\overline{\mathbf{B}}))^{1/p'} \\
  & = |T_1G(x)|+|T_2H(x)|.
\end{align*}

We define $T_2$ similarly:
\[ T_2G(x) = \big[ W^{-1}(x) P_{W^{-1}}(W(N_IG)^{p})(x)\big]^{1/p}. \]
Since $W^{-1} \in \A_{p'}$, the same argument as above shows that
\begin{equation}\label{T2}
  \|T_2 G\|_{L^q_\K(\R^n,|\cdot|)} \leq C[W^{-1}]_{\A_{p'}}
  \|G\|_{L^q_\K(\R^n,|\cdot|)}, 
  \end{equation}
and that $T_2$ is sublinear and
monotone.

Define the operator $T=T_1+T_2$.  Then $T$ satisfies the hypotheses of
Theorem~\ref{thm:gen-rdf-algorithm} with operator norm $\|T\|_{L^q_\K(\R^n,|\cdot|)} \lesssim [W]_{\A_p}$ in light of \eqref{T1} and \eqref{T2}.
Hence, if we define
\[ SG(x) = \sum_{k=0}^\infty 2^{-k} \|T\|_{L^q_\K(\R^n,|\cdot|)}^{-k} T^kG(x), \]
then $\|S G\|_{L^q_\K(\R^n,|\cdot|)}
\leq 2\|G\|_{L^q_\K(\R^n,|\cdot|)}$  and %
\begin{equation} \label{eqn:rev-fact-A1}
  T(SG)(x) \subset2\|T\|_{L^q_\K(\R^n,|\cdot|)}SG(x).  
\end{equation}

Now fix a ball-valued function
$G=r\overline{\mathbf{B}} \in L^q_\K(\R^n,|\cdot|)$.  Then by
Lemma~\ref{lemma:ball-iteration},  $SG=\bar{r}\overline{\mathbf{B}}$.
It follows from~\eqref{eqn:rev-fact-A1} that
$ T_1(SG)(x) \subset2\|T\|_{L^q_\K(\R^n,|\cdot|)}SG(x)$; equivalently,
\[ W(x) P_W(W^{-1} (SG)^{p'})(x) \subset C_1SG(x)^{p'},  \]
where $C_1=2^{p'}\|T\|_{L^q_\K (\R^n,|\cdot|)}^{p'}$.
Define the matrix $W_1(x)= \bar{r}(x)^{-p'}W(x)$.  Then
\[ M(W_1^{-1}\overline{\mathbf{B}})(x) \subset P_W(W^{-1} (SG)^{p'})(x)
  \subset C_1W(x)^{-1}SG(x)^{p'}= C_1W_1(x)^{-1}\overline{\mathbf{B}}. \]
Therefore, $W_1^{-1}\overline{\mathbf{B}} \in \A_1^\K$;   by
Corollary~\ref{cor:matrix-norm-A1}, $W_1^{-1} \in \A_1$, and so $W_1
\in \A_\infty$. Moreover, by \eqref{T1}
\[
[W_1]_{\A_\infty} = [W_1^{-1}]_{\A_1} \lesssim C_1 \lesssim [W]^{p'}_{\A_p}.
\]

We can repeat the above argument, replacing $T_1$ by $T_2$; if we
define $W_0(x)= \bar{r}(x)^pW(x)$, then we get that $W_0\overline{\mathbf{B}}\in
\A_1^\K$, and so $W_0 \in \A_1$. Moreover, by \eqref{T2}
\[
[W_0]_{\A_1} \lesssim  \|T\|_{L^q_\K( \R^n,|\cdot|)}^{p} \lesssim [W]^{p}_{\A_p}
\]
  Finally, we have that
\[ W_0^{1/p}(x) W_1^{1/p'}(x) = \bar{r}(x)W^{1/p}(x)
  \bar{r}(x)^{-1}W^{1/p'}(x)= W(x). \qedhere
  \]
\end{proof}

\subsection{Reverse factorization}
We now prove the so-called ``reverse factorization'' property, that
the product of suitable powers of $\A_1$ and $\A_\infty$ weights is an
$\A_p$ weight.  In the scalar case this is an immediate consequence of
the definitions. However, in the matrix case it is much more difficult since the statement involves a weighted geometric mean of two matrices, while the proof
requires working with norms rather than matrix weights.  To state our
result, recall Definition~\ref{geme}:  given two symmetric, positive
definite matrices, for $0<t<1$, let $ A\#_tB = A^{1/2}(A^{-1/2}BA^{-1/2})^tA^{1/2}$.


\begin{prop} \label{prop:reverse-factorization}
 Suppose that $W_0 \in \A_1$, $W_1\in \A_\infty$, and $1<p<\infty$. Then,
 \[
 \bar{W}= ((W_0)^2\#_{1/p'} (W_1)^2)^{1/2} \in \A_p.
 \]
In particular, if $W_0$ and $W_1$ commute, then $W_0^{1/p}W_1^{1/p'} \in \A_p$.
\end{prop}

The second half of Theorem~\ref{thm:jones-factorization} follows
immediately from Proposition~\ref{prop:reverse-factorization}.  We
will in fact prove a much more general result.

\begin{prop} \label{prop:general-reverse-factorization}
Given $1\leq q_0, q_1\leq \infty$, suppose that $W_0 \in \A_{q_0}$ and $W_1 \in \A_{q_1}$.  Fix $0<t<1$ and define $\bar{W}= ((W_0)^2\#_t (W_1)^2)^{1/2}$.  Then, $\bar{W} \in \A_q$,
  where $\frac{1}{q}=\frac{1-t}{q_0}+\frac{t}{q_1}$.  Moreover,
  \[ [\bar{W}]_{\A_q}
    \leq c(d) [W_0]_{\A_{q_0}}^{1-t} [W_1]_{\A_{q_1}}^{t}. \]
\end{prop}

Proposition~\ref{prop:reverse-factorization} follows at once from this
if we take $q_0=1$, $q_1=\infty$, and $t=1/p'$.    Beyond its
intrinsic interest, we prove
Proposition~\ref{prop:general-reverse-factorization} because the
following corollary, which again follows at once by the correct choice
of $q_0$ and $q_1$, plays an important role in the proof of
extrapolation in Section~\ref{section:extrapolation}.  In the scalar
case this result is used to prove sharp constant
extrapolation and is due to
Duoandikoetxea~\cite[Lemma~2.1]{10.1016/j.jfa.2010.12.015} (see
also~\cite[Theorem~3.22]{MR2797562}).

\begin{corollary} \label{cor:Duo-sharp-factorization}
Given $1<p<\infty$ and $W\in \A_p$, suppose there exists a scalar
function $s$ such that $W_1 =sW\in \A_\infty$.  Then for
$p<p_0<\infty$, $\bar{W}=W^{p/p_0}W_1^{1-p/p_0} \in \A_{p_0}$;
moreover, 
  \[ [\bar{W}]_{\A_{p_0}}
    \leq c(d) [W]_{\A_{p}}^{p/p_0} [W_1]_{\A_{\infty}}^{1-p/p_0}. \]
Similarly, if there exists a scalar function $r$ such that $W_0=rW\in
\A_1$, then for $1<p_0<p$, $\bar{W}=W_0^{1-p'/p_0'}W^{p'/p_0'} \in \A_{p_0}$;
moreover, 
  \[ [\bar{W}]_{\A_{p_0}}
    \leq c(d) [W_0]_{\A_{1}}^{1-p'/p_0'} [W]_{\A_{p}}^{p'/p_0'}. \]
\end{corollary}

\begin{proof}[Proof of
  Proposition~\ref{prop:general-reverse-factorization}]
  We define three norm functions:
  \begin{align*}
    \rho_0(x,v) &= |W_0(x)v|, \\
    \rho_1(x,v) &= |W_1(x)v|, \\
    \rho(x,v) &= |\bar{W}(x)v|, \qquad\text{where }\bar{W}= ((W_0)^2\#_t (W_1)^2)^{1/2}.
  \end{align*}
For fixed point $x$, define $p_t(v) = \rho_0(x,v)^{1-t} \rho_1(x,v)$, $v\in \R^d$. 
By Corollary \ref{dwgm} followed by Lemma \ref{lemma:double-dual-geo-mean}  we have
\begin{equation}\label{rh1}
 \rho(x,v) \approx p_t^{**}(v) \le p_t(v) = \rho_0(x,v)^{1-t}\rho_1(x,v)^{t}.
 \end{equation}
Since 
\[
A^{-1} \#_t B^{-1} = (A \#_t B)^{-1}, \qquad\text{for } A,B \in \Sd, \ 0<t<1,
\]
by Proposition~\ref{prop:dual-matrix-norm} we have a similar inequality for dual norms 
\begin{equation}
\label{rh2}
 \rho^*(x,v) \lesssim \rho_0^*(x,v)^{1-t} \rho_1^*(x,v)^{t}.
 \end{equation}

Fix a cube $Q$. Since $1= \frac{q(1-t)}{q_0}+\frac{qt}{q_1}$, by \eqref{rh1} and 
H\"older's inequality (if $q_0,q_1<\infty$),
\begin{multline} \label{eqn:gen-fac-A1}
  \langle \rho \rangle_{q,Q}(v)
  \lesssim \bigg(\avgint_Q
  \rho_0(x,v)^{q(1-t)}\rho_1(x,v)^{qt}\,dx\bigg)^{\frac{1}{q}} \\
  \leq \bigg(\avgint_Q \rho_0(x,v)^{q_0}\bigg)^{\frac{1-t}{q_0}}
  \bigg(\avgint_Q \rho_q(x,v)^{q_1}\bigg)^{\frac{t}{q_0}}
  = \langle \rho_0 \rangle_{q_0,Q}(v)^{1-t}
  \langle \rho_1 \rangle_{q_1,Q}(v)^t.
\end{multline}
A simple modification of this argument shows that this inequality
holds if $q_0$ or $q_1=\infty$.  Since we also have that $1=\frac{q'(1-t)}{q_0'}+\frac{q't}{q_1'}$, we
can repeat this argument using \eqref{rh2} to get that
\begin{equation} \label{eqn:gen-fac-B1}
  \langle \rho^* \rangle_{q',Q}(v) \lesssim
  \langle \rho_0^* \rangle_{q_0',Q}(v)^{1-t}
  \langle \rho_1^* \rangle_{q_1',Q}(v)^{t}.
\end{equation}
Since $W_0\in \A_{q_0}$ and $W_1\in \A_{q_1}$,  by
Definition~\ref{defn:Ap-defn} we
have that
\begin{equation} \label{eqn:gen-fac-C1}
  \langle \rho_0^* \rangle_{q_0',Q}(v)^{1-t}
  \langle \rho_1^* \rangle_{q_1',Q}(v)^t
  \leq [\rho_0]_{\A_{q_0}}^{1-t} [\rho_1]_{\A_{q_1}}^{t}
  \langle \rho_0 \rangle_{q_0,Q}^*(v)^{1-t}
  \langle \rho_1 \rangle_{q_1,Q}^*(v)^{t}.
\end{equation}

To simplify notation, we define several norms:
\begin{align*}
     \sigma_0 = \langle \rho_0^* \rangle_{q_0',Q}, \qquad
  & 
    \tau_0 = \langle \rho_0 \rangle_{q_0,Q}, \\
     \sigma_1 = \langle \rho_1^* \rangle_{q_1',Q}, \qquad
  & 
    \tau_1 = \langle \rho_1 \rangle_{q_1,Q}, \\
     \sigma = \langle \rho^* \rangle_{q',Q}, \qquad
  & 
    \tau = \langle \rho \rangle_{q,Q},
\end{align*}
and the two geometric means
\[ \sigma_t(v) = \sigma_0(v)^{1-t}\sigma_1(v)^{t}, \qquad
  \tau_t(v) = \tau_0(v)^{1-t}\tau_1(v)^{t}. \]
By inequality~\eqref{eqn:gen-fac-B1}, $\sigma(v)\lesssim
\sigma_0(v)^{1-t}\sigma_1(v)^{t}$.  By the definition of the dual norm
\eqref{eqn:dual-norm} it follows that $\sigma^*(v)\gtrsim
\sigma_t^*(v)$.  By 
Corollary~\ref{cor:dual-norm} and Lemma~\ref{lemma:geometric-mean},
these are both norms and if we 
dualize again, we get
\begin{equation} \label{eqn:gen-frac-D1}
\sigma(v) = \sigma^{**}(v) \lesssim \sigma_t^{**}(v).
\end{equation}

  Similarly, by inequality~\eqref{eqn:gen-fac-A1}, $\tau(v)\lesssim
  \tau_t(v)$, so we can repeat the above argument to get that
  $\tau(v)\lesssim \tau_t^{**}(v)$.  If we dualize yet again, we get
  \begin{equation} \label{eqn:gen-fac-E1}
    \tau_t^*(v)  = \tau_t^{***}(v) \lesssim \tau^*(v). 
  \end{equation}

  Finally, by inequality~\eqref{eqn:gen-fac-C1} we have that
  \[ \sigma_t(v)
    \leq [\rho_0]_{\A_{q_0}}^{1-t}
  [\rho_1 ]_{\A_{q_1}}^{t}
\tau_0^*(v)^{1-t} \tau_1^*(v)^{t}.  \]
If we dualize twice, and then apply the dual of
equivalence~\eqref{eqn:geometric-double-dual1} in
Proposition~\ref{prop:geometric-double-dual}, we get that 
\begin{multline} \label{eqn:gen-fac-F1}
  \sigma_t^{**}(v)
  \leq
  [\rho_0]_{\A_{q_0}}^{1-t}
  [\rho_1 ]_{\A_{q_1}}^{t} \big(\tau_0^*(v)^{1-t}
    \tau_1^*(v)^{t}\big)^{**} \\
   \approx   [\rho_0]_{\A_{q_0}}^{1-t}
    [\rho_1 ]_{\A_{q_1}}^{t}
    \tau_t^{***}(v) 
 =  [\rho_0]_{\A_{q_0}}^{1-t}
    [\rho_1 ]_{\A_{q_1}}^{t}
    \tau_t^{*}(v).
  \end{multline}

  If we now combine
  inequalities~\eqref{eqn:gen-frac-D1},~\eqref{eqn:gen-fac-E1},
  and~\eqref{eqn:gen-fac-F1}, we have
  \begin{multline*}
    \langle \rho^* \rangle_{q',Q}(v) =  \sigma(v)
    \lesssim \sigma_t^{**}(v) 
    \lesssim  [\rho_0]_{\A_{q_0}}^{1-t}
    [\rho_1 ]_{\A_{q_1}}^{t}
    \tau_t^{*}(v) \\
    \lesssim    [\rho_0]_{\A_{q_0}}^{1-t}
    [\rho_1 ]_{\A_{q_1}}^{t}
    \tau^*(v)
    =   [\rho_0]_{\A_{q_0}}^{1-t}
    [\rho_1 ]_{\A_{q_1}}^{t}
    \langle \rho \rangle_{q,Q}^*(v).
  \end{multline*}
Since the cube $Q$ is arbitrary, we get the desired result.
\end{proof}

\begin{remark}
Proposition \ref{prop:general-reverse-factorization} can be also shown using the complex interpolation method. Here we give only a brief sketch of the argument. Applying the exactness of the complex interpolation functor of exponent $t$ 
\cite[Theorem 4.1.2]{BL} to the identity operator on $\R^d$ equipped with norms appearing in Definition~\ref{defn:Ap-defn}, we deduce that the complex interpolation norms satisfy
\begin{equation}
\begin{aligned}
\label{coi}
[\langle \rho_0^* \rangle_{q_0',Q},
  \langle \rho_1^* \rangle_{q_1',Q}]_t
  & \le  
  [\rho_0]_{\A_{q_0}}^{1-t} [\rho_1]_{\A_{q_1}}^{t}
  [\langle \rho_0 \rangle_{q_0,Q}^*,
  \langle \rho_1 \rangle_{q_1,Q}^*]_t
  \\
  &= 
  [\rho_0]_{\A_{q_0}}^{1-t} [\rho_1]_{\A_{q_1}}^{t}
   [\langle \rho_0 \rangle_{q_0,Q},
  \langle \rho_1 \rangle_{q_1,Q}]_t^*.
  \end{aligned}
  \end{equation}
  The last identity is a consequence of the duality theorem \cite[Corollary 4.5.2]{BL}. Then, we use the fact that for any two norms $p_0$ and $p_1$ on $\R^d$, the complex interpolation norm satisfies
\[
[p_0,p_1]_t \approx (p_0^{1-t} p_1^t)^{**}.
\]
This can be shown using Corollary \ref{dwgm} and the complex interpolation of weighted $L^p$ spaces \cite[Theorem 5.5.3]{BL}. Hence, applying the double dual to \eqref{eqn:gen-fac-B1} followed by \eqref{coi} and then by the triple dual of \eqref{eqn:gen-fac-A1} yields
\[
\langle \rho^* \rangle_{q',Q} \lesssim
 [\langle \rho_0^* \rangle_{q_0',Q},
  \langle \rho_1^* \rangle_{q_1',Q}]_t \lesssim
  [\rho_0]_{\A_{q_0}}^{1-t} [\rho_1]_{\A_{q_1}}^{t}
   [\langle \rho_0 \rangle_{q_0,Q},
  \langle \rho_1 \rangle_{q_1,Q}]_t^*
  \lesssim  [\rho_0]_{\A_{q_0}}^{1-t} [\rho_1]_{\A_{q_1}}^{t} \langle \rho \rangle_{q,Q}^*.
  \]
 This shows that $\rho$ belongs to $\A_q$ with appropriate bound on $[\rho]_{\A_{q}}$.
  \end{remark}

\section{Extrapolation of matrix weights}
\label{section:extrapolation}

In this section we state and prove the Rubio de Francia extrapolation
theorem for
matrix $\A_p$ weights, originally formulated as
Theorem~\ref{thm:matrix-extrapolation-intro} in the Introduction.  As we noted there, we prove a version of sharp constant
extrapolation; this proof
requires multiple cases.  
A
 simpler proof, with only one case but which does not give the best
 possible  constant or include the endpoint result $p_0=\infty$, is possible,
 following the proof given in~\cite[Theorem~3.9]{MR2797562}.
 We leave the details to the interested reader.

To state our result, we introduce the convention of extrapolation
pairs.  This approach to extrapolation was developed in
\cite{MR2797562}.  Hereafter, $\F$ will denote a family of pairs $(f,g)$ of
measurable, vector-valued functions such that neither $f$ nor
$g$ is equal to $0$ almost everywhere.  If we write an
inequality of the form
\[ \|f\|_{L^p(\R^n,W)} \leq C\|g\|_{L^p(\R^n,W)}, \qquad (f,g) \in
    \F, \]
  we mean that this inequality holds for all pairs $(f,g)\in\F$ such
  that the lefthand side of this inequality is finite.  The constant,
  whether given explicitly or implicitly, is assumed to be independent
  of the pair $(f,g)$ and to depend only on $[W]_{\A_p}$ and not on
  the particular weight $W$.  We want to stress that
  $\|f\|_{L^p(\R^n,W)}<\infty$ is a crucial technical assumption in
  our proof, and to apply extrapolation an appropriate family $\F$
  must be constructed.  In the scalar case this can easily be done via
  a truncation argument and approximation:
  see~\cite[Section~6]{preprint-DCU}.  In the case of matrix weights a
  similar argument can be applied: see Section \ref{section:applications}.

  \begin{theorem} \label{thm:matrix-extrapolation}
Suppose that for some $p_0$, $1 \leq p_0\leq\infty$, there exists an
increasing function $K_{p_0}$ such that for every $W_0\in \A_{p_0}$,
\begin{equation} \label{eqn:matrix-extrapol1}
\|f\|_{L^{p_0}(\R^n, W_0)}
  \leq K_{p_0}([W_0]_{\A_{p_0}})\|g\|_{L^{p_0}(\R^n, W_0)},
  \qquad (f,g) \in
    \F. 
  \end{equation}
  Then for all $p$, $1<p<\infty$, and 
  for all $W\in \A_p$,
  \begin{equation} \label{eqn:matrix-extrapol2}
   \|f\|_{L^p(\R^n,W)}
  \leq K_{p}(p,p_0,n,d,[W]_{\A_{p}})\|g\|_{L^{p}(\R^n,W)},
  \qquad (f,g) \in
    \F, 
  \end{equation}
  where
  \[ K_p (p,p_0,n,d,[W]_{\A_{p}})
    = C(p,p_0)K_{p_0}\bigg(C(n,d,p,p_0)[W]_{\A_p}^{\max\big\{\frac{p}{p_0},\frac{p'}{p_0'}\big\}}\bigg). \]
\end{theorem}

   \medskip




\begin{proof}
 The proof has four cases and is modeled on the proof of sharp-constant
 extrapolation in~\cite[Theorem~3.22]{MR2797562}.

 Fix $1<p<\infty$ and $W\in \A_p$.  We begin by defining two iteration
 operators.  To define the first, let $P_W=N_WM$, where $M$ is the
 convex-set valued maximal operator and $N_W$ is from
 Definition~\ref{defn:Nw}.  By 
 Lemma~\ref{lemma:ellipsoid-max} and
 Theorem~\ref{thm:wtd-convex-max-bounds},
 \[ \|P_W\|_{L^p_\K(\R^n,W)}
   = \|M\|_{L^p_\K(\R^n,W)}
   \leq C(n,d,p) [W]_{\A_p}^{p'}. \]
Moreover, by Lemmas~\ref{lemma:Mf-sublinear}
and~\ref{lemma:ellipsoid-max}, $P_W$ is sublinear and monotone.
Therefore, by Theorem~\ref{thm:gen-rdf-algorithm} we can define
 \[ \Rdf_W H(x) =
   \sum_{k=0}^\infty 2^{-k}\|P_W\|_{L^p_\K(\R^n,W)}^{-k}
   P_W^kH(x), \]
and we have that
 \begin{enumerate}
 \item[(A)]  $H(x) \subset \Rdf_W H(x)$,
 \item[(B)] $\|\Rdf_W H\|_{L^p_\K(\R^n,W)}
   \leq 2 \|H\|_{L^p_\K(\R^n,W)}$,
 \item [(C)] $\Rdf_WH \in \A_1^\K$ and $M(\Rdf_W H)(x)
   \subset P_W(\Rdf_W H)(x)
   \subset 2 C(n,d,p) [W]_{\A_p}^{p'} \Rdf_W H(x)$;
 \end{enumerate}
 the first inclusion in (C) follows from
 Lemma~\ref{lemma:ellipsoid-max}.  Further, by the definition of
 $N_W$, we have that if $H=rW^{-1}\overline{\mathbf{B}}$, then $P_WH$ is also an
 ellipsoid-valued function with respect to $W^{-1}$.  Hence, by
 Lemma~\ref{lemma:ball-iteration}, $\Rdf_W H$ is also of this form.

 We now define the second iteration operator.  Since $W\in \A_p$,
 $W^{-1}\in \A_{p'}$, so by
 Theorem~\ref{thm:wtd-convex-max-bounds}, $M$  is bounded on
 $L^{p'}_\K(\R^n,W^{-1})$ and
 \[ \|M\|_{L^{p'}_\K(\R^n,W^{-1})}
   \leq C(n,d,p)\|W^{-1}\|_{\A_{p'}}^p
   = C(n,d,p)\|W\|_{\A_{p}}^p. \]
 Define $M'H(x) = W^{-1}(x)M(WH)(x)$.  Then
 \[ \|M'H\|_{L^{p'}_\K(\R^n,|\cdot|)}
   = \|M(WH)\|_{L^{p'}_\K(\R^n,W^{-1})}
   \leq C(n,d,p)\|W\|_{\A_{p}}^p\|H\|_{L^{p'}_\K(\R^n,|\cdot|)}. \]
 Now let $P_I' = N_IM'$;  again by Lemmas~\ref{lemma:Mf-sublinear}
and~\ref{lemma:ellipsoid-max}, $P_I'$ is sublinear and monotone, so by
Theorem~\ref{thm:gen-rdf-algorithm} we can define
 \[ \Rdf_I' H(x)
   = \sum_{k=0}^\infty 2^{-k}\|P_I'\|_{L^{p'}_\K(\R^n,|\cdot|)}^{-k}
   (P_I')^kH(x), \]
and we have that
 \begin{enumerate}
 \item[(A$'$)]  $H(x) \subset \Rdf_I' H(x)$,
 \item[(B$'$)] $\|\Rdf_I' H\|_{L^{p'}_\K(\R^n,|\cdot|)}
   \leq 2 \|H\|_{L^{p'}_\K(\R^n,|\cdot|)}$,
 \item [(C$'$)] $W\Rdf_I' H\in \A_1^\K$ and
   $M(W\Rdf_I' H)(x) 
   \subset 2 C(n,d,p) [W]_{\A_p}^{p} W\Rdf_I' H(x)$.
 \end{enumerate}
 To see why (C$'$) holds, note that by
 Theorem~\ref{thm:gen-rdf-algorithm} and
 Lemma~\ref{lemma:ellipsoid-max} we have that
 \[  W^{-1}(x) M(W\Rdf_I'H)(x)
   \subset N_I( W^{-1} M(W\Rdf_I'H))(x) \subset
   2\|P_I'\|_{L^{p'}_\K(\R^n,|\cdot|)} \Rdf_I' H(x). \]
Finally, by
Lemma~\ref{lemma:ball-iteration}, if $H$ is a ball-valued function,
then so is $\Rdf_I'H$.

To prove extrapolation we consider four cases,
depending on the relative sizes of $p$ and $p_0$.

\subsection*{Case I: $\mathbf{1<p<p_0<\infty}$}
Fix $(f,g)\in \F$.  To prove
 inequality~\eqref{eqn:matrix-extrapol2}, we may suppose, by our
 assumptions on the family $\F$, that
 $0<\|f\|_{L^p(\R^n,W)}<\infty$.   Similarly, we may assume that
   $0<\|g\|_{L^p(\R^n,W)}<\infty$; we may assume the second inequality
     since otherwise~\eqref{eqn:matrix-extrapol2} is trivially true.
     Define the functions
     \begin{gather*}
       F(x) = \conv\{ -f(x), f(x) \}, \qquad
        N_W F(x) = |W(x)F(x)|W^{-1}(x)\overline{\mathbf{B}}, \\
       G(x) = \conv\{ -g(x), g(x) \}, \qquad
        N_W G(x) = |W(x)G(x)|W^{-1}(x)\overline{\mathbf{B}}, 
     \end{gather*}
     where we have that
     \[ 
       |W(x)F(x)| = |W(x)f(x)|, \qquad
       |W(x)G(x)| = |W(x)g(x)|.
\]
     Now define the ellipsoid-valued function
     \[ \bar{H}(x) = \bigg(\frac{|W(x)f(x)|}{\|f\|_{L^p(\R^n,W)}}
       +\frac{|W(x)g(x)|}{\|g\|_{L^p(\R^n,W)}}\bigg)W^{-1}(x)\overline{\mathbf{B}}
       = \bar{h}(x) W^{-1}(x)\overline{\mathbf{B}}. \]
     Then we have that $\|\bar{H}\|_{L^p_\K(\R^n,W)} \leq 2$. 
     The function $\Rdf_W \bar{H}$ is
     also ellipsoid-valued with respect to $W^{-1}$. Hence, there exists a scalar function, which we denote by $\Rdf_W \bar{h}$, such that
     \[ \Rdf_W \bar{H}(x) = \Rdf_W \bar{h}(x) W^{-1}(x)\overline{\mathbf{B}}. \]
     By property (A), $\bar{H}(x) \subset \Rdf_W \bar{H}(x)$, which implies
     that $\bar{h}(x) \leq \Rdf_W \bar{h}(x)$.

     By H\"older's inequality with
     exponents $p_0/p$ and $(p_0/p)'=\frac{p_0}{p_0-p}$,
     \begin{align*}
       & \bigg(\int_{\R^n} |W(x)f(x)|^p\,dx\bigg)^{\frac{1}{p}} \\
       & \qquad \qquad = \bigg(\int_{\R^n} |\Rdf_W \bar{h}(x)^{-\frac{p_0-p}{p_0}}
         W(x)f(x)|^p \Rdf_W
         \bar{h}(x)^{p\frac{p_0-p}{p_0}}\,dx\bigg)^{\frac{1}{p}} \\
       & \qquad \qquad \leq
         \bigg(\int_{\R^n} |\Rdf_W \bar{h}(x)^{-\frac{p_0-p}{p_0}}
         W(x)f(x)|^{p_0}\,dx\bigg)^{\frac{1}{p_0}}
          \bigg(\int_{\R^n}
         \Rdf_W\bar{h}(x)^{p}\,dx\bigg)^{\frac{1}{p}\frac{p_0-p}{p_0}}
       \\
       & \qquad \qquad  = I_1^{\frac{1}{p_0}}I_2^{\frac{1}{p}\frac{p_0-p}{p_0}}. 
     \end{align*}

     We estimate $I_1$ and $I_2$ separately.  To estimate the latter:
     by the definition of $\Rdf_W\bar{h}$, by property (B), and by
     Lemma~\ref{lemma:ellipsoid-max}, 
     \[
       I_2  = \int_{\R^n} \Rdf_W\bar{h}(x)^p\,dx 
             = \int_{\R^n} |W(x) \Rdf_W \bar{H}(x)|^p\,dx 
            \leq 2^p \int_{\R^n} |W(x) \bar{H}(x)|^p\,dx 
                    \leq 4^p.
\]

      To estimate $I_1$ note first that by property (C),
      \[  (\Rdf_W \bar{h}) W^{-1}\overline{\mathbf{B}}
        = \Rdf_W \bar{H} \in \A_1^\K. \]
By Corollary~\ref{cor:matrix-norm-A1}, $(\Rdf_W \bar{h})
      W^{-1}\in \A_1$, and so $(\Rdf_W \bar{h})^{-1} W \in \A_\infty$.  
      Thus, by       Corollary~\ref{cor:Duo-sharp-factorization},
      \[ W_0 = (\Rdf_W \bar{h})^{-\frac{p_0-p}{p_0}}W
        = [(\Rdf_W \bar{h})^{-1} W]^{\frac{p_0-p}{p_0}}
        W^{\frac{p}{p_0} }\in \A_{p_0}, \]
      and
      \begin{equation}\label{wc} [W_0]_{\A_{p_0}}
        \leq C(n,d,p,p_0)[W]_{\A_p}^{\frac{p}{p_0}} [W]_{\A_p}^{p'
          \frac{p_0-p}{p_0}}=  C(n,d,p,p_0)[W]_{\A_p}^{\frac{p'}{p_0'}}. 
          \end{equation}

        Second, since $\Rdf_W \bar{h}(x)  \geq |\bar{h}(x)| \geq
        |W(x)f(x)|/\|f\|_{L^p(\R^n,W)}$, we have     
        \begin{align*}
          I_1 &= ||f||_{L^{p_0}(W_0,\R^d)}^{p_0}
          =\int_{\R^n} |\Rdf_W \bar{h}(x)^{-\frac{p_0-p}{p_0}}
         W(x)f(x) |^{p_0}\,dx \\
         & \leq  \int_{\R^n} |\Rdf_W \bar{h}(x)|^{-(p_0-p)}|W(x)F(x)|^{p_0}\,dx
         \\
         &\leq \|f\|_{L^p(\R^n,W)}^{p_0-p} \int_{\R^n} |W(x)F(x)|^p\,dx
         = \|f\|_{L^p(\R^n,W)}^{p_0} < \infty. 
       \end{align*}
       
       Likewise,  using the fact that
       $\Rdf_W \bar{h}(x)  \geq |\bar{h}(x)| \geq
        |W(x)G(x)|/\|g\|_{L^p(\R^n,W)}$, we have
        \[
        ||g||_{L^{p_0}(W_0,\R^d)}^{p_0}=\int_{\R^n} |\Rdf_W \bar{h}(x)^{-\frac{p_0-p}{p_0}}
         W(x)g(x) |^{p_0}\,dx \le \|g\|_{L^p(\R^n,W)}^{p_0} <\infty.
         \]
        
       Taken together, these estimates imply that we can apply our
       hypothesis~\eqref{eqn:matrix-extrapol1} to the pair $(f,g)$
       with the weight $W_0$.  Therefore, by \eqref{wc} we have
       \[
         I_1^{\frac{1}{p_0}}
         = \|f\|_{L^{p_0}(W_0,\R^d)}
         \leq K_{p_0}([W_0]_{A_{p_0}}) \|g\|_{L^{p_0}(W_0,\R^d)} 
         \leq K_{p_0}\bigg(C(n,d,p,p_0)[W]_{\A_p}^{\frac{p'}{p_0'}}\bigg)
           \|g\|_{L^p(\R^n,W)}.
\]
       Combining this inequality with the estimate for $I_2$ yields \eqref{eqn:matrix-extrapol2}.

\subsection*{Case II: $\mathbf{p_0=\infty}$}
As in the
previous case we have $W_0=(\Rdf_W \bar{h})^{-1} W \in \A_\infty$.  Moreover, for
almost every $x$ we have that 
\[
|W(x)F(x)|\Rdf_W \bar{h}(x)^{-1}\leq |W(x)F(x)| \bar{h}(x)^{-1} \le 
\|f\|_{L^p(\R^n,W)}.
\]
Thus, $\|f\|_{L^\infty(\R^n, W_0 )}\le \|f\|_{L^p(\R^n,W)}<\infty$.  The same argument also shows that
\[ \|g\|_{L^\infty(\R^n, W_0 )} \leq \|g\|_{L^p(\R^n,W)}<\infty. \]
Therefore, we can
apply~\eqref{eqn:matrix-extrapol1} to the pair $(f,g)\in \mathcal F$ and argue as in Case I to get
\begin{align*}
  \|f\|_{L^p(\R^n,W)}^p &=  \int_{\R^n} |W(x)f(x)|^p\,dx \\
  & =  \int_{\R^n} |\Rdf_W \bar{h}(x)^{-1} W(x)f(x)|^p \Rdf_W
\bar{h}(x)^{p}\,dx \\
& \leq \|f\|_{L^\infty(\R^n, W_0 )}^p \int_{\R^n} \Rdf_W
\bar{h}(x)^{p}\,dx 
\\ 
& \leq 4^p K_\infty( [W_0]_{\A_{p_0}})^p \|g\|_{L^\infty(\R^n, W_0)}^p \\
& \leq  4^p K_\infty(C(n,d,p) [W]_{\A_p}^{p'})^p\|g\|_{L^p(\R^n,W)}^p. 
\end{align*}

\subsection*{Case III: $\mathbf{1<p_0<p}$}
 Fix $(f,g)\in \F$.  To prove
 inequality~\eqref{eqn:matrix-extrapol2}, we may again assume that
 $0<\|f\|_{L^p(\R^n,W)},\, \|g\|_{L^p(\R^n,W)}<\infty$.  Since the
 dual of the scalar function space $L^p(\R^n)$ is $L^{p'}(\R^n)$,
 there exists $h\in L^{p'}(\R^n)$, $\|h\|_{L^{p'}(\R^n)}=1$, such that
 \[ \|f\|_{L^p(\R^n,W)} = \int_{\R^n} |W(x)f(x)|h(x)\,dx. \]
   Define the ball-valued function $H(x)=h(x)\overline{\mathbf{B}}$; since 
   $H\in L^{p'}_\K(\R^n,|\cdot|)$, $\Rdf_I'H$ is defined and is
   ball-valued function; set $\Rdf_I'H(x)=\Rdf_I'h(x)\overline{\mathbf{B}}$.  As
   before, by (A$'$), we have that $h(x)\leq \Rdf_I'h(x)$.  Therefore, by
   H\"older's inequality with exponent $p_0$, we have
   that 
   \begin{align*}
     \int_{\R^n} |W(x)f(x)|h(x)\,dx
     & \leq \int_{\R^n} |\Rdf_I'h(x)
       ^{1-p'/p_0'}W(x)f(x)|h(x)^{p'/p_0'}\,dx \\
     & \leq \bigg(\int_{\R^n} |\Rdf_I'h(x)
       ^{1-p'/p_0'}W(x)f(x)|^{p_0}\,dx\bigg)^{\frac{1}{p_0}}
       \bigg(\int_{\R^n} h(x)^{p'}\,dx\bigg)^{\frac{1}{p_0'}} \\
       & = \bigg(\int_{\R^n} |\Rdf_I'h(x)
       ^{1-p'/p_0'}W(x)f(x)|^{p_0}\,dx\bigg)^{\frac{1}{p_0}}.
   \end{align*}

   To complete the estimate, first note that by (C$'$),
   $(\Rdf_I'h)W\overline{\mathbf{B}} \in \A_1^\K$. Hence, by
   Corollary~\ref{cor:matrix-norm-A1},
   $(\Rdf_I'h)W \in \A_1$.  Therefore, by
   Corollary~\ref{cor:Duo-sharp-factorization}, $W_0=\Rdf_I'h(x)
   ^{1-p'/p_0'}W(x) \in \A_{p_0}$ and
   \[ [W_0]_{\A_{p_0}} \leq C(n,d,p,p_0)[W]_{\A_p}^{p(1-p'/p_0')}
     [W]_{\A_p}^{p'/p_0'} = C(n,d,p,p_0)[W]_{\A_p}^{p/p_0}. \]
Moreover, the above estimates yield
\[
\|f\|_{L^p(\R^n,W)} \le \|f\|_{L^{p_0}(\R^n,W_0)}.
\]
On the other hand, by H\"older's inequality with exponent $p/p_0$ and property (B$'$),
   \begin{multline*}
     \|f\|_{L^{p_0}(\R^n,W_0)}^{p_0}=
     \int_{\R^n} |\Rdf_I'h(x)
     ^{1-p'/p_0'}W(x)f(x)|^{p_0}\,dx \\
     \leq
     \bigg(\int_{\R^n} |W(x)f(x)|^p\,dx\bigg)^{p_0/p}
     \bigg(\int_{\R^n} \Rdf_I'h(x)^{p'}\,dx\bigg)^{1/(p/p_0)'}
       \leq 2^{1/(p/p_0)'}\|f\|_{L^p(\R^n,W)}^{p_0} < \infty. 
     \end{multline*}
     Likewise, we have
     \[
     \|g\|_{L^{p_0}(\R^n,W_0)}^{p_0}
       \leq 2^{1/(p/p_0)'}\|g\|_{L^p(\R^n,W)}^{p_0} < \infty. 
       \]

     Therefore, we can apply our hypothesis
     \eqref{eqn:matrix-extrapol1} to the pair $(f,g)$ with the weight
     $W_0$
     \begin{multline*}
     \|f\|_{L^p(\R^n,W)} \le \|f\|_{L^{p_0}(\R^n,W_0)}
 \leq K_{p_0}([W_0]_{A_{p_0}})
         \|g\|_{L^{p_0}(\R^n,W_0)} \\
  \leq C(p,p_0) K_{p_0}(C(n,d,p,p_0)[W]_{\A_p}^{p/p_0})
         \|g\|_{L^p(\R^n,W)}.
       \end{multline*}
       
     \subsection*{Case IV: $\mathbf{p_0=1}$}
We make the same assumptions and use the same notation as in the
previous case.   Then we have that $W_0=(\Rdf_I' h) W \in \A_1$, and the
above argument shows that 
\[
\|f\|_{L^1(\R^n, W_0)} \le 2^{1/p'}\|f\|_{L^p(\R^n,W)} < \infty.
\]
The same inequality holds for $g$.
Therefore, we can apply inequality~\eqref{eqn:matrix-extrapol1} to the
pair $(f,g)$ to get
\begin{multline*}
  \|f\|_{L^p(\R^n,W)}
   = \int_{\R^n} |W(x)f(x)|h(x)\,dx \leq \int_{\R^n} |\Rdf_I'h(x)W(x)f(x)|\,dx  
  = \|f\|_{L^1(\R^n, W_0)} 
  \\
   \leq K_1([W_0]_{\A_p}) \|g\|_{L^1(\R^n, W_0)} 
  \leq 2 K_1(C(n,d,p)[W]_{\A_p}^p) \|g\|_{L^p(\R^n,W)}.
  \qedhere
\end{multline*}
\end{proof}

   \medskip

   \section{An application of Rubio de Francia extrapolation}
   \label{section:applications}

   In this section we illustrate
   Theorem~\ref{thm:matrix-extrapolation} by deducing Theorem
   \ref{thm:matrix-extrapolation-intro} and by proving quantitative
   $L^p$ bounds for maximal rough singular integral operators,
   extending the results from~\cite{MR4245601}. Indeed, Theorem
   \ref{thm:matrix-extrapolation-intro} is a simple consequence of
   Theorem~\ref{thm:matrix-extrapolation}.

\begin{proof}[Proof of Theorem \ref{thm:matrix-extrapolation-intro}]
Let $T$ be a scalar-valued operator. We assume that the extension of $T$
to vector-valued functions, which is given for ${f}=(f_1,\ldots,f_d)^t$ by
applying it to each coordinate $T{f}=(Tf_1,\ldots,Tf_d)^t$, 
fulfills the hypothesis of Theorem \ref{thm:matrix-extrapolation-intro}.
That is, for some $p_0$, $1 \leq p_0\leq\infty$, there exists an
increasing function $K_{p_0}$ such that for every $W_0\in \A_{p_0}$ we have \eqref{ext1}. 
If we take any scalar weight $w_0\in \A_{p_0}$ we can define $W_0$ to
be the diagonal matrix with copies of $w_0$ on the diagonal, so $T$ is
bounded on $L^{p_0}(\R^n,w_0)$.  Hence, the scalar-valued extrapolation theorem implies that
 $T$ is bounded on the scalar weighted spaces $L^p(\R^n, w)$, for any $w\in \A_p$, $1<p<\infty$.
   To apply Theorem~\ref{thm:matrix-extrapolation}, we need to
   construct a family $\F$ of pairs $(Tf,f)$, such that 
   given any $p$, $1<p<\infty$, and $W\in \A_p$, then for any pair
   $(Tf,f)\in \F$, we have $\|Tf\|_{L^p(\R^n, W)} < \infty$.   
   
   Given a function $f \in
   L^\infty_c(\R^n,\R^d)$, we have that $f\in L^p(\R^n,w)$ for any scalar
   weight $w\in \A_p$.  Moreover, by
   Corollary~\ref{cor:eigenvalues-scalar-Ap}, if matrix $W\in
   \A_p$, then $|W|_{\op}$ is a scalar $\A_p$ weight. 
   Hence, since
   $T$ is bounded on the scalar weighted spaces,
\[ \int_{\R^n} |W(x)Tf(x)|^p\,dx
\leq
\int_{\R^n} \big( |W(x)|_{\op} |Tf(x)|\big)^p\,dx 
\leq
C\int_{\R^n} \big( |W(x)|_{\op} |f(x)|\big)^p\,dx 
< \infty. \]
Therefore, if we form the family of extrapolation pairs
\[ \F = \{ (Tf,f) : f\in  L^\infty_c(\R^n,\R^d) \}, \]
then for each $1<p<\infty$ we can apply the conclusion of
Theorem~\ref{thm:matrix-extrapolation} to the family $\F$.
This gives us the desired inequality \eqref{ext2} for all $f\in
L^\infty_c(\R^n,\R^d)$.  

Now suppose that $T$ is linear. Take any $f\in L^p(\R^n,W)$.
Let $\{f_j\}_{j=1}^\infty$ be a sequence in $L^\infty_c(\R^n,\R^d)$, which converges to $f$ in $L^p(\R^n,W)$ norm. Since $T$ is linear, we have $Tf_j - Tf_k = T(f_j - f_k)$, and by \eqref{ext2}, $\{Tf_j\}_{j=1}^\infty$ is a Cauchy sequence in $L^p(\R^n,W)$. Since the space $L^p(\R^n,W)$ is complete, this sequence converges. Define $Tf$ to be the limit. Thus, 
we get the inequality \eqref{ext2} for  all $f\in L^p(\R^n,W)$. 
\end{proof}

   Take $\Omega \in
   L^\infty(S^{n-1})$ with $||\Omega||_\infty \le 1$ and vanishing
   integral on the unit sphere $S^{n-1}$ in $\R^n$.
   For each $0<\delta<1$, define the truncated rough singular integral by
   \[ T_{\Omega,\delta} f(x) = \int_{\delta<|x-y|<\delta^{-1}} \frac{\Omega((x-y)/|x-y|)}{|x-y|^n}
     f(y)\,dy. \]
   Define the maximal rough singular integral operator
   \[ T_\Omega^{\natural} f(x) = \sup_{0<\delta<1}
     |T_{\Omega,\delta}f(x)|. \]
   In~\cite{MR4245601}, di Plinio, Hyt\"onen, and Li proved that if
   $W\in \A_2$, then
   \begin{equation} \label{eqn:rough-sio-base}
     \| \sup_{0<\delta<1} |W T_{\Omega,\delta}f|\|_{L^2(\R^n)}
       \leq
       C[W]_{\A_2}^5 \|f\|_{L^2(\R^n,W)}. 
     \end{equation}

Our extrapolation result, Theorem \ref{thm:matrix-extrapolation}, yields the following extension to $L^p$ spaces.

   \begin{theorem} \label{thm:rough-sio}
Given $\Omega \in L^\infty(S^{n-1})$ with  $||\Omega||_\infty \le 1$ and $\int_{S^{n-1}} \Omega(x)\, dx=0$, for all $p$, $1<p<\infty$, and $W\in
\A_p$, we have
 \begin{equation} \label{eqn:rough-sio-p}
     \| \sup_{0<\delta<1} |W T_{\Omega,\delta}f|\|_{L^p(\R^n)}
       \leq
       C[W]_{\A_p}^{5\max\{\frac{p}{2},\frac{p'}{2}\}} \|f \|_{L^p(\R^n,W)}. 
     \end{equation}
   \end{theorem}
   
   \begin{remark}
     If we restate the constant in terms of \eqref{eqn:roudenko}, the traditional definition
     of matrix $A_p$, it becomes
     $[W]_{A_p}^{\frac{5}{2}\max\{1,\frac{1}{p-1}\}}$.  This is larger
     than the constant gotten in~\cite{MR4357365} for rough singular integrals.
   \end{remark}

   \begin{proof}
Fix $p$, $1<p<\infty$.  Following the proof of Theorem \ref{thm:matrix-extrapolation-intro}, we construct an appropriate family $\F$ of
extrapolation pairs.  Fix $f\in L^\infty_c(\R^n,\R^d)$; then $f\in
L^p(W)$ for all $p$ and $W\in \A_p$.   Since $T_\Omega^{\natural} $
is bounded on $L^p(\R^n,\R^d)$, we have that for almost every $x\in
\R^n$
\[ \sup_{0<\delta<1} |T_{\Omega,\delta} f(x)| < \infty.  \]
Since matrix multiplication takes a bounded set to a bounded set, we
have that
\[ \sup_{0<\delta<1} |W(x)T_{\Omega,\delta} f(x)| < \infty  \]
almost everywhere.  Therefore, for all such $x$ we can find
$\delta_x>0$ such that
\[ \sup_{0<\delta<1} |W(x)T_{\Omega,\delta} f(x)|
  \leq
  2 |W(x)T_{\Omega,\delta_x} f(x)|. \]
Let $g(x) = T_{\Omega,\delta_x} f(x)$; it is straightforward to show
that we can choose $\delta_x$ measurably, so $g$ is a measurable function.  Since $T_\Omega^{\natural}$ is
bounded on $L^p(\R^n,w)$ for all scalar $w\in \A_p$~\cite{DuRu}, by Corollary~\ref{cor:eigenvalues-scalar-Ap} we have that
\[  \int_{\R^n} |W(x)g(x)|^p\,dx
    \leq
    \int_{\R^n} \big(  |W(x)|_{op} |T_\Omega^{\natural} f(x)|\big)^p\,dx
    \leq
    C  \int_{\R^n} \big( |W(x)|_{op} f(x)|\big)^p\,dx < \infty. \]

  By inequality~\eqref{eqn:rough-sio-base}, for any $V\in A_2$,
  \[\bigg( \int_{\R^n} |V(x)g(x)|^2\,dx\bigg)^{\frac12}
    \leq
     \bigg(\int_{\R^n} \big(\sup_{0<\delta<1}
     |V(x)T_{\Omega,\delta}f(x)|\big)^2\,dx\bigg)^{\frac12}
     \leq C[V]_{\A_2}^5 \|f \|_{L^2(\R^n,V)}. \]
   Therefore, if we let $\F= \{(g,f) : f \in
   L^\infty_c(\R^n,\R^d), g(x) = T_{\Omega,\delta_x} f(x) \}$, we have that this family satisfies the hypotheses of
 Theorem~\ref{thm:matrix-extrapolation}. Hence, it follows
 that~\eqref{eqn:rough-sio-p} holds for all $f\in
 L^\infty_c(\R^n,\R^d)$.

 Now fix $p$, $1<p<\infty$, and $W\in \A_p$.  Since the operator
 $T_{\Omega,\delta}$ is linear, the operator $T$, which is defined initially for $f\in
 L^\infty_c(\R^n,\R^d)$ by
 \[ Tf(x)= \sup_{0<\delta<1} |W(x)T_{\Omega,\delta} f(x)| \]
 is sublinear. Arguing as we did above in the proof of
 Theorem~\ref{thm:matrix-extrapolation-intro}, we can show that $T$ has a
 continuous extension to all $f\in L^p(\R^n,W)$. Indeed, take any $f\in L^p(\R^n,W)$.
Let $\{f_j\}_{j=1}^\infty$ be a sequence in $L^\infty_c(\R^n,\R^d)$, which converges to $f$ in $L^p(\R^n,W)$ norm. Since $T$ is sublinear, we have $|Tf_j - Tf_k| \le |T(f_j - f_k)|$, and by \eqref{eqn:rough-sio-p}, $\{Tf_j\}_{j=1}^\infty$ is a Cauchy sequence in $L^p(\R^n)$, so this sequence converges. Define $Tf$ to be the limit. Clearly, $T: L^p(\R^n,W) \to L^p(\R^n)$ is bounded.

 We will now show that for all $f\in L^p(\R^n,W)$,
 \begin{equation}\label{eqn:orig-defn}
  \sup_{0<\delta<1} |W(x)T_{\Omega,\delta} f(x)| \le  Tf(x) 
\end{equation}
 almost everywhere.
 Fix such an $f$ and for each $k\in \N$,
 define $f_k \in L^\infty_c(\R^n,\R^d)$ by
 \[ f_k(x) = f(x) \min\bigg(1,
   \frac{k}{|f(x)|}\bigg)\chi_{B(0,k)}(x).  \]
 Then we have that $|f_k(x)|\leq |f(x)|$ and $f_k \to f$ pointwise
 a.e.  Moreover, since $f$ and $f_k$ are parallel vectors,
 $|W(x)f_k(x)|\leq |W(x)f(x)|$.  Hence, by the dominated convergence
 theorem, $f_k \to f$ in $L^p(\R^n,W)$. Therefore, we have that $
 \sup_{0<\delta<1} |W(x)T_{\Omega,\delta} f_k|$ converges to $Tf$ 
 in $L^p(\R^n)$.  By passing to a subsequence, we then have that for
 almost every $x\in \R^n$, 
 \begin{equation} \label{eqn:first-limit}
 Tf(x) = \lim_{k\to \infty} \sup_{0<\delta<1}
   |W(x)T_{\Omega,\delta} f_k(x)|. 
 \end{equation}

 Now fix $\delta$, $0<\delta<1,$ and let
 $B=B(0,N)$, where $N>\delta^{-1}$.  For brevity, in the definition of
 $T_{\Omega,\delta}$  we write $[x-y]'= (x-y)|x-y|^{-1}$. Then we have that
 \begin{align*}
   \int_B |W(x)T_{\Omega,\delta}f(x)|^p\,dx
   & = \int_B \bigg| W(x) \int_{\delta<|x-y|<\delta^{-1}}
     \frac{ \Omega([x-y]')}{|x-y|^n} f(y)\,dy\bigg|^p \,dx \\
   & \leq \delta^{-np} \int_B
     \bigg(\int_{|x-y|<\delta^{-1}} |W(x)W^{-1}(y)W(y)f(y)|\,dy\bigg)^p
     \,dx \\
   & \leq  \delta^{-np} \int_{2B}
     \bigg( \int_{2B} |W(x)W^{-1}(y)|^{p'}_{\op}\,dy
     \bigg)^{\frac{p}{p'}}\,dx
     \bigg( \int_{\R^n} |W(y)f(y)|^p\,dy \bigg) \\
   & \leq C[W]_{\A_p}^p |B|^p \|f\|^p_{L^p(\R^n, W)};
 \end{align*}
 the last inequality holds by applying the $\A_p$ condition with balls
 instead of cubes. Since $N$ can be made arbitrarily large, we get
 that 
 for all $\delta$,  $WT_{\Omega,\delta}f \in L^p_{loc}(\R^n)$,
 and so $|WT_{\Omega,\delta}f(x)|<\infty$ a.e.

 Moreover, again since $f$ and $f_k$ are parallel vectors, for a.e. $x,\, y\in \R^n$
 \[ |W(x) \Omega([x-y]')|x-y|^{-n} f_k(y)| \leq
   |W(x) \Omega([x-y]')|x-y|^{-n} f(y)|. \]
 Therefore, by the dominated convergence theorem, we have that
for every $\delta$ and for almost every
 $x$,
 \begin{equation} \label{eqn:limit-two}
   W(x)T_{\Omega, \delta}f(x) = \lim_{k\to\infty}
   W(x)T_{\Omega, \delta}f_k(x). 
 \end{equation}
 Fix an $x$ such that both \eqref{eqn:first-limit} and \eqref{eqn:limit-two} hold. Consequently,
 \begin{align*}
   \sup_{0<\delta<1} |W(x)T_{\Omega, \delta}f(x)|\big| & 
   =  \sup_{0<\delta<1} \lim_{k\to\infty} |W(x)T_{\Omega, \delta}f_k(x)| 
\\
&  \le \lim_{k\to\infty} \sup_{0<\delta<1}  |W(x)T_{\Omega, \delta}f_k(x)|  = Tf(x). \end{align*}
 This proves \eqref{eqn:orig-defn} and the boundedness of $T: L^p(\R^n,W) \to L^p(\R^n)$ yields \eqref{eqn:rough-sio-p}.
   \end{proof}
   
\bibliographystyle{plain}
\bibliography{convex-maximal}

\end{document}